\pdfoutput=1
\RequirePackage{ifpdf}
\ifpdf 
\documentclass[pdftex]{sigma}
\else
\documentclass{sigma}
\fi

\usepackage{mathrsfs}
\usepackage{dsfont}
\usepackage{tikz}
\usetikzlibrary{matrix}
\usetikzlibrary{decorations.markings}
\usepackage[all]{xy}

\numberwithin{equation}{section}

\newtheorem{thm}{Theorem}[section]
\newtheorem{cor}[thm]{Corollary}
\newtheorem{lem}[thm]{Lemma}
\newtheorem{prop}[thm]{Proposition}
\newtheorem{conj}[thm]{Conjecture}
{ \theoremstyle{definition}
\newtheorem{defn}[thm]{Definition}
\newtheorem{exmp}[thm]{Example}
\newtheorem{conv}[thm]{Convention}
\newtheorem{rmk}[thm]{Remark} }

\newcommand{\conf}{\operatorname{Conf}}

\newcommand{\Gr}{\operatorname{Gr}}
\newcommand{\GL}{\operatorname{GL}}
\newcommand{\PGL}{\operatorname{PGL}}
\newcommand{\SL}{\operatorname{SL}}
\newcommand{\mat}{\operatorname{Mat}}

\newcommand{\Hom}{\operatorname{Hom}}

\newcommand{\diag}{\operatorname{diag}}

\newcommand{\dGr}{\mathscr{G}{\rm r}}
\newcommand{\dconf}{\mathscr{C}\!{\rm onf}}
\newcommand{\uf}{\operatorname{uf}}
\newcommand{\Tr}{\operatorname{Tr}}
\newcommand{\up}{\mathbf{up}}
\newcommand{\wt}{\operatorname{wt}}
\newcommand{\spec}{\operatorname{Spec}}
\newcommand{\Fr}{\operatorname{Fr}}
\newcommand{\codim}{\operatorname{codim}}

\newcommand{\inprod}[2]{\left\langle#1,#2\right\rangle}

\begin{document}
\allowdisplaybreaks

\newcommand{\arXivNumber}{1803.06901}

\renewcommand{\thefootnote}{}

\renewcommand{\PaperNumber}{067}

\FirstPageHeading

\ShortArticleName{Cyclic Sieving and Cluster Duality of Grassmannian}

\ArticleName{Cyclic Sieving and Cluster Duality of Grassmannian\footnote{This paper is a~contribution to the Special Issue on Cluster Algebras. The full collection is available at \href{https://www.emis.de/journals/SIGMA/cluster-algebras.html}{https://www.emis.de/journals/SIGMA/cluster-algebras.html}}}

\Author{Linhui SHEN and Daping WENG}
\AuthorNameForHeading{L.~Shen and D.~Weng}
\Address{Department of Mathematics, Michigan State University,\\ 619 Red Cedar Road, East Lansing, MI 48824, USA}
\Email{\href{mailto:linhui@math.msu.edu}{linhui@math.msu.edu}, \href{mailto:wengdap1@msu.edu}{wengdap1@msu.edu}}
\URLaddress{\url{https://sites.google.com/site/linhuishenmath/},\newline
\hspace*{10.5mm}\url{https://users.math.msu.edu/users/wengdap1/}}

\ArticleDates{Received January 07, 2020, in final form July 14, 2020; Published online July 25, 2020}

\Abstract{We introduce a decorated configuration space $\mathscr{C}\!{\rm onf}_n^\times(a)$ with a potential function~$\mathcal{W}$. We prove the cluster duality conjecture of Fock--Goncharov for Grassmannians, that is, the tropicalization of $\big(\mathscr{C}\!{\rm onf}_n^\times(a), \mathcal{W}\big)$ canonically parametrizes a linear basis of the homogeneous coordinate ring of the Grassmannian $\operatorname{Gr}_a(n)$ with respect to the Pl\"{u}cker embedding. We prove that $\big(\mathscr{C}\!{\rm onf}_n^\times(a), \mathcal{W}\big)$ is equivalent to the mirror Landau--Ginzburg model of the Grassmannian considered by Eguchi--Hori--Xiong, Marsh--Rietsch and Rietsch--Williams. As an application, we show a cyclic sieving phenomenon involving plane partitions under a~sequence of piecewise-linear toggles.}


\Classification{05E10; 13F60; 14J33; 14M15; 14N35; 14T05}

\renewcommand{\thefootnote}{\arabic{footnote}}
\setcounter{footnote}{0}

\section{Introduction}

Throughout we let $a$, $b$ be positive integers and let $c$ be a non-negative integer. We set $n:=a+b$.

\subsection{Cluster duality of Grassmannians}
{\it Cluster algebras} are a class of commutative algebras introduced by Fomin and Zelevinsky \cite{FZI}.
Their geometric counterparts form a family of log Calabi--Yau varieties called cluster varieties.

A {\it cluster ensemble}\footnote{See Appendix~\ref{appendixA} for a brief review on cluster ensemble.} is a pair $(\mathscr{A}, \mathscr{X})$ of cluster varieties associated to an equivalence class of skew-symmetrizable matrices introduced by Fock and Goncharov~\cite{FGensemble}. The variety $\mathscr{A}$ is equipped with an exceptional class $\{\alpha\}$ of coordinate charts called $K_2$ clusters. A~rational function of~$\mathscr{A}$ is called a {\it universal Laurent polynomial} if it can be expressed as a Laurent polynomial in every~$\alpha$. The ring ${\bf up}(\mathscr{A})$ of universal Laurent polynomials of $\mathscr{A}$ coincides with the {\it upper cluster algebra} of~\cite{BFZ}. The variety $\mathscr{X}$ is equipped with an exceptional class $\{\chi\}$ of coordinate charts called Poisson clusters. Let ${\bf up}(\mathscr{X})$ be the ring of universal Laurent polynomials in every~$\chi$. The {\it cluster modular group} $\mathcal{G}$ is a discrete group acting on $\mathscr{A}$ and $\mathscr{X}$ that respects the cluster structures.

 The Fock--Goncharov cluster duality conjecture \cite{FGensemble} asserts that the ring ${\bf up}(\mathscr{A})$ admits a~natural basis $\mathcal{G}$-equivariantly paramatrized by the $\mathbb{Z}$-tropical points of $\mathscr{X}$, and vice versa. Cluster duality can be viewed as a manifestation of mirror symmetry between $\mathscr{A}$ and $\mathscr{X}$ \cite{GHK}. For example, the cluster duality for moduli spaces of local systems has been investigated in \cite{GS1, GS2}.

The present paper focuses on the cluster duality for Grassmannians.

In details, we introduce a pair of spaces, i.e., the {\it decorated Grassmannian} $\dGr_a(n)$ and the {\it decorated configuration space} $\dconf_n(a)$, both of which are variants of the Grassmanian $\Gr_a(n)$.

The decorated Grassmannian $\dGr_a(n)$ is essentially an affine cone over $\Gr_a(n)$ (see Section~\ref{sec 2.1} for its definition). In particular, the coordinate ring of $\dGr_a(n)$ coincides with the homogeneous coordinate ring of $\Gr_a(n)$ in its Pl\"ucker embedding. The decorated Grassmannian $\dGr_a(n)$ admits a particular divisor, whose complement is an affine variety denoted by $\dGr_a^\times(n)$. A~result of Scott~\cite{Sco} implied that the coordinate ring $\mathcal{O}\big(\dGr_a^\times(n)\big)$ coincides with an upper cluster algebra~${\bf up}(\mathscr{A})$. In this sense, $\dGr_a^\times(n)$ is naturally equipped with a cluster $K_2$ structure.

The decorated configuration space $\dconf_n(a)$ parametrizes $\PGL(V)$-orbits of $n$-many lines in an $a$-dimensional vector space $V$ together with a linear isomorphism between every pair of cyclic neighboring lines. After imposing a consecutive general position condition, we obtain a smooth subvariety $\dconf_n^\times(a)$. We prove that

\begin{thm}[Theorem \ref{2018.3.2.15.06}] The variety $\dconf_n^\times(a)$ is an affine variety and is equipped with a cluster Poisson structure. Its coordinate ring $\mathcal{O}\big(\dconf_n^\times(a)\big)$ coincides with the algebra ${\bf up}(\mathscr{X})$ of universal Laurent polynomials and therefore $\dconf_n^\times(a)\cong \mathrm{Spec}\left(\up \left(\mathscr{X}\right)\right)$.
\end{thm}

The cluster Poisson structure on $\dconf_n^\times(a)$ is related to the Poisson structure on the space of $a\times n$ matrices studied by Gekhtman and Yakimov~\cite{GY}. We also notice similarities between our proof of the affine-ness of $\dconf_n^\times(a)$ and techniques used by Morier-Genoud, Ovsienko, and Tabachnikov~\cite{MOT} in their work on the relation between Grassmannian $\Gr_{3}(n)$ and the moduli space of 2-frieze patterns. Further studies should be done on the connection between frieze patterns and Grassmannians in general.

Combining Theorem \ref{2018.3.2.15.06}, results from \cite{GS2} and \cite{Weng}, and the work of Gross, Hacking, Keel, and Kontsevich \cite{GHKK}, we prove the following theorem in Section~\ref{donj.con.10.32}.

\begin{thm}\label{main.3.1.152} The pair $\big(\dGr_a^\times(n), \dconf_n^\times(a)\big)$ admits a natural cluster ensemble structure. The duality conjecture of Fock--Goncharov holds in this case, that is, the coordinate ring $\mathcal{O}\big(\dGr_a^\times(n)\big)$ admits a cluster modular group equivariant basis parametrized by the $\mathbb{Z}$-tropical set of $\dconf_n^\times(a)$, and vice versa.
\end{thm}

In Section \ref{prelim}, we introduce several natural functions and maps on $\dGr_a^\times(n)$ and $\dconf_n^\times(a)$, which are summarized as follows
\begin{gather*}
\left\{\begin{matrix}\text{decorated Grassmannian $\dGr_a^\times(n)$},\\
\text{free rescaling $\mathbb{G}_m$ action},\\
\text{boundary divisor $D=\bigcup_i D_i$},\\
\text{action by a maximal torus $T\subset \GL_n$},\\
\text{twisted cyclic rotation $C_a$}
\end{matrix}
\right\}
\!\leftrightarrow\!
\left\{
\begin{matrix} \text{decorated configuration space $\dconf_n^\times(a)$},\\
\text{twisted monodromy $P$},\\
\text{potential function $\mathcal{W}=\sum_i \vartheta_i$},\\
\text{weight map $M\colon\dconf_n^\times(a)\rightarrow T^\vee$},\\
\text{cyclic rotation $R$}
\end{matrix}\right\}.
\end{gather*}

We investigate the natural cluster correspondence between the ingredients in the above dictionary. In particular, the potential $\mathcal{W}$ exhibits an explicit cyclic symmetry. It is essentially equivalent to the potential of the Grassmannian considered in \cite{EHX,MR,RW}. We identify $\mathcal{W}$ with the sum of theta functions associated to frozen vertices under the framework of \cite{GHKK}. As a~consequence, we prove the following.

\begin{thm}[Corollaries \ref{basis} and~\ref{weight decomposition}] \label{main2} Within the basis of $\mathcal{O}\big(\dGr_a^\times(n)\big)$ stated in Theorem~{\rm \ref{main.3.1.152}}, the subset $\big\{\theta_q\, | \, \mathcal{W}^t(q)\geq 0\big\}$ is a basis of $\mathcal{O} (\dGr_a(n) )$. Moreover, this basis of $\mathcal{O}\left(\dGr_a(n)\right)$ is compatible with both the decomposition $\mathcal{O} (\dGr_a(n) )\cong \bigoplus_{c\geq 0} V_{c\omega_a}$ into irreducible $\GL_n$-representations as well as the weight space decomposition with respect to a~maximal torus $T\subset \GL_a$.
\end{thm}

Marsh and Rietsch constructed a $B$-model for the Grassmannian in \cite{MR}, which is of the form $\big(\Gr_a^\times(n)\times \mathbb{G}_m,\mathcal{W}_q\big)$. Rietsch and Williams proved in \cite{RW} that this $B$-model is an example of a cluster dual space of the Grassmannian. In contrast to their approach, our approach is more geometric and is purely motivated by the associated cluster structures. We include a section in the appendix describing the connection between our version of cluster duality and the version considered by Rietsch and Williams.

\subsection{Cyclic sieving phenomenon of plane partitions} \label{intro csp}
Let $S$ be a finite set and let $g$ be a permutation of~$S$ of order~$n$. We are interested in the size of
the fixed point set $S^{g^d}$ of $g^d$ for $d\geq 0$. Let $F(q)$ be a polynomial of positive integral coefficients and let $\zeta= {\rm e}^{2\pi \sqrt{-1}/n}$. Following Reiner--Stanton--White~\cite{RSW}, we make the following definition.
\begin{defn} We say that the triple $(S, g, F(q))$ exhibits the \emph{cyclic sieving phenomenon $($CSP$)$} if the fixed point set cardinality $\# S^{g^d}$ is equal to the polynomial evaluation~$F\big(\zeta^d\big)$ for all~$d\geq 0$.
\end{defn}

Many combinatorial models have been found to exhibit the cyclic sieving phenomenon. Interestingly, the proofs often involve deep results in representation theory. For example, Rhoades~\cite{Rh} proved the CSP for rectangular Young tableaux under the action of promotion, by using Kazhdan--Lusztig theory and a representation of the Hecke algebra.
In~\cite{FKam}, Fontaine and Kamnizter studied the CSP for minuscule Littelmann paths under rotation using the geometric Satake correspondence and intersection homology of quiver varieties.
For more examples, we refer the interested reader to a survey on this topic by Sagan~\cite{Sa}.

As an application of our result on cluster duality for Grassmannian, we prove the CSP for plane partitions under a sequence of piecewise-linear toggles.

\begin{defn}\label{P(a,b,c)} A size $a\times b$ plane partition is an $a\times b$ matrix $\pi=\left(\pi_{i,j}\right)$ of non-negative integer entries $\pi_{i,j}$ that is weakly decreasing in rows and columns, and for which we define
\[|\pi|:=\sum_{i,j} \pi_{i,j}.\]
Denote by $P(a,b,c)$ the set of size $a\times b$ plane partitions with largest entry $\pi_{1,1}\leq c$.
\end{defn}

Let $[m]_q$ denote the quantum integer $\frac{1-q^m}{1-q}.$
MacMahon's formula asserts that
\[
M_{a,b,c}(q):= \sum_{\pi \in P(a,b,c)}q^{|\pi|}=\prod_{i=1}^a\prod_{j=1}^b\prod_{k=1}^c \frac{\left[i+j+k-1\right]_q}{\left[i+j+k-2\right]_q}.
\]

Following \cite[Section~4]{Roby}, we consider {\it piecewise-linear toggles} on $P(a,b,c)$.
Let $\pi \in P(a,b,c)$. Let $1\leq i\leq a$ and let $1\leq j\leq b$. We make the convention that $\pi_{0,j}=\pi_{i,0}=c$ and $\pi_{a+1,j}=\pi_{i,b+1}=0$. The piecewise-linear {toggle} $\tau_{i,j}$ at the $(i,j)$-th entry is an involution that sends $\pi$ to a new plane partition $\tau_{i,j}\pi$ such that
\begin{equation}\label{toggle}
 (\tau_{i,j}\pi )_{k,l}:= \begin{cases} \pi_{k,l} & \text{if $(k,l)\neq (i,j)$},\\
\max\left\{\pi_{i,j+1},\pi_{i+1,j}\right\}+\min\left\{\pi_{i-1,j},\pi_{i,j-1}\right\}-\pi_{i,j} & \text{if $(k,l)=(i,j)$}.
\end{cases}
\end{equation}
A birational version of $\tau_{i,j}$ has been constructed by Einstein and Propp in \cite{EP} and by Musiker and Roby in \cite{MuRo}. See also \cite[Section~9.3]{GS2} for the same construction of the birational Sch\"utzenberger involution of Gelfand--Tsetlin patterns.

Let $\eta$ be the sequence of piecewise-linear toggles that hits each entry exactly once in the order from bottom to top and from left to right, that is,
\[
\eta=\nu_b\circ \nu_{b-1}\circ \cdots \circ \nu_1, \qquad\mbox{where} \quad \nu_j=\tau_{1,j}\circ \tau_{2,j}\circ \cdots\circ \tau_{a,j}.
\]
For example, applying $\eta$ to $\left(\begin{smallmatrix}3 & 2 & 2 \\ 3 & 1 & 0
\end{smallmatrix}\right) \in P(2,3,6)$
yields
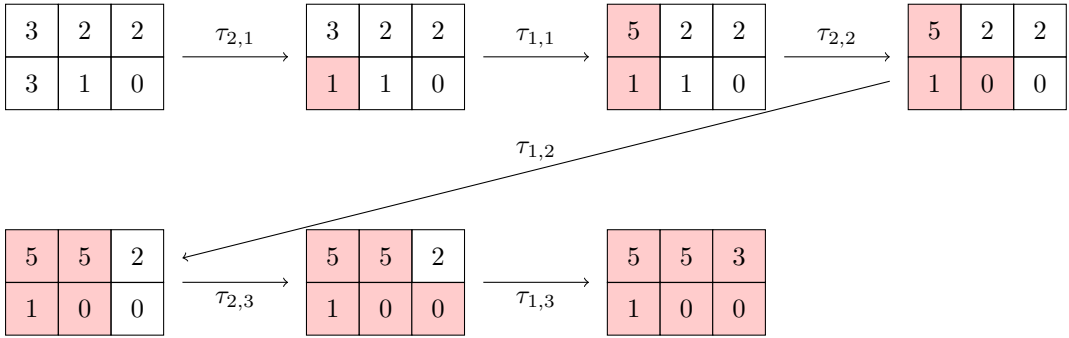
\[
\tikz{
\node (1) at (0,3) []
{$\begin{tikzpicture}[element/.style={minimum width=0.7cm,minimum height=0.7cm}]
\matrix (m) [matrix of nodes,nodes={element},column sep=-\pgflinewidth, row sep=-\pgflinewidth]{
 |[draw]|$3$ & |[draw]|$2$ & |[draw]|$2$ \\
 |[draw]|$3$ & |[draw]| $1$ & |[draw]|$0$ \\};
\end{tikzpicture}$};
\node (2) at (4,3) [] {$\begin{tikzpicture}[element/.style={minimum width=0.7cm,minimum height=0.7cm}]
\matrix (m) [matrix of nodes,nodes={element},column sep=-\pgflinewidth, row sep=-\pgflinewidth]{
 |[draw]|$3$ & |[draw]|$2$ & |[draw]|$2$ \\
 |[draw,fill=red!20]|$1$ & |[draw]| $1$ & |[draw]|$0$ \\};
\end{tikzpicture}$};
\node (3) at (8,3) [] {$\begin{tikzpicture}[element/.style={minimum width=0.7cm,minimum height=0.7cm}]
\matrix (m) [matrix of nodes,nodes={element},column sep=-\pgflinewidth, row sep=-\pgflinewidth]{
 |[draw,fill=red!20]|$5$ & |[draw]|$2$ & |[draw]|$2$ \\
 |[draw,fill=red!20]|$1$ & |[draw]| $1$ & |[draw]|$0$ \\};
\end{tikzpicture}$};
\node (4) at (12,3) [] {$\begin{tikzpicture}[element/.style={minimum width=0.7cm,minimum height=0.7cm}]
\matrix (m) [matrix of nodes,nodes={element},column sep=-\pgflinewidth, row sep=-\pgflinewidth]{
 |[draw,fill=red!20]|$5$ & |[draw]|$2$ & |[draw]|$2$ \\
 |[draw,fill=red!20]|$1$ & |[draw,fill=red!20]| $0$ & |[draw]|$0$ \\};
\end{tikzpicture}$};
\node (5) at (0,0) [] {$\begin{tikzpicture}[element/.style={minimum width=0.7cm,minimum height=0.7cm}]
\matrix (m) [matrix of nodes,nodes={element},column sep=-\pgflinewidth, row sep=-\pgflinewidth]{
 |[draw,fill=red!20]| $5$ & |[draw,fill=red!20]|$5$ & |[draw]|$2$ \\
 |[draw,fill=red!20]|$1$ & |[draw,fill=red!20]| $0$ & |[draw]|$0$ \\};
\end{tikzpicture}$};
\node (6) at (4,0) [] {$\begin{tikzpicture}[element/.style={minimum width=0.7cm,minimum height=0.7cm}]
\matrix (m) [matrix of nodes,nodes={element},column sep=-\pgflinewidth, row sep=-\pgflinewidth]{
 |[draw,fill=red!20]|$5$ & |[draw,fill=red!20]|$5$ & |[draw]|$2$ \\
 |[draw,fill=red!20]| $1$ & |[draw,fill=red!20]| $0$ & |[draw,fill=red!20]| $0$ \\};
\end{tikzpicture}$};
\node (7) at (8,0) [] {$\begin{tikzpicture}[element/.style={minimum width=0.7cm,minimum height=0.7cm}]
\matrix (m) [matrix of nodes,nodes={element},column sep=-\pgflinewidth, row sep=-\pgflinewidth]{
 |[draw,fill=red!20]| $5$ & |[draw,fill=red!20]| $5$ & |[draw,fill=red!20]| $3$ \\
 |[draw,fill=red!20]| $1$ & |[draw,fill=red!20]| $0$ & |[draw,fill=red!20]| $0$ \\};
\end{tikzpicture}$};
\draw [->] (1) -- node [above] {$\tau_{2,1}$} (2);
\draw [->] (2) -- node [above] {$\tau_{1,1}$} (3);
\draw [->] (3) -- node [above] {$\tau_{2,2}$} (4);
\draw [->] (4) -- node [above] {$\tau_{1,2}$} (5);
\draw [->] (5) -- node [below] {$\tau_{2,3}$} (6);
\draw [->] (6) -- node [below] {$\tau_{1,3}$} (7);
}
\]

Applying the cluster duality of Grassmannians, we obtain the following result.

\begin{thm} \label{main1} The action $\eta$ on $P(a,b,c)$ is of order $n=a+b$. The triple $ (P(a,b,c), \eta, M_{a,b,c}(q) )$ exhibits a cyclic sieving phenomenon.
\end{thm}

The proof of Theorem \ref{main1} involves representations of ${\rm GL}_n$. Recall that irreducible finite dimensional representations of ${\rm GL}_n$ are parametrized by weakly decreasing sequences, $\mu_1\geq\dots\geq\mu_n$, with all $\mu_i\in \mathbb{Z}$. The $a$th fundamental weight $\omega_a:=(\mu_1, \dots, \mu_n)$ takes the form $\mu_1=\dots=\mu_a=1$ and $\mu_{a+1}=\dots= \mu_n=0$. Let $V_{c\omega_a}$ be the irreducible representation of $\GL_n$ with $c\omega_a$ as its highest weight. We define the \emph{twisted cyclic rotation} on $V_{c\omega_a}$ to be the action of
\begin{equation}\label{cyclic.twist.ration}
C_a:=\begin{pmatrix} 0 & (-1)^{a-1} \\ \text{Id}_{n-1} & 0\end{pmatrix}\in \GL_n.
\end{equation}
Theorem \ref{main1} is then an easy consequence of the following result.

\begin{thm}\label{main9} The basis of $V_{c\omega_a}$ obtained in Theorem~{\rm \ref{main2}} is in natural bijection with $P(a,b,c)$. The bijection is equivariant with respect to the twisted cyclic rotation $C_a$ on $V_{c\omega_a}$ and the sequence of toggles $\eta$.
\end{thm}

After the first version of this paper was posted on arXiv, Hopkins \cite{Ho} proved that our Theorem \ref{main1} is equivalent to a result of Rhoades \cite[Theorem~1.4]{Rh}. In detail, this paper investigates the model of plane partitions under toggles, while Rhoades studied the model of semistandard tableaux under promotion. In Appendix~A of~\cite{Ho}, these two models are shown to be equivariantly equivalent to each other. However, we would like to emphasize that our approach uses Fock--Goncharov's cluster duality of Grassmannian, which is new and is significantly different from Rhoades's approach of using Kazhdan--Lusztig theory. In particular, we present a geometric interpretation of the plane partitions under toggles, i.e., they are the tropicalization of decorated configurations under rotation. It follows immediately from the latter that the order of the toggle action is~$n$. In general, the framework of cluster duality can potentially be applied to proving CSP involving other types of cluster varieties.

\section{Main definitions}\label{prelim}
\subsection{Decorated Grassmannian} \label{sec 2.1}
Let $V$ be an $n$-dimensional vector space and let $V^*$ be its dual.

\begin{defn} The \emph{decorated Grassmannian} $\dGr_a (V^* )$ is a space which parametrizes the pairs $ (W^*,f^*)$, where $W^*$ is an $a$-dimensional subspace of $V^*$ and $f^*\in \bigwedge^a W^*$ is a non-zero $a$-form.
\end{defn}

These non-zero $a$-forms $f^*$ are also known as the \emph{decomposable} elements of the exterior power $\bigwedge^aV^*$ because they can be written as a single exterior product. Therefore decorated Grassmannian $\dGr_a(V^*)$ naturally sits inside the exterior power $\bigwedge^aV^*$ as a quasi-affine variety.

There is a free right $\mathbb{G}_m$-action on $\dGr_a (V^* )$ defined via rescaling the $a$-forms
\begin{equation}\label{torus action1044}
 (W^*,f^* ) . t := (W^*,tf^* ).
\end{equation}
This action coincides with the rescaling action on $\bigwedge^aV^*$. If we projectivize the inclusion of the decorated Grassmannian $\dGr_a(V^*)\hookrightarrow \bigwedge^aV^*$ with respect to this action, we recover the Pl\"{u}cker embedding of the ordinary Grassmannian $\Gr_a (V^* )\hookrightarrow \mathbb{P}\big(\bigwedge^aV^*\big)$.
Recall that the affine cone of the ordinary Grassmannian $\Gr_a(V^*)$ in
its Pl\"{u}cker embedding is the affine subvariety of decomposable $a$-forms in $\bigwedge^a V^*$.
Therefore $\dGr_a(V^*)$ is isomorphic to the affine cone deleting the point $0\in \bigwedge^aV^*$.
This is why, throughout the paper, we sometimes refer to the ring of regular functions $\mathcal{O} (\dGr_a(V^*) )$ as a homogeneous coordinate ring of the ordinary Grassmannian~$\Gr_a(V^*)$.

The right $\mathbb{G}_m$-action on $\dGr_a(V^*)$ induces a left $\mathbb{G}_m$-action on $\mathcal{O} (\dGr_a(V^*) )$. Irreducible representations of $\mathbb{G}_m$ are 1-dimensional and are classified by integers. Denote by $\mathcal{O} (\dGr_a(V^*) )_c$ the eigenspace in $\mathcal{O} (\dGr_a(V^*) )$ that is of weight $c$ with respect to the $\mathbb{G}_m$-action.

By the natural pairing between $\bigwedge^aV^*$ and $\bigwedge^aV$, every $g\in \bigwedge^aV$ gives rise to a regular function
\[
\Delta_g\colon \ \dGr_a(V^*)\longrightarrow \mathbb{A}^1, \qquad
\Delta_g (W^*,f^* ):= \inprod{f^*}{g}.
\]
The ring $\mathcal{O} (\dGr_a(V^*) )$ is generated by $\Delta_g$ modulo certain homogeneous relations called \emph{Pl\"{u}cker relations}~(\eqref{Plucker relation}; see also, e.g., \cite[p.~211]{GH} for more details). Note that under the $\mathbb{G}_m$-action, we have
\begin{equation*}
 (t.\Delta_g ) (W^*,f^* )=\Delta_g (W^*,tf^* )=\Delta_{tg} (W^*,f^* )= t \Delta_g (W^*,f^* ).
\end{equation*}
Since the Pl\"{u}cker relations are homogeneous, the weight space decomposition of $\mathcal{O} (\dGr_a(V^*) )$ with respect to the $\mathbb{G}_m$-action coincides with the decomposition into subspaces of homogeneous degrees of the generators~$\Delta_g$. Therefore we conclude the following statement.

\begin{prop} \label{18.3.12.3.35}
The map $g\mapsto \Delta_g$ is an isomorphism between $\bigwedge^aV$ and $\mathcal{O}\left(\dGr_a(V^*)\right)_1$, and
\begin{equation}\label{grad.gr}
\mathcal{O} (\dGr_a(V^*) )=\bigoplus_{c\geq 0}\mathcal{O} (\dGr_a(V^*) )_c.
\end{equation}
\end{prop}
The group $\GL(V)$ acts on $V$ as well as its dual space $V^*$, and hence on the decorated Grassmannian~$\dGr_a(V^*)$ and on the ring $\mathcal{O} (\dGr_a(V^*) )$. In particular, the irreducible representation~$V_{\omega_a}$ of~${\rm GL}(V)$ can be constructed as $\bigwedge^a V$ (see, e.g., \cite[Lecture~15]{FH}) and hence $V_{\omega_a}$ is isomorphic to~$\mathcal{O} (\dGr_a(V^*) )_1$ by Proposition~\ref{18.3.12.3.35}.
 In general, by the Borel--Weil theorem, equation~\eqref{grad.gr} is the decomposition of $\mathcal{O} (\dGr_a(V^*) )$ into irreducible representations of $\GL(V)$ with
\begin{equation*}
\mathcal{O} (\dGr_a(V^*) )_c\cong V_{c\omega_a}.
\end{equation*}

From now on, we fix a basis $ \{e_1,\dots, e_n \}$ of $V$ and identify $V$ with the vector space $\mathds{k}^n$, where~$\mathds{k}$ is the base field. We abbreviate $\Gr_a(V^*)$ to $\Gr_a(n)$ and $\dGr_a(V^*)$ to $\dGr_a(n)$. Every $a$-element subset $I= \{i_1,\dots, i_a \}\in \binom{[n]}{a}$ gives rise to a regular function $\Delta_I:=\Delta_{e_I}$ on $\dGr_a(n)$, where~$e_I$ denotes the wedge product of vectors $e_{i_1},\dots, e_{i_a}$ taken in ascending order (e.g., $\Delta_{\{5,7,4\}}:=\Delta_{e_4\wedge e_5\wedge e_7}$). The functions $\Delta_I$ are also known as the \emph{Pl\"{u}cker coordinates}.

The Pl\"{u}cker coordinates of $\dGr_a(n)$ satisfy a set of homogeneous relations called \emph{Pl\"{u}cker relations}, and they are generated by 3-term homogeneous quadratic equations of the following form: for any $(a-2)$-element subset $J\subset \{1,\dots, n\}$ and any four distinct elements $i,j,k,l\in \{1,\dots, n\}\setminus J$ with $i<j<k<l$, the corresponding Pl\"{u}cker relation is
\begin{equation}\label{Plucker relation}
\Delta_{J\cup\{i,j\}}\Delta_{J\cup\{k,l\}}+\Delta_{J\cup \{i,l\}}\Delta_{J\cup \{j,k\}} = \Delta_{J\cup \{i,k\}}\Delta_{J\cup \{j,l\}}.
\end{equation}

Let $D_i$ be the vanishing locus of the Pl\"{u}cker coordinate $\Delta_{\{i,i+1,\dots, {i+a-1}\}}$ (indices taken modulo~$n$). Let $\dGr_a^\times (n)$ denote the complement of $D:=\cup_{i=1}^nD_i$.
Since $\dGr_a^\times (n)\subset \dGr_a(n)$, we have $\mathcal{O}\left(\dGr_a(n)\right)\subset \mathcal{O}\big(\dGr_a^\times (n)\big)$.
The image of $D$ under the projection from $\dGr_a(n)$ to $\Gr_a(n)$ is an anticanonical divisor of $\Gr_a(n)$ \cite[equation~(19.3)]{MR}. We denote the complement of~$D$ by~$\Gr_a^\times(n)$.

 Let $\mat_{a,n}^\times$ be the space of $a\times n$ matrices with column vectors $v_i$ such that every collection $\{v_i,v_{i+1},\dots, v_{i+a-1}\}$ of $a$-many cyclically consecutive column vectors is linearly independent. The group ${\rm SL}_a$ acts freely on $\mat_{a,n}^\times$ by matrix multiplication on the left.

 \begin{lem} \label{gr isomorphic to conf of column vectors} The space $\dGr_a^\times(n)$ is canonically isomorphic to the quotient space $\SL_a  \backslash \mat_{a,n}^\times $ as algebraic varieties.
 \end{lem}

 \begin{proof} Let $(W^*, f^*)\in \dGr_a^\times(n)$ and let $W$ be the dual space of $W^*$. Every subspace \mbox{$W^*\subset (\mathds{k}^n)^\ast$} naturally induces a surjection $\pi\colon \mathds{k}^n\rightarrow W$. Let $v_i:= \pi(e_i)$ be the image of the basis element~$e_i$ under $\pi$. The coordinate $\Delta_{i+1,\dots, i+a}\neq 0$ is equivalent to the linear independence of $\{v_{i+1}, \dots, v_{i+a}\}$. Up to the action of ${\rm SL}_a$ on $\mathds{k}^a$, there is a unique choice of linear isomorphisms from $W$ to $\mathds{k}^a$ whose induced pull-back map maps the standard $a$-form on $\mathds{k}^a$ (which is an element of $\bigwedge^a(\mathds{k}^a)^*$) to the $a$-form $f^*$ on $W^*$. Hence we get a configuration in ${\rm SL}_a \backslash \mat_{a,n}^\times$. It is easy to see that such a map is bijective.
 \end{proof}

 \begin{rmk} Under the above isomorphism, the coordinates $\Delta_I$ are identified with the minors of $I$-columns in an $a\times n$ matrix.
 \end{rmk}

Let $T=(\mathbb{G}_m)^n$ be the maximal torus of $\GL_n$ consisting of invertible diagonal matrices. It acts on the right of $\mat_{a,n}^\times$ by rescaling the column vectors $v_1, \dots, v_n$. Since $\dGr_a^\times(n)\cong \SL_a \backslash \mat_{a,n}^\times$, the $T$-action on $\mat_{a,n}^\times$ descends to a $T$-action on the decorated Grassmannian $\dGr_a^\times(n)$.

Define the linear transformation $C_a$ on $V$ such that
 \[
 C_a(e_i):= \begin{cases}e_{i-1} & \text{if $i\neq 1$},\\
(-1)^{a-1}e_n & \text{if $i=1$}.
\end{cases}
 \]
 It induces a twisted cyclic rotation on $\dGr_a(n)$ still denoted by $C_a$.

 To summarize, we obtain the following data
\begin{equation} \begin{cases}\text{the decorated grassmannian $\dGr_a^\times(n)$},\\
\text{the free $\mathbb{G}_m$-action on $\dGr_a^\times(n)$ by rescaling the $a$-form},\\
\text{the boundary divisor $D=\bigcup_i D_i$},\\
\text{the $T$-action on $\dGr_a^\times(n)$},\\
\text{the twisted cyclic rotation $C_a$ on $\dGr_a^\times(n)$}.
\end{cases}
\end{equation}

\subsection{Decorated configuration space}

Let $W$ be a vector space of dimension $a$.

\begin{defn} The \emph{configuration space} $\conf_n(W)$ parametrizes the $\PGL(W)$-orbits of $n$ many (not necessarily distinct) lines in $W$, i.e.,
\[
\conf_n(W):=\PGL(W)\left\backslash\left(\prod_n \mathbb{P}W\right)\right..
\]
As \looseness=-1 in Fig.~\ref{2017.5.9.6.24h}, a \emph{decorated configuration}\footnote{The definition of decorated configuration is motivated by an idea of A.B.~Goncharov on pinnings; see also \cite[Section~2.2]{GS3} for more details.} is a $\PGL(W)$-orbit of
$n$ lines in~$W$ together with linear isomorphisms $\phi_i\colon l_i\rightarrow l_{i-1}$ for each pair of neighboring lines. The \emph{decorated configuration space} is
\[
\dconf_n(W):=\PGL(W)\left\backslash\left\{
\begin{array}{@{}l@{}} \text{1-dimensional subspaces $l_1,\dots, l_n\subset W$}\\
\text{and linear isomorphisms $\phi_i\colon l_i\rightarrow l_{i-1}$}\end{array}\right\}.\right.
\]
We denote a decorated configuration as $[\phi_1, l_1, \dots, \phi_n, l_n]$.
We frequently omit the subscript of~$\phi_i$.
\begin{figure}[h!]\centering
\begin{tikzpicture}[scale=0.8, baseline=0ex]
\foreach \i in {1,...,6}
 {
 \node (\i) at (120-\i*60:1.5) [] {$\bullet$};
 \node at (120-\i*60:1.9) [] {$l_\i$};
 }
\foreach \i in {1,...,5}
 {
 \pgfmathtruncatemacro{\k}{\i+1};
 \draw [->] (\k) -- (\i);
 \node at (150-60*\i:1.6) [] {$\phi_\i$};
 }
\draw [->] (1) -- (6);
\node at (150:1.6)[] {$\phi_6$};
 \end{tikzpicture}
 \caption{A decorated configuration in $\dconf_6(W)$.} \label{2017.5.9.6.24h}
\end{figure}
 \end{defn}

Two vector spaces of the same dimension are isomorphic up to choices of bases. Since the action of $\PGL(W)$ has been quotiented out, the configuration spaces of $n$ lines in vector spaces of the same dimension are canonically isomorphic to each other. Therefore we may abbreviate $\conf_n(W)$ to $\conf_n(a)$. For the same reason, we may abbreviate $\dconf_n(W)$ to $\dconf_n(a)$.

Let $\conf^\times_n(a)$ be the subspace of $\conf_n(a)$ consisting of configurations $[l_1, \dots, l_n]$ such that every collection $\{l_{i+1},\dots, l_{i+a}\}$ of $a$-many cyclically consecutive lines is a linearly independent set of lines. The subspace $\dconf^\times_n(a)$ of $\dconf_n(a)$ is defined in the same way.

Let $[\phi_1, l_1,\dots, \phi_n, l_n]\in \dconf_n^\times(a)$.
Let us compose $\phi$ in anti-clockwise order as in Fig.~\ref{2017.5.9.6.24h}. Let $(-1)^{a-1}P$ be the $\mathbb{G}_m$-valued rescaling factor of the automorphism $\phi_{i+1}\circ \cdots \circ \phi_n\circ \phi_1\circ \cdots \circ \phi_i$ on~$l_i$. Note that $P$ is independent of the initial index $i$ chosen. We get a projection map called \emph{twisted monodromy}
\begin{equation}\label{def of P}
 P\colon \  \dconf_n^\times(a)\longrightarrow \mathbb{G}_m.
\end{equation}

\begin{prop}\label{affine} The space $\dconf^\times_n(a)$ is an affine variety of dimension $a(n-a)+1$.
\end{prop}
\begin{proof} Pick a non-zero vector $v_n\in l_n$. Using the maps $\phi_i$ recursively, we get $v_{i-1}:=\phi_i\left(v_i\right)\in l_{i-1}$ for $i=n, n-1, \dots, 2$. From the consecutive general position condition we know that for any $1\leq i\leq n$, $(v_i,v_{i-1},\dots, v_{i-a+1})$ is a basis of $\mathds{k}^a$. Define $\Phi_i$ to be the $a\times a$ matrix such that
\begin{equation}\label{def Phi}
\begin{pmatrix} v_{i-1} \\ \vdots \\ v_{i-a}\end{pmatrix}=\Phi_i\begin{pmatrix} v_i \\ \vdots \\ v_{i-a+1}\end{pmatrix}.
\end{equation}
Then $\Phi_i$ must take the form
\[
\Phi_i=\begin{pmatrix} 0 & 1 & \cdots & 0 \\
\vdots & \vdots & \ddots & \vdots \\
0 & 0 & \cdots & 1\\
* & * & \cdots & *
\end{pmatrix}.
\]
Note that $\Phi_i$ does not depend on the choice of $v_n\in l_n$. Moreover, since the monodromy $\phi^n=(-1)^{a-1}P$, the product $\Phi_1 \Phi_2\cdots \Phi_n$ should be the identity matrix multiplied by $(-1)^{a-1}P$. Therefore the configuration space $\dconf_n^\times(a)$ satisfies the equation
\begin{equation}\label{Phis}
\Phi_1\Phi_2\cdots \Phi_n=\left((-1)^{a-1}P\right)\mathrm{Id}_{a\times a}
\end{equation}
with $a$ variables in each matrix $\Phi_i$ and $P$ a non-zero parameter. Conversely, one can construct a~configuration in $\dconf_n^\times(a)$ from any solution to equation~\eqref{Phis} using the matrices $\Phi_i$. Therefore we conclude that $\dconf_n^\times(a)$ is the intersection of the vanishing loci of $Pt-1$ and the functions that are entries of $\Phi_1\Phi_2\cdots \Phi_n-\big((-1)^{a-1}P\big)\mathrm{Id}_{a\times a}$ in the affine space $\mathbb{A}^{na}_\Phi\times \mathbb{A}^1_P\times \mathbb{A}^1_t$. In particular, this shows that the map $P$ is regular on $\dconf_n^\times(a)$.

The dimension of $\dconf^\times_n(a)$ is
\begin{align*}\dim \dconf^\times_n(a)&= n\dim \big(\mathbb{P}^{a-1}\big)+\dim \{\text{isomorphisms $ \phi_i$}\}-\dim \PGL_a \\
&=n(a-1)+n-\big(a^2-1\big)  =a(n-a)+1. \tag*{\qed}
\end{align*}
\renewcommand{\qed}{}
\end{proof}

Below we introduce three natural maps from $\dconf_n^\times(a)$. Let $[\phi_1, l_1,\dots, \phi_n, l_n]\in \dconf_n^\times(a)$. Pick a non-zero vector $v_i\in l_i$ for $i=1,\dots, n$.

Because of the cyclically consecutive general position condition, the vector space quotient $W/\mathrm{Span} \{l_{i-a+2},\dots, l_i\}$ is 1-dimensional and is spanned by $\overline{v_{i-a+1}}$ (the image of $v_{i-a+1}$ under this quotient). Let $\vartheta_i$ be the scalar such that in the quotient space $W/\mathrm{Span} \{l_{i-a+2},\dots, l_i\}$,
\begin{equation}\label{theta.f}
\overline{\phi(v_{i-a+1})} = \vartheta_i \overline{v_{i-a+1}} .
\end{equation}
Note that $\vartheta_i$ is independent of the choices of $v_i$.

\looseness=-1 Note that for $i\neq a$, $\vartheta_i$ is the $(a,a)$-entry of the matrix $\Phi_i$ defined in equation~\eqref{def Phi}; for $i=a$, $\vartheta_a$ is the product of the $(a,a)$-entry of $\Phi_a$ and $(-1)^{a-1}P$. Since the function $P$ and the entries of~$\Phi_i$ are all regular functions on $\dconf_n^\times(a)$, the functions $\vartheta_i$ are also regular on $\dconf_n^\times(a)$ for all~$i$.

The \emph{potential function} on $\dconf^\times_n(a)$ is defined to be the regular function
\begin{equation}\label{potential.W}
\mathcal{W}=\sum_{i=1}^n \vartheta_i\colon \  \dconf_n^\times(a)\longrightarrow \mathbb{A}^1.
\end{equation}

Since the top exterior power $\bigwedge^aW$ is 1-dimensional, under the cyclically consecutive general position condition we may define
\begin{equation}\label{M}
M_k:= \frac{\phi\left(v_{k-a+1}\right)\wedge \dots \wedge \phi\left(v_k\right)}{v_{k-a+1}\wedge \dots \wedge v_k}.
\end{equation}
Note that the value $M_k$ is nonzero and does not depend on the choices of $v_i$. Therefore we obtain a {\it weight} map
\begin{align}
 M\colon \ \dconf_n^\times(a)&\longrightarrow T^\vee,\nonumber\\
 [\phi_1,l_1,\dots, \phi_n,l_n ] & \longmapsto  (M_1,\dots, M_n ),\label{weightmap1142}
 \end{align}
where $T^\vee\cong (\mathbb{G}_m )^n$ is the dual torus of the maximal torus $T\subset \GL_n$.

Lastly, there is an order $n$ biregular map
\begin{align*}
R\colon \ \dconf_n^\times(a) &\longrightarrow \dconf_n^\times(a),\\
 [\phi_1,l_1,\phi_2,l_2, \dots, \phi_n, l_n ] & \longmapsto  [\phi_n,l_n,\phi_1,l_1,\dots, \phi_{n-1},l_{n-1} ].
\end{align*}

To summarize, we get the following data
\begin{equation*}
\begin{cases} \text{the decorated configuration space $\dconf_n^\times(a)$},\\
\text{the twisted monodromy $P\colon \dconf_a^\times(n)\rightarrow \mathbb{G}_m$},\\
\text{the potential function $\mathcal{W}=\sum_i \vartheta_i$},\\
\text{the weight map $M\colon \dconf_n^\times(a)\rightarrow T^\vee$},\\
\text{the cyclic rotation $R$ on $\dconf_n^\times(n)$}.
\end{cases}
\end{equation*}

\subsection{Maps among the decorated spaces}

Recall the $n$-dimensional vector space $V$ with a basis $ \{e_1,\dots, e_n \}$. Let $\hat{l}_i$ be the line spanned by~$e_i$ and let $\hat{\phi}_i\colon \hat{l}_i\rightarrow \hat{l}_{i-1}$ be the linear isomorphism such that
\begin{equation}
\label{19.9.15.113}
\hat{\phi}_i\left(e_i\right):= \begin{cases}e_{i-1} & \text{if $i\neq 1$},\\
(-1)^{a-1}e_n & \text{if $i=1$}.
\end{cases}
\end{equation}

As mentioned in Section~\ref{sec 2.1}, the cyclically consecutive Pl\"{u}cker coordinates $\Delta_{\{i,\dots, i+a-1\}}$, with indices taken modulo $n$, cut out an anti-canonical divisor on $\Gr_a(V^*)$, whose complement is denoted by
$\Gr_a^\times(V^*)$. Every $W^* \in \Gr_a^\times(V^*)$ induces a projection $\pi$ from~$V$ to the dual space~$W$ of $W^*$, and since $\Delta_{\{i,\dots, i+a-1\}} (W^* )\neq 0$ for all $i$, the image of every $e_i$ under $\pi$ is a~non-zero vector $v_i$, and the lines $l_i:=\pi(\hat{l}_i)$ in $W$ automatically satisfy the consecutive general position condition. Furthermore, isomorphisms $\hat{\phi}_i$ descend to isomorphisms $\phi_i\colon l_i\rightarrow l_{i-1}$. Thus we obtain $[\phi_1, l_1, \dots, \phi_n, l_n]\in \dconf_n^\times (W)\cong\dconf_n^\times (a)$. This defines a natural map
\begin{align}
 \Gr_a^\times(n) &\longrightarrow \dconf_n^\times(a),\nonumber\\
 W^* & \longmapsto  [\phi_1,l_1,\dots, \phi_n,l_n ].\label{p map}
 \end{align}

\begin{prop}\label{2.20} The map~\eqref{p map} is injective. The image of~\eqref{p map} consists of decorated configurations of twisted monodromy $P=1$.
\end{prop}
\begin{proof} By \eqref{19.9.15.113}, every point in the image of \eqref{p map} is of twisted monodromy $P=1$.

Let $[\phi_1,l_1,\dots, \phi_n,l_n]$ be a configuration in $\dconf^\times_n(a)$ with $P=1$.
It remains to show that there is a unique $W^\ast \in {\Gr}_a^\times(V^\ast)$ whose image under the map \eqref{p map} is $[\phi_1,l_1,\dots, \phi_n,l_n]$.
Let $(\phi_1,l_1,\dots, \phi_n,l_n)$ be a representative of the configuration $[\phi_1,l_1,\dots, \phi_n,l_n]$.
Pick a non-zero vector $v_n\in l_n$ and then use the isomorphisms $\phi_n,\dots, \phi_2$ to get non-zero vectors $v_i\in l_i$. It gives rise to a linear projection $\pi\colon V\rightarrow \mathds{k}^a$ which sends $e_i\mapsto v_i$. The dual of this map determines an embedding $\pi^*\colon \mathds{k}^a\hookrightarrow V^*$, whose image is denoted as $W^*$. Note that changing the choice of representative and the choice of $v_n$ corresponds to postcomposing the projection $\pi\colon V\rightarrow \mathds{k}^a$ with an element of~$\GL_a$. It corresponds to precomposing the dual map $\pi^\ast\colon \mathds{k}^a\hookrightarrow V^*$ with an element of~$\GL_a$, which does not effect the image~$W^*$ of~$\pi^\ast$ as a subspace of~$V^*$.
\end{proof}

\begin{defn} \label{twisted mono}
 Denote by $\widetilde{\dconf_n^{\times}(a)}$ the space of ${\rm SL}_a$-orbits of
\[
(\phi_1, v_1, \phi_2, v_2,\dots, \phi_n, v_n),
\]
where $v_i$ are vectors in $\mathds{k}^a$ satisfying consecutive general position condition, and each $\phi_i$ is a linear isomorphism from the line spanning $v_i$ to the line spanning $v_{i-1}$, and $\SL_a$ acts on $(\phi_1,v_1,\dots, \phi_n,v_n)$ by
\[
 g. (\phi_1,v_1,\dots, \phi_n,v_n ):=\big(g\phi_1g^{-1}, gv_1,\dots, g\phi_ng^{-1},gv_n\big).
\]
\end{defn}

By Lemma \ref{gr isomorphic to conf of column vectors}, we can think of the space $\dGr_a^\times(n)$ as the moduli space of configurations of vectors $[v_1, \dots, v_n]$ satisfying the cyclic general position condition. Define the following map
\begin{align*}
\dGr_a^\times(n) &\longrightarrow \widetilde{\dconf_n^\times(a)},\\
[v_1,\dots, v_n]& \longmapsto [\phi_1,v_1,\dots, \phi_n,v_n],
\end{align*}
with the isomorphisms $\phi$ defined by $\phi_1(v_1):=(-1)^{a-1}v_n$ and $\phi_i(v_i):= v_{i-1}$ for the other $i$'s. It is not hard to see that this map is injective.

There is a surjective map $\widetilde{\dconf_n^\times(a)}\rightarrow \dconf_n^\times(a)$ by replacing each~$v_i$ by its spanning line~$l_i$.

There is a surjective map $\dconf_n^\times(a)\rightarrow \conf_n^\times(a)$ defined by forgetting the isomorphisms $\phi_i$.

There is a surjective map $\dGr_{a}^\times(n)\rightarrow\Gr_a^\times(n)$ defined by forgetting the $a$-form.

\begin{prop} Putting all the aforementioned maps together, we obtain commutative diagram
\begin{gather}\label{big diagram}
\vcenter{\vbox{\xymatrix@R-1pc{&\widetilde{\dconf_n^\times(a)} \ar@{->>}[dr] & & \\
\dGr_a^\times(n) \ar@{^(->}[ur] \ar@{->>}[dr] \ar@{->>}@/^5ex/[rrr] & & \dconf_n^\times(a). \ar@{->>}[r]
& \conf_n^\times(a) \\
 & \Gr_a^\times(n) \ar@{^(->}[ur] & &
}}}
\end{gather}
\end{prop}
\begin{proof} It remains to prove the commutativity of the rhombus and the surjectivity of the map $\dGr_a^\times(n)\twoheadrightarrow \conf^\times_n(a)$.

First, by Lemma \ref{gr isomorphic to conf of column vectors}, we can again view $\dGr_a^\times(n)$ as moduli space of configurations of vectors $[v_1,\dots, v_n]$ in $\mathds{k}^a$ satisfying cyclic general position condition. But these $v_i$ are precisely the vectors that span the lines $l_i$ in the construction of the map $\Gr_a^\times(n)\hookrightarrow \dconf^\times_n(a)$. Furthermore, in both definitions of the maps $\dGr_a^\times(n)\hookrightarrow \widetilde{\dconf_n^\times(a)}$ and $\Gr_a^\times(n)\hookrightarrow \dconf^\times_n(a)$ we have set $\phi_i (v_i )=v_{i-1}$ for all $i\neq 1$ and $\phi_1 (v_1 )=(-1)^{a-1}v_n$. Therefore the image in $\dconf^\times_n(a)$ from $\dGr_a^\times(n)$ by going down either side of the rhombus is the same.

As for the surjectivity of the map $\dGr_a^\times(n)\twoheadrightarrow \conf_n^\times(a)$, by Proposition~\ref{2.20} it suffices to show that any element of $\conf_n^\times(a)$ is the image of an element in $\dconf_n^\times(a)$ with $P=1$. But this is clear since we have $n$ degrees of freedom to choose the isomorphisms $\phi_i$ to put between the lines.
\end{proof}

\subsection{Cluster structures}\label{clustersec1712}

The pair $ \big(\dGr_a^\times (n), \dconf_n^\times(a)\big)$ admits a natural structure of cluster ensemble associated to Postnikov's reduced plabic graphs\footnote{A rapid review on reduced plabic graphs has been included in the appendix. See also Postnikov's article \cite{Pos} for more details.} of rank $a$ on a disk with $n$ marked points on the boundary. In this section we focus on one particular reduced plabic graph $\Gamma_{a,n}$ for each pair of parameters~$(a,n)$, as depicted in the following picture:
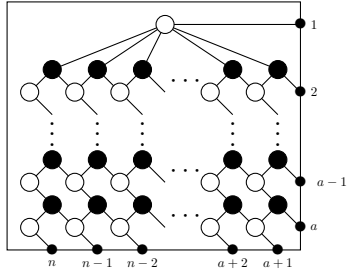
\begin{figure}[h!]\centering
\begin{tikzpicture}[scale=0.6]
\foreach \i in {0,1,2,4,5}
 {
 \draw (\i+0.5,3.5) -- (3,4.5);
 }
\draw (3,4.5) -- (6,4.5);
\draw [fill=white] (3,4.5) circle [radius=0.2];
\foreach \i in {0,1,2,4,5}
 {
 \draw (\i+1,3) -- (\i+0.5,3.5) -- (\i,3) -- (\i+0.5,2.5);
 \draw [fill=white] (\i,3) circle [radius=0.2];
 \draw [fill] (\i+0.5,3.5) circle [radius=0.2];
 }
\node at (3.5,3.25) [] {$\cdots$};
\foreach \i in {0,1,2,4,5}
 {
 \node at (\i+0.5,2.25) [] {$\vdots$};
 }
\foreach \i in {0,1,2,4,5}
 {
 \draw (\i+1,1) -- (\i+0.5,1.5) -- (\i,1) -- (\i+0.5,0.5);
 \draw [fill=white] (\i,1) circle [radius=0.2];
 \draw [fill] (\i+0.5,1.5) circle [radius=0.2];
 }
\node at (3.5,1.25) [] {$\cdots$};
\foreach \i in {0,1,2,4,5}
 {
 \draw (\i+1,0) -- (\i+0.5,0.5) -- (\i,0) -- (\i+0.5,-0.5);
 \draw [fill=white] (\i,0) circle [radius=0.2];
 \draw [fill] (\i+0.5,0.5) circle [radius=0.2];
 }
\node at (3.5,0.25) [] {$\cdots$};
\draw (-0.5,-0.5) rectangle (6,5);
\node at (6,4.5) [] {$\bullet$};
\node [scale=0.5] at (6.3,4.5) [] {$1$};
\node at (6,3) [] {$\bullet$};
\node [scale=0.5] at (6.3,3) [] {$2$};
\node at (6,1) [] {$\bullet$};
\node [scale=0.5] at (6.7,1) [] {$a-1$};
\node at (6,0) [] {$\bullet$};
\node [scale=0.5] at (6.3,0) [] {$a$};
\node at (5.5,-0.5) [] {$\bullet$};
\node [scale=0.5] at (5.5,-0.8) [] {$a+1$};
\node at (4.5,-0.5) [] {$\bullet$};
\node [scale=0.5] at (4.5,-0.8) [] {$a+2$};
\node at (2.5,-0.5) [] {$\bullet$};
\node [scale=0.5] at (2.5,-0.8) [] {$n-2$};
\node at (1.5,-0.5) [] {$\bullet$};
\node [scale=0.5] at (1.5,-0.8) [] {$n-1$};
\node at (0.5,-0.5) [] {$\bullet$};
\node [scale=0.5] at (0.5,-0.8) [] {$n$};
\end{tikzpicture}
\end{figure}

From a reduced plabic graph $\Gamma$ we obtain a quiver $Q_\Gamma$ by the following standard four-step procedure. This quiver (or its opposite) is used to study the cluster structure on Grassmannians by many others, including (but not restricted to) Postnikov~\cite{Pos}, Gekhtman, Shapiro, and Vainshtein~\cite{GSVcp}, and Rietsch and Williams~\cite{RW}.
\begin{itemize}\setlength\itemsep{0pt}
\item Assign a vertex to each face of $\Gamma$.
\item For each black vertex of $\Gamma$, draw a clockwise cycle of arrows as follows:
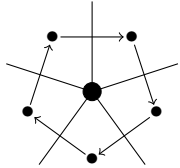
\[
\begin{tikzpicture}[scale=0.6]
\foreach \i in {0,...,4}
 {
 \draw (0,0) -- +(90-72*\i:2);
 \node at +(54-72*\i:1.5) [] {$\bullet$};
 \draw [->] +(48-72*\i:1.4) to +(-12-72*\i:1.4);
 }
\draw [fill] (0,0) circle [radius=0.2];
\end{tikzpicture}
\]
\item Remove a maximal subset of 2-cycles.
\item Freeze the vertices corresponding to boundary faces of $\Gamma$.
\end{itemize}
For example, the reduced plabic graph $\Gamma_{a,n}$ gives rise to a quiver $Q_{a, n}$ as follows (please keep in mind that we have the convention $n=a+b$). The unfrozen part of this quiver is also known as the triangle product $\mathrm{A}_{a-1}\boxtimes \mathrm{A}_{b-1}$ (see, e.g., \cite[p.~8]{Kelperiod}, \cite[p.~2]{Weng}):
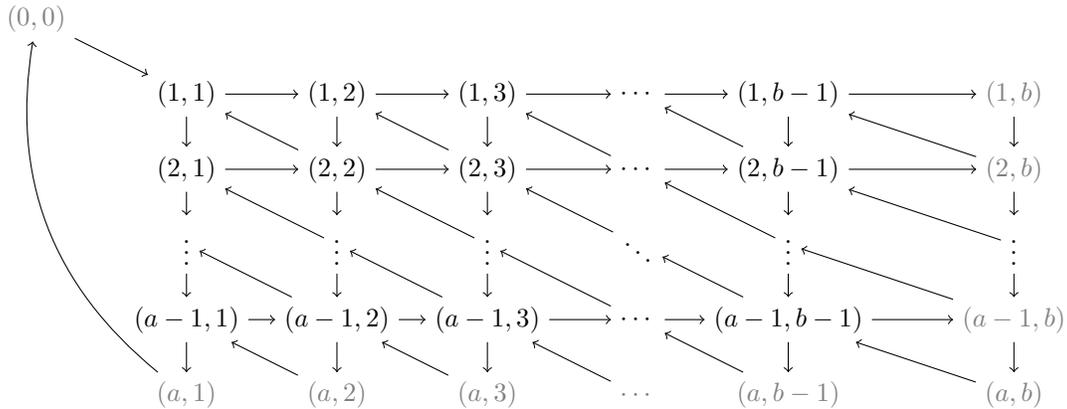
\[
\begin{tikzpicture}
\node (0-0) [gray] at (0,5) [] {$(0,0)$};
\foreach \i in {1,2}
{
\foreach \j in {1,2,3}
{
\node (\i-\j) at (2*\j,5-\i) [] {$(\i,\j)$};
}
\node (\i-4) at (8,5-\i) [] {$\cdots$};
\node (\i-5) at (10,5-\i) [] {$(\i,b-1)$};
\node (\i-6) [gray] at (13,5-\i) [] {$(\i,b)$};
}
\foreach \j in {1,2,3}
{
\node (3-\j) at (2*\j, 2) [] {$\vdots$};
}
\node (3-4) at (8,2) [] {$\ddots$};
\node (3-5) at (10,2) [] {$\vdots$};
\node (3-6) at (13,2) [] {$\vdots$};
\foreach \j in {1,2,3}
{
\node (4-\j) at (2*\j,1) [] {$(a-1,\j)$};
}
\node (4-4) at (8,1) [] {$\cdots$};
\node (4-5) at (10,1) [] {$(a-1,b-1)$};
\node (4-6) [gray] at (13,1) [] {$(a-1,b)$};
\foreach \j in {1,2,3}
{
\node (5-\j) [gray] at (2*\j,0) [] {$(a,\j)$};
}
\node (5-4) [gray] at (8,0) [] {$\cdots$};
\node (5-5) [gray] at (10,0) [] {$(a,b-1)$};
\node (5-6) [gray] at (13,0) [] {$(a,b)$};
\foreach \i in {1,2,4}
{
\foreach \j in {2,...,6}
 {
 \pgfmathtruncatemacro{\k}{\j-1};
 \draw [<-] (\i-\j) -- (\i-\k);
 }
}
\foreach \j in {1,2,3,5,6}
{
\draw [<-] (2-\j) -- (1-\j);
\draw [<-] (3-\j) -- (2-\j);
\draw [<-] (4-\j) -- (3-\j);
\draw [<-] (5-\j) -- (4-\j);
}
\foreach \i in {1,...,4}
 {
 \foreach \j in {1,...,5}
 {
 \pgfmathtruncatemacro{\k}{\i+1};
 \pgfmathtruncatemacro{\l}{\j+1};
 \draw [<-] (\i-\j) -- (\k-\l);
 }
 }
\draw [<-] (1-1) -- (0-0);
\draw [<-] (0-0) to [bend right] (5-1);
\end{tikzpicture}
\]
Here the vertex assigned to the top-left face is indexed by $(0,0)$. The other faces of $\Gamma_{a,n}$ form an $a\times b$ grid. Their corresponding vertices of $Q_{a,n}$ are indexed in the same way as matrix entries. The gray vertices are frozen. Denote by $I$ the set of vertices of $Q_{a,n}$ and by $I^{\uf}$ the set of unfrozen vertices of $Q_{a,n}$. The exchange matrix $\varepsilon$ of $Q_{a,n}$ is defined to be an $I\times I$ matrix with entries
\[
\varepsilon_{fg}=\#\{g\rightarrow f\}-\#\{f\rightarrow g\}.
\]

\begin{rmk} The definition of the exchange matrix $\epsilon_{fg}$ differs from the usual convention by an extra minus sign. The reason we include this extra minus sign is to simplify later computations.
\end{rmk}

For simplicity, we will also use an integer $i\in\{1,\dots, n\}$ to denote the frozen vertex corresponding to the boundary face lying between $i$ and $i+1$. In other words,
\[
\text{frozen vertex $i$}= \begin{cases}
 (i,b) & \text{if $1\leq i\leq a$}, \\
 (a,n-i) & \text{if $a\leq i<n$},\\
 (0,0) & \text{if $i=n$}.
\end{cases}
\]

Let $(\mathscr{A}_{a,n}, \mathscr{X}_{a,n})$ be the cluster ensemble associated to $Q_{a,n}$. See \eqref{cluster.ensemble.def.7.19} for its rigorous definition. Let $\{A_{i,j}\}_{(i,j)\in I}$ be the $K_2$ cluster of $\mathscr{A}_{a,n}$ associated to the quiver $Q_{a,n}$ and let $\{X_{i,j}\}_{(i,j)\in I}$ be the Poisson cluster of $\mathscr{X}_{a,n}$ associated to $Q_{a,n}$.
Abusing notation, we will frequently write $f$ instead of $(i,j)\in I$ with $f$ being the face of $\Gamma_{a,n}$ corresponding to the vertex $(i,j)$ of $Q_{a,n}$.

{\bf Cluster $\boldsymbol{K_2}$ structure on $\boldsymbol{\dGr_{a}^\times(n)}$.}
We associate to each vertex $(i,j)$ (including the vertex $(i,j) = (0,0)$) of $Q_{a,n}$ an $a$-element set\footnote{The set is determined by the zig-zag strands of the reduced plabic graph $\Gamma_{a,n}$. See Appendix~\ref{appendixB} for more details.}
\begin{equation}\label{I}
I(i,j):=\{\underbrace{b-j+1,\dots, b-j+i}_\text{$i$ indices}, \underbrace{b+i+1, \dots, n}_\text{$a-i$ indices}\}.
\end{equation}
Recall the Pl\"ucker coordinates $\Delta_I$ of $\dGr_a^\times(n)$. By defining
\begin{equation}\label{aij}
A_{i,j}:=\Delta_{I(i,j)},
\end{equation}
 we get a rational map
 \begin{equation} \label{he.9.25}
 \psi\colon \ \dGr_a^\times(n)\dashrightarrow \mathscr{A}_{a,n}.
 \end{equation}
 Scott \cite[Theorem~3]{Sco} showed that the pull-back map $\psi^*$ gives an algebra isomorphism between $\mathcal{O} (\dGr_a(n))$ and the ordinary cluster algebra defined by the quiver $Q_{a,n}$; by allowing ourselves to invert the frozen variables we generalize her result to the following theorem.

\begin{thm}\label{Scott} The pull-back map $\psi^*$ is an algebra isomorphism between the upper cluster algebra $\mathbf{up} (\mathscr{A}_{a,n} ):=\mathcal{O} (\mathscr{A}_{a,n} )$ and $\mathcal{O}\big(\dGr_{a}^\times(n)\big)$.
\end{thm}

Both $\mathcal{G}r_a^\times(n)$ and $\mathscr{A}_{a,n}$ are rational varieties. The map $\psi$ induces an isomorphism between their function fields. Therefore the map $\psi\colon \dGr_a^\times(n)\dashrightarrow \mathscr{A}_{a,n}$ is birational.

{\bf Cluster Poisson structure on $\boldsymbol{\conf_n^\times(a)}$.}
Let $Q^{\uf}_{a,n}$ denote the full subquiver of $Q_{a,n}$ spanned by vertices in $I^{\uf}$. Let $\big(\mathscr{A}^{\uf}_{a,n}, \mathscr{X}^{\uf}_{a,n}\big)$ be the cluster ensemble associated to $Q^{\uf}_{a,n}$.

 There is a canonical regular map $p\colon \mathscr{A}_{a,n}\rightarrow \mathscr{X}_{a,n}^{\uf} $ defined on the cluster coordinate charts associated to $Q_{a,n}$ such that
\begin{equation}\label{in.p.9.43}
p^*(X_g)=\prod_{f\in I} A_f^{\varepsilon_{fg}}, \qquad \forall\, g\in I^{\uf}.
\end{equation}
Recall the surjective map $\dGr_a^\times(n)\rightarrow \conf_n^\times(a)$ in \eqref{big diagram}. We define a rational map
\begin{equation}\label{X.map.9.26}
\psi\colon \ \conf_n^\times(a)\dashrightarrow \mathscr{X}_{a,n}^{\uf}
\end{equation}
by first taking a lift from $\conf_n^\times(a)$ to $\dGr_a^\times(n)$, mapping over to $\mathscr{A}_{a,n}$ via the birational equivalence~\eqref{he.9.25}, and then mapping it down to $\mathscr{X}_{a,n}^{\uf}$ by the canonical $p$ map. The map~\eqref{X.map.9.26} is well-defined and does not depend on the lift, because $\psi^*(X_g)$ for $g\in I^{\uf}$ is a ratio of Pl\"{u}cker coordintes with the same collection of indices (counted with multiplicity) in the numerator and in the denominator. It is known that~\eqref{X.map.9.26} is a~birational equivalence (see for example~\cite{Weng}).

\begin{lem}\label{full.rank.lem}
The restricted exchange matrix $\varepsilon|_{I\times I^{\uf}}$ of $Q_{a,n}$ is of full-rank.
\end{lem}
\begin{proof} From~\eqref{in.p.9.43} we see that the restriction of the map $p\colon \mathscr{A}_{a,n}\rightarrow \mathscr{X}_{a,n}^{\uf}$ to a cluster coordinate chart is a map between algebraic tori defined by the matrix $\varepsilon|_{I\times I^{\uf}}$ with integer entries. Therefore to show that $\varepsilon|_{I\times I^{\uf}}$ is full-ranked, it suffices to show that $p\colon \mathscr{A}_{a,n}\rightarrow \mathscr{X}_{a,n}^{\uf}$ is surjective.

By definition, we have the following commutative diagram
\begin{equation*}
\vcenter{\vbox{\xymatrix{ \dGr_a^\times(n) \ar@{-->}[r]^\psi \ar[d]_p & \mathscr{A}_{a,n} \ar[d]^p \\
\conf_n^\times(a) \ar@{-->}[r]_\psi & \mathscr{X}_{a,n}^{\uf}.}}}
\end{equation*}
with the map $p$ on the left as defined in~\eqref{big diagram}. Since both of the horizontal maps $\psi$ are birational and the map~$p$ on the left is surjective, we know that the map $p$ on the right is dominant. Note that the restriction of the map~$p$ on the right to each seed torus is a dominant morphism induced by a linear map between their character lattices, which forces it to be surjective. Therefore the map~$p$ on the right is surjective.
\end{proof}

{\bf Cluster $\boldsymbol{K_2}$ structure on $\boldsymbol{\widetilde{\dconf_n^\times(a)}}$.}
The quiver $\widetilde{Q_{a,n}}$ is obtained from $Q_{a,n}$ by adding, for each frozen vertex $i$, a new frozen vertex $i'$ and a new arrow from $i$ to $i'$, so that the number of frozen vertices increases from $n$ to $2n$. Denote the set of vertices of $\widetilde{Q_{a,n}}$ by $\tilde{I}$ and the exchange matrix of $\widetilde{Q_{a,n}}$ by $\tilde{\varepsilon}$. For instance, the quiver $\widetilde{{Q}_{3,6}}$ is as follows:
\begin{equation}\label{expand quiver example}
\begin{tikzpicture}[scale=0.7, baseline=10ex]
\node (-1-0) at (-1,5) [blue,label=left:{$6'$}] {$\square$};
\node (0-0) at (0,5) [red] {$\circ$};
\foreach \i in {1,2}
{
\foreach \j in {1,2}
{
\node (\i-\j) at (\j,5-\i) [] {$\bullet$};
}
\node (3-\i) at (\i, 2) [red] {$\circ$};
\node (\i-3) at (3,5-\i) [red] {$\circ$};
\node (\i-4) at (4,5-\i) [blue, label=right:{$\i'$}] {$\square$};
}
\node (4-1) at (1, 1) [blue, label=below:{$5'$}] {$\square$};
\node (4-2) at (2, 1) [blue, label=below:{$4'$}] {$\square$};
\node (3-3) at (3,2) [red] {$\circ$};
\node (4-4) at (3,1) [blue, label=below:{$3'$}] {$\square$};
\draw[<-] (4-4) -- (3-3);
\foreach \i in {1,2}
{
\foreach \j in {2,3}
 {
 \pgfmathtruncatemacro{\k}{\j-1};
 \draw [<-] (\i-\j) -- (\i-\k);
 }
 \draw [<-] (4-\i) -- (3-\i);
 \draw [->] (\i-3) -- (\i-4);
}
\foreach \j in {1,2,3}
{
\draw [<-] (2-\j) -- (1-\j);
\draw [<-] (3-\j) -- (2-\j);
}
\foreach \i in {1,2}
 {
 \foreach \j in {1,2}
 {
 \pgfmathtruncatemacro{\k}{\i+1};
 \pgfmathtruncatemacro{\l}{\j+1};
 \draw [<-] (\i-\j) -- (\k-\l);
 }
 }
\draw [<-] (1-1) -- (0-0);
\draw [<-] (-1-0) -- (0-0);
\draw [<-] (0-0) to [bend right] (3-1);
\end{tikzpicture}
\end{equation}
Let $\widetilde{\mathscr{A}_{a,n}}$ be the cluster $K_2$ variety associated to the quiver $\widetilde{Q_{a,n}}$.

Let $[\phi_1, v_1, \dots, \phi_n, v_n]\in \widetilde{\dconf_n^\times(a)}$. For each $1\leq i\leq n$ we define a \emph{scaling factor} $\lambda_i$ by
\begin{equation}
\label{rescalefa}
\phi(v_{i+1})=  \begin{cases}
 \lambda_i v_{i} & \mbox{if } i\neq n, \\
 (-1)^{a-1}\lambda_n v_n & \mbox{if } i=n.
\end{cases}
\end{equation}
We define a rational map
\[
\tilde{\psi}\colon \ \widetilde{\dconf_n^\times(a)}\dashrightarrow \widetilde{\mathscr{A}_{a,n}}
\]
by setting\footnote{Throughout the paper we adopt the convention that the indices of $\lambda_i$ are taken modulo $n$.}
\begin{equation}\label{1.tt}
\tilde{\psi}^* \left(A_{f}\right):= \begin{cases}
\Delta_{I(i,j)} & \mbox{if $f$ is a vertex $(i,j)$ of $Q_{a,n}$}, \vspace{1mm}\\
\lambda_{i-a}\displaystyle\frac{\Delta_{\{i-a,i-a+1,\dots, i-1\}}}{\Delta_{\{i-a+1,i-a+2,\dots, i\}}}& \mbox{if $f$ is a newly added vertex $i'$}.
\end{cases}
\end{equation}

\begin{cor}\label{A isomorphism} The map $\tilde{\psi}\colon \widetilde{\dconf_n^\times(a)}\dashrightarrow \widetilde{\mathscr{A}_{a,n}}$ is birational.
Its pull-back map $\tilde{\psi}^*$ is an algebra isomorphism between the upper cluster algebra $\mathbf{up}\big(\widetilde{\mathscr{A}_{a,n}}\big):=\mathcal{O}\big(\widetilde{\mathscr{A}_{a,n}}\big)$ and $\mathcal{O}\big(\widetilde{\dconf_n^\times(a)}\big)$.
\end{cor}
\begin{proof} Let ${\rm H}= (\mathbb{G}_m )^n$ be the split algebraic torus with coordinates $ (A_{1'}, \dots, A_{n'} )$. Note that there is no arrow between the vertices $i'$ and the unfrozen vertices in $Q_{a,n}$. Therefore the variables $A_{i'}$ do not affect the cluster mutations. Hence we get
\[
\widetilde{\mathscr{A}_{a,n}}\cong \mathscr{A}_{a,n}\times {\rm H}.
\]

On the other hand, Let $\mathrm{T}= (\mathbb{G}_m^n )$ be the split algebraic torus with coordinates $ (\lambda_1,\dots, \lambda_n )$. Note that the $\SL_a$-orbit of the $n$-tuple of vectors $ (v_1,\dots, v_n )$ satisfying the cyclic general position condition is captured by a point in $\dGr_a^\times(n)$ (Lemma~\ref{gr isomorphic to conf of column vectors}), and the $\SL_a$-orbit of the $n$-tuple of linear isomorphisms $ (\phi_1,\dots, \phi_n )$ is captured by a point in~$\mathrm{T}$. Hence we get
\[
\widetilde{\dconf_n^\times(a)} \cong \dGr_a^\times(n)\times \mathrm{T}.
\]

Now we have the following commutative diagram, where the vertical maps are a trivial $\mathrm{T}$-fiber bundle and a trivial $\mathrm{H}$-fiber bundle respectively
\begin{equation}\label{ext a to a}\begin{split}&
\xymatrix{ \widetilde{\dconf_n^\times(a)} \ar@{-->}[r]^{\tilde{\psi}} \ar[d] & \widetilde{\mathscr{A}_{a,n}} \ar[d] \\
\dGr^\times_a(n) \ar@{-->}[r]_{\psi} & \mathscr{A}_{a,n}.}\end{split}
\end{equation}
Recall from Theorem~\ref{Scott} that the bottom map $\psi$ is a birational map which induces an isomorphism between algebras of regular functions. Moreover, since the Pl\"{u}cker coordinates in the second line of \eqref{1.tt} are both invertible, $\tilde{\psi}$ restricts to an isomorphism between the algebraic torus fibers over any pair of corresponding points on the bases. Therefore we can conclude that~$\tilde{\psi}$ is birational and~$\tilde{\psi}^*$ is an algebra isomorphism between $\mathbf{up}\big(\widetilde{\mathscr{A}_{a,n}}\big)$ and $\mathcal{O}\big(\widetilde{\dconf_n^\times(a)}\big)$.
\end{proof}

{\bf Cluster Poisson structure on $\boldsymbol{\dconf_n^\times(a)}$.}
Analogous to the map \eqref{in.p.9.43}, there is a canonical map $p\colon \widetilde{\mathscr{A}_{a,n}}\rightarrow \mathscr{X}_{a,n}$ defined on the cluster coordinate charts associated to $Q_{a,n}$ by
\begin{equation}\label{def poisson cluster}
p^* (X_g ) :=\prod_{f\in \tilde{I}} A_f^{\tilde{\varepsilon}_{fg}}, \qquad \forall\, g\in I.
\end{equation}

Let $[\phi_1, l_1,\dots, \phi_n, l_n] \in {\dconf_n^\times(a)}$. Let us lift it against $p\colon \widetilde{\dconf_n^\times(a)}\twoheadrightarrow\dconf_n^\times(a)$ by picking a nonzero vector $v_i\in l_i$ for each~$i$. Composing with the map
$
p\circ \tilde{\psi}\colon  \widetilde{\dconf_n^\times(a)}{\dashrightarrow} \widetilde{\mathscr{A}_{a,n}}{\longrightarrow} \mathscr{X}_{a,n},
$
we get a rational map
\[
\psi\colon \ \dconf_n^\times(a)\dashrightarrow \mathscr{X}_{a,n}.
\]

The map $\psi$ does not depend on the choices of $v_i\in l_i$. Indeed, if the vertex $(i,j)$ is unfrozen, then $X_{i,j}$ coincides with the cluster Poisson coordinates $X_{i,j}$ on $\mathscr{X}_{a,n}^{\uf}$. For a frozen vertex $i$, one gets
\begin{equation}\label{frozen X}
\big(p\circ \tilde{\psi}\big)^* (X_i )= \begin{cases}\lambda_{1-a}\dfrac{\Delta_{\{3-a,\dots, 2\}}\Delta_{\{1-a,\dots, n\}}}{\Delta_{\{2-a,\dots, 1\}}\Delta_{\{2,2-a,\dots, n\}}} & \text{if $i=1$},\vspace{1mm}\\
\lambda_{i-a}\dfrac{\Delta_{\{i-a+2,\dots, i+1\}}\Delta_{\{2,\dots, i, i-a,\dots, n\}}}{\Delta_{\{i-a+1,\dots, i\}}\Delta_{\{2, \dots, i+1, i-a+1,\dots, n\}}} & \text{if $1< i<a$},\vspace{1mm}\\
\lambda_{i-a}\dfrac{\Delta_{\{i-a,\dots, i-1\}}\Delta_{\{i-a+2,\dots, i, n\}}}{\Delta_{\{i-a+1,\dots, i\}}\Delta_{\{i-a+1,\dots, i-1,n\}}} & \text{if $a\leq i\leq n-1$},\vspace{1mm}\\
\lambda_b \dfrac{\Delta_{\{b,b+2,\dots, n\}}}{\Delta_{\{b+1,\dots, n\}}} & \text{if $i=n$},
\end{cases}
\end{equation}
from which one can verify that $X_i$ are independent of the choices of $v_i\in l_i$.

\begin{prop} The map $\psi\colon \dconf_n^\times(a)\dashrightarrow \mathscr{X}_{a,n}$ is a birational equivalence.
\end{prop}
\begin{proof} Let $U$ be an open subset of $\dconf_n^\times(a)$ consisting of $[\phi_1, l_1, \dots, \phi_n, l_n]$ such that
\begin{equation}\label{bullet condition}
 \text{every collection $\{l_{i_1}, \dots, l_{i_a}\}$ of $a$-many lines is linearly independent.}
\end{equation}
Note that any Pl\"ucker coordinate on the lift of $U$ to $\widetilde{\dconf^\times_n(a)} \cong \dGr_a^\times(n)\times \mathrm{H}$ is nonzero.

Let $\mathcal{T}_{Q_{a,n}}\subset \mathscr{X}_{a,n}$ be the algebraic torus corresponding to the cluster chart associated to $Q_{a,n}$. By definition, the map $\psi$ restricted to $U$ is a regular map to $\mathcal{T}_{Q_{a,n}}$.
Let $ (X_f)_{f\in I}$ be a generic point in $\mathcal{T}_{Q_{a,n}}$. It suffices to show that it has a unique pre-image in $U$.

Recall that the map $\psi\colon \conf_n^\times(a)\dashrightarrow \mathscr{X}_{a,n}^{\uf}$ is birational. After imposing the genericity condition, by using the unfrozen part $ (X_f )_{f\in I^{\uf}}$, one can uniquely reconstruct a configuration of lines $[l_1,\dots, l_n]$ satisfying the condition \eqref{bullet condition}. Take a representative $v_i\in l_i$ for each $i$. One can use the frozen part $(X_i)_{i=1}^n$ to uniquely reconstruct the isomorphisms $\phi_i\colon l_i\rightarrow l_{i-1}$ by deducing the scaling factors $\lambda_i$ from the frozen variables $X_i$ and the non-zero Pl\"{u}cker coordinates using~\eqref{frozen X}. It is easy to see that the resulting isomorphisms~$\phi_i$ are independent of~$v_i$ chosen. Therefore we obtain a unique configuration in~$U$.
\end{proof}

\begin{thm} \label{X isomorphism}\label{2018.3.2.15.06}
The birational equivalence $\psi\colon \dconf_n^\times(a)\dashrightarrow \mathscr{X}_{a,n}$ induces an algebra isomorphism between $\up (\mathscr{X}_{a,n} ):=\mathcal{O} (\mathscr{X}_{a,n} )$ and $\mathcal{O}\big(\dconf_n^\times(a)\big)$. Combining this result with Proposition~{\rm \ref{affine}} we deduce that $\dconf^\times_n(a)\cong \mathrm{Spec} (\up (\mathscr{X}_{a,n} ) )$.
\end{thm}
\begin{proof}
Note that we have the following commutative diagram
\begin{equation}\label{diagram tilde}
\vcenter{\vbox{\xymatrix{\widetilde{ \dconf_n^\times(a)} \ar@{-->}[r]^{\tilde{\psi}} \ar[d]_p & \widetilde{\mathscr{A}_{a,n}}\ar[d]^p \\
\dconf_n^\times(a) \ar@{-->}[r]_\psi & \mathscr{X}_{a,n}.}}}
\end{equation}
Since the map $\psi$ is birational, its pull-back map $\psi^*$ is an isomorphism between fields of rational functions on $\dconf_n^\times(a)$ and $\mathscr{X}_{a,n}$. Let $F$ be a rational function on $\mathscr{X}_{a,n}$. It suffices to show that
\begin{equation}\label{claim}
\mbox{$F$ is regular on ${\mathscr{X}_{a,n}}$} \quad \Longleftrightarrow \quad \mbox{$\psi^*(F)$ is regular on $\dconf_n^\times(a)$.}
\end{equation}

Let us use the other three maps in the commutative diagram \eqref{diagram tilde} to prove the above statement. First by applying Lemma \ref{tech.lem.545} to the vertical map $p$ on the right, we know that
\[
\mbox{$F$ is regular on ${\mathscr{X}_{a,n}}$} \quad \Longleftrightarrow \quad \mbox{$p^*(F)$ is regular on $\widetilde{\mathscr{A}_{a,n}}$.}
\]
On the other hand, note that the vertical map $p$ on the left is a surjective morphism onto a~smooth affine variety. By applying Lemma \ref{surjective lift} below we can deduce that
\[
\mbox{$\varphi^*(F)$ is regular on $\dconf_n^\times(a)$} \quad \Longleftrightarrow \quad \mbox{$p^*\circ \psi^*(F)$ is regular on $\widetilde{\dconf_n^\times(a)}$.}
\]
By the commutativity of diagram \eqref{diagram tilde}, we know that $p^*\circ\psi^*(F)=\tilde{\psi}^*\circ p^*(F)$. Therefore the claim \eqref{claim} reduces to
\[
\mbox{$p^*(F)$ is regular on $\widetilde{\mathscr{A}_{a,n}}$} \quad \Longleftrightarrow \quad \mbox{$\tilde{\psi}^*\circ p^*(F)$ is regular on $\widetilde{\dconf_n^\times(a)}$,}
\]
which is true by Corollary~\ref{A isomorphism}.
\end{proof}

\begin{lem}\label{surjective lift} Let $f\colon \spec B\rightarrow \spec A$ be a surjective morphism where~$A$ is a normal Noetherian domain. Let $\Fr(A)$ be the field of fractions of~$A$. Then for any $F\in \Fr(A)$, we have $F\in A$ if and only if $f^*(F)\in B$.
\end{lem}
\begin{proof} The implication $F\in A\implies f^*(F)\in B$ is trivial.

Conversely, suppose $f^*(F)\in B$ but $F\notin A$. Since $A$ is a normal Noetherian domain, a~standard result of commutative algebra (e.g., \cite[Corollary~11.4]{Eis}) states that
\[
A=\bigcap_{\codim \mathfrak{p}=1}A_\mathfrak{p}.
\]
If $F\notin A$, then there exists a prime ideal $\mathfrak{p}$ of codimension 1 such that $F\notin A_\mathfrak{p}$. Since $A_\mathfrak{p}$ is a~discrete valuation ring, $F\notin A_\mathfrak{p}$ implies that $\frac{1}{F}\in \mathfrak{p}A_\mathfrak{p}$. Let $\mathfrak{q}\in {\rm Spec}~B$ be a~preimage of $\mathfrak{p}$ under the morphism $f$. Here $\mathfrak{q}$ is a prime ideal in~$B$, and $f^\ast$ is local homomorphism from~$A_\mathfrak{p}$ to~$B_{\mathfrak{q}}$. Hence, $f^*\big(\frac{1}{F}\big)\in \mathfrak{q}B_\mathfrak{q}$. Combining this with the assumption that $f^*(F)\in B$ we deduce that
\[
1=f^*(F)f^*\left(\frac{1}{F}\right)\in \mathfrak{q}B_\mathfrak{q},
\]
which is absurd. Therefore $F$ must be an element of $A$.
\end{proof}

Let us end the section with a disclaimer of notation convention for the rest of the paper.

\begin{conv} Let us summarize the following algebra isomorphisms:
\begin{gather*}
\text{Theorem \ref{Scott}}\colon \ \mathcal{O}\big(\dGr_a^\times(n)\big)\cong \up\left(\mathscr{A}_{a,n}\right),
\\
\text{Corollary \ref{A isomorphism}} \colon \ \mathcal{O}\big(\widetilde{\dconf_n^\times(a)}\big)\cong \up \big(\widetilde{\mathscr{A}_{a,n}} \big),
\\
\text{Theorem \ref{X isomorphism}}\colon \ \mathcal{O}\big(\dconf_n^\times(a)\big)\cong \up (\mathscr{X}_{a,n} ).
\end{gather*}
From now on, we will treat each pair as \textit{identical algebras}, and treat their corresponding fields of rational functions as \textit{identical fields} as well. In other words, we will drop the pull-back maps $\psi^*$ and $\tilde{\psi}^*$ to simplify our notations. We view $ \{A_f \}$ as regular functions on $\dGr_a^\times(n)$ and view $ \{X_f \}$ as rational functions on $\dconf_n^\times(a)$. In particular, using the cluster Poisson coordinates $ \{X_f \}$ we also get a log-canonical Poisson structure (see~\cite{FGensemble} for more details) on $\dconf_n^\times(a)$ where
\begin{equation}\label{dconf poisson}
 \{X_f,X_g \}=2 \varepsilon_{fg}X_fX_g.
\end{equation}

Recall from the proof of Corollary~\ref{A isomorphism} that $\widetilde{\dconf_n^\times (a)}\cong \dGr_a^\times(n)\times \mathrm{H}$. We can pull back any Pl\"{u}cker coordinate $\Delta_I$ on $\dGr_a^\times(n)$ to a regular function on $\widetilde{\dconf_n^\times(a)}$, which we will also denote by $\Delta_I$ as well.
\end{conv}

\section{Cluster nature of decorated configuration space}

This section is devoted to studying the cluster nature of the following three maps on $\dconf_{n}^\times(a)$: the twisted monodromy $P$, the cyclic rotation $R$, and the potential function $\mathcal{W}$, with the goal to express them in terms of the cluster coordinates $\{X_f\}_{f\in I}$ associated to the quiver $Q_{a,n}$.

\subsection{Twisted monodromy and Casimir}\label{mon.pot}

\begin{prop}\label{product} The twisted monodromy
\[
P=\prod_{f\in I} X_f.
\]
The function $P$ is a \emph{Casimir element} with respect to the Poisson bracket~\eqref{dconf poisson}, that is, $\{P,F\}\allowbreak =0$ for any rational function $F$ on $\dconf_n^\times(a)$.
\end{prop}

\begin{proof} Let $[\phi_1, l_1, \dots, \phi_n, l_n]\in \dconf_n^\times(a)$.
Let us lift it against the projection $p\colon \widetilde{\dconf_n^\times(a)}\twoheadrightarrow \dconf_n^\times(a)$ to a point $[\phi_1, v_1, \dots, \phi_n, v_n]$. By~\eqref{def poisson cluster}, we get
\begin{equation}
\label{look22.1246}
p^*\left(\prod_{f\in I}X_f\right) = \prod_{(f, g)\in I \times \tilde{I}} A_{g}^{\tilde{\varepsilon}_{gf}} = \prod_{g\in \tilde{I}} A_{g}^{\sum\limits_{f\in I} \tilde{\varepsilon}_{gf}}.
\end{equation}
Note that the quiver $Q_{a,n}$ is made of cycles. Therefore for every $f\in I$, one has \begin{equation}
\label{look.1217}
 \sum_{g\in I} \varepsilon_{gf} =0.
\end{equation}
Hence the only factors contributing to the product~\eqref{look22.1246} are from the extra frozen vertices that extend $Q_{a,n}$ to $\widetilde{Q_{a,n}}$. Therefore
\[
p^*\left(\prod_{f\in I}X_f\right)=\prod_{i=1}^n A_{i'}=\prod_{i=1}^n\frac{ \lambda_{i-a} \Delta_{\{i-a,i-a+1,\dots, i-1\}}}{\Delta_{\{i-a+1,i-a+2,\dots, i\}}}= \prod_{i=1}^n \lambda_i.
\]
By comparing \eqref{rescalefa} with the definition of~$P$ in~\eqref{def of P}, we see that $\prod\limits_{i=1}^n\lambda_i=p^*(P)$. But $p^*\Big(\prod\limits_{f\in I}X_f\Big)=p^*(P)$ implies $\prod\limits_{f\in I}X_f=P$ automatically since $p$ is a projection.

The statement that $P$ is Casimir is a direct consequence of~\eqref{look.1217} and~\eqref{dconf poisson}.
\end{proof}

\subsection{Cyclic rotation and cluster transformation}\label{cyc.rot}

Recall the twisted cyclic rotation $C_a$ on $\dGr_a^{\times}(n)$.
Gekhtman, Shapiro, and Vainshtein proved that $C_a$ is a cluster $K_2$ automorphism that can be realized by a mutation sequence~$\rho$ (see \cite[p.~90]{GSVcp}). In this section, we briefly recall the definition of~$\rho$.
We show that the cyclic rotation~$R$ on $\dconf_n^\times(a)$ is a cluster Poisson automorphism realized by the same mutation sequence~$\rho$.

The vertices of the quiver $Q_{a,n}$ are indexed by $(i,j)\in I$. Let $\rho$ be a mutation sequence hitting in the unfrozen part of $Q_{a, n}$ along every column from bottom to top, starting at the leftmost unfrozen column and going all the way to the rightmost unfrozen column, that is,
\begin{equation}\label{rho}
\rho:=\sigma_{b-1} \circ \cdots \circ \sigma_2\circ \sigma_{1}, \qquad \mbox{where} \quad \sigma_i:=\mu_{(1,i)}\circ \mu_{(2,i)}\circ \dots \circ \mu_{(a-1,i)}.
\end{equation}
 Equivalently, the sequence $\rho$ may be realized by a sequence of 2-by-2 moves\footnote{An explicit realization of $\rho$ as a sequence of 2-by-2 moves has been included in the appendix.} on the reduced plabic graph $\Gamma_{a,n}$. Its resulting reduced plabic graph $\Gamma_{a,n}'$ is identical to $\Gamma_{a,n}$ with the non-boundary faces remaining in the same places and with all the boundary faces rotated to the neighboring one in the clockwise direction.
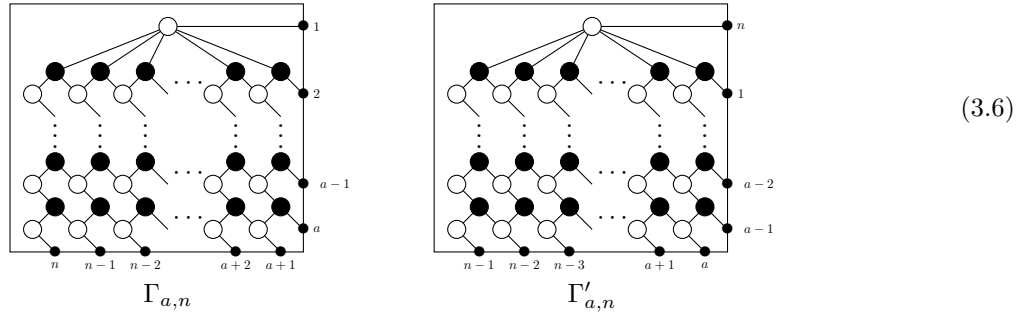
\begin{figure}[h!]\centering 
\begin{tikzpicture}[scale=0.6, baseline=10ex]
\foreach \i in {0,1,2,4,5}
 {
 \draw (\i+0.5,3.5) -- (3,4.5);
 }
\draw (3,4.5) -- (6,4.5);
 \draw [fill=white] (3,4.5) circle [radius=0.2];
\foreach \i in {0,1,2,4,5}
 {
 \draw (\i+1,3) -- (\i+0.5,3.5) -- (\i,3) -- (\i+0.5,2.5);
 \draw [fill=white] (\i,3) circle [radius=0.2];
 \draw [fill] (\i+0.5,3.5) circle [radius=0.2];
 }
\node at (3.5,3.25) [] {$\cdots$};
\foreach \i in {0,1,2,4,5}
 {
 \node at (\i+0.5,2.25) [] {$\vdots$};
 }
\foreach \i in {0,1,2,4,5}
 {
 \draw (\i+1,1) -- (\i+0.5,1.5) -- (\i,1) -- (\i+0.5,0.5);
 \draw [fill=white] (\i,1) circle [radius=0.2];
 \draw [fill] (\i+0.5,1.5) circle [radius=0.2];
 }
\node at (3.5,1.25) [] {$\cdots$};
\foreach \i in {0,1,2,4,5}
 {
 \draw (\i+1,0) -- (\i+0.5,0.5) -- (\i,0) -- (\i+0.5,-0.5);
 \draw [fill=white] (\i,0) circle [radius=0.2];
 \draw [fill] (\i+0.5,0.5) circle [radius=0.2];
 }
\node at (3.5,0.25) [] {$\cdots$};
\draw (-0.5,-0.5) rectangle (6,5);
\node at (6,4.5) [] {$\bullet$};
\node [scale=0.5] at (6.3,4.5) [] {$1$};
\node at (6,3) [] {$\bullet$};
\node [scale=0.5] at (6.3,3) [] {$2$};
\node at (6,1) [] {$\bullet$};
\node [scale=0.5] at (6.7,1) [] {$a-1$};
\node at (6,0) [] {$\bullet$};
\node [scale=0.5] at (6.3,0) [] {$a$};
\node at (5.5,-0.5) [] {$\bullet$};
\node [scale=0.5] at (5.5,-0.8) [] {$a+1$};
\node at (4.5,-0.5) [] {$\bullet$};
\node [scale=0.5] at (4.5,-0.8) [] {$a+2$};
\node at (2.5,-0.5) [] {$\bullet$};
\node [scale=0.5] at (2.5,-0.8) [] {$n-2$};
\node at (1.5,-0.5) [] {$\bullet$};
\node [scale=0.5] at (1.5,-0.8) [] {$n-1$};
\node at (0.5,-0.5) [] {$\bullet$};
\node [scale=0.5] at (0.5,-0.8) [] {$n$};
\node at (3,-1.5) [] {$\Gamma_{a,n}$};
\end{tikzpicture}
 \quad \quad \quad
\begin{tikzpicture}[scale=0.6, baseline=10ex]
\foreach \i in {0,1,2,4,5}
 {
 \draw (\i+0.5,3.5) -- (3,4.5);
 }
 \draw (3,4.5) -- (6,4.5);
 \draw [fill=white] (3,4.5) circle [radius=0.2];
\foreach \i in {0,1,2,4,5}
 {
 \draw (\i+1,3) -- (\i+0.5,3.5) -- (\i,3) -- (\i+0.5,2.5);
 \draw [fill=white] (\i,3) circle [radius=0.2];
 \draw [fill] (\i+0.5,3.5) circle [radius=0.2];
 }
\node at (3.5,3.25) [] {$\cdots$};
\foreach \i in {0,1,2,4,5}
 {
 \node at (\i+0.5,2.25) [] {$\vdots$};
 }
\foreach \i in {0,1,2,4,5}
 {
 \draw (\i+1,1) -- (\i+0.5,1.5) -- (\i,1) -- (\i+0.5,0.5);
 \draw [fill=white] (\i,1) circle [radius=0.2];
 \draw [fill] (\i+0.5,1.5) circle [radius=0.2];
 }
\node at (3.5,1.25) [] {$\cdots$};
\foreach \i in {0,1,2,4,5}
 {
 \draw (\i+1,0) -- (\i+0.5,0.5) -- (\i,0) -- (\i+0.5,-0.5);
 \draw [fill=white] (\i,0) circle [radius=0.2];
 \draw [fill] (\i+0.5,0.5) circle [radius=0.2];
 }
\node at (3.5,0.25) [] {$\cdots$};
\draw (-0.5,-0.5) rectangle (6,5);
\node at (6,4.5) [] {$\bullet$};
\node [scale=0.5] at (6.3,4.5) [] {$n$};
\node at (6,3) [] {$\bullet$};
\node [scale=0.5] at (6.3,3) [] {$1$};
\node at (6,1) [] {$\bullet$};
\node [scale=0.5] at (6.7,1) [] {$a-2$};
\node at (6,0) [] {$\bullet$};
\node [scale=0.5] at (6.7,0) [] {$a-1$};
\node at (5.5,-0.5) [] {$\bullet$};
\node [scale=0.5] at (5.5,-0.8) [] {$a$};
\node at (4.5,-0.5) [] {$\bullet$};
\node [scale=0.5] at (4.5,-0.8) [] {$a+1$};
\node at (2.5,-0.5) [] {$\bullet$};
\node [scale=0.5] at (2.5,-0.8) [] {$n-3$};
\node at (1.5,-0.5) [] {$\bullet$};
\node [scale=0.5] at (1.5,-0.8) [] {$n-2$};
\node at (0.5,-0.5) [] {$\bullet$};
\node [scale=0.5] at (0.5,-0.8) [] {$n-1$};
\node at (3,-1.5) [] {$\Gamma'_{a,n}$};
\end{tikzpicture}
\end{figure}

One advantage of using reduced plabic graphs is that one may make use of zig-zag strands to assign an $a$-element subset of $\{1,\dots, n\}$ to each face of the graph (see Definition~\ref{dominating set}). For example, the $a$-element subsets assigned to faces of $\Gamma_{a,n}$ in \eqref{I} arise precisely in this way, which in turn associates Pl\"ucker coordinates to these faces. Moreover, it is not hard to see that a 2-by-2 move combined with the cluster $K_2$ mutation formula yields precisely a Pl\"{u}cker relation (see~\eqref{plucker609} for more details); hence we can conclude that the Pl\"{u}cker coordinates on any bipartite graph obtained from $\Gamma_{a,n}$ via a sequence of 2-by-2 moves can be computed by using zig-zag strands as well.

Let $\{A'_f\}$ be the $K_2$ cluster associated to $\Gamma'_{a,n}$ after applying the mutation sequence $\rho$.
According to the above discussion, the cluster $\{A'_f\}$ is defined by the $a$-element sets~\eqref{I} assigned to faces of the minimal bipartitie graphs $\Gamma'_{a,n}$. Using the zig-zag strands on $\Gamma'_{a,n}$, we find that
\begin{equation}\label{I'}
 I'(i,j) =\{\underbrace{b-j,\dots, b-j+i-1}_\text{consecutive $i$ indices}, \underbrace{b+i, \dots, n-1}_\text{consecutive $a-i$ indices}\},
\end{equation}
which yields
\[
A'_{i,j}=\Delta_{I'(i,j)}=C_a^* (A_{i,j} ).
\]
In other words, the twisted cyclic rotation $C_a$ is equal to the cluster transformation defined by the mutation sequence $\rho$.

If we apply the mutation sequence $\rho$ to the extended quiver $\widetilde{{Q}_{a,n}}$, it is easy to see that the resulting quiver $\rho\widetilde{Q_{a,n}}$ is again the same as $\widetilde{{Q}_{a,n}}$ up to rotations of frozen vertices. For example, if we start with $\widetilde{Q_{3,6}}$ as in~\eqref{expand quiver example}, then $\rho\widetilde{Q_{3,6}}$ is as follows:
\[
\begin{tikzpicture}[scale=0.7]
\node (-1-0) at (-2,5) [blue,label=left:{$6'$}] {$\square$};
\node (0-0) at (0,5) [red, label=right:{5}] {$\circ$};
\foreach \i in {1,2}
{
\foreach \j in {1,2}
{
\node (\i-\j) at (\j,5-\i) [] {$\bullet$};
}
}
\node (1-4) at (5,4) [blue, label=right:{$1'$}] {$\square$};
\node (2-4) at (5,3) [blue, label=right:{$2'$}] {$\square$};
\node (1-3) at (3,4) [red, label=right:{6}] {$\circ$};
\node (2-3) at (3,3) [red, label=right:{1}] {$\circ$};
\node (3-1) at (1, 2) [red, label=below:{4}] {$\circ$};
\node (3-2) at (2, 2) [red, label=below:{3}] {$\circ$};
\node (4-1) at (1, 0) [blue, label=below:{$5'$}] {$\square$};
\node (4-2) at (2, 0) [blue, label=below:{$4'$}] {$\square$};
\node (3-3) at (3,2) [red, label=below:{2}] {$\circ$};
\node (4-4) at (3,0) [blue, label=below:{$3'$}] {$\square$};
\foreach \i in {1,2}
{
\foreach \j in {2,3}
 {
 \pgfmathtruncatemacro{\k}{\j-1};
 \draw [<-] (\i-\j) -- (\i-\k);
 }
}
\foreach \j in {1,2,3}
{
\draw [<-] (2-\j) -- (1-\j);
\draw [<-] (3-\j) -- (2-\j);
}
\foreach \i in {1,2}
 {
 \foreach \j in {1,2}
 {
 \pgfmathtruncatemacro{\k}{\i+1};
 \pgfmathtruncatemacro{\l}{\j+1};
 \draw [<-] (\i-\j) -- (\k-\l);
 }
 }
\draw [<-] (1-1) -- (0-0);
\draw [<-] (0-0) to (3-1);
\draw [->] (1-3) to [bend right] (-1-0);
\draw [->] (0-0) to [bend right] (4-1);
\draw [->] (3-1) -- (4-2);
\draw [->] (3-2) -- (4-4);
\draw [->] (3-3) -- (2-4);
\draw [->] (2-3) -- (1-4);
\end{tikzpicture}
\]
In particular, since the extra frozen vertex $i'$ is only connected to the frozen vertex $i$, any cluster mutation does not change such connectivity. Therefore after the mutation sequence $\rho$, the frozen vertex $i'$ remains connected to $i$. Recall from the commutative diagram~\eqref{ext a to a} that $\widetilde{\dconf_n^\times(a)}$ fibers over $\dconf_n^\times(a)$ the same way $\widetilde{\mathscr{A}_{a,n}}$ fibers over $\mathscr{A}_{a,n}$, with the data of the linear isomorphisms $\phi_i$ captured by the extra frozen cluster~$K_2$ coordinates~$A_{i'}$. By an additional checking on the pull-backs of the frozen cluster $K_2$ coordiantes $A_{i'}$ we see that the twisted cyclic rotation $\tilde{C}_a\colon  [\phi_1,v_1,\dots, \phi_n,v_n ]\mapsto \big[\phi_n,(-1)^{a-1}v_n,\phi_1,v_1,\dots, \phi_{n-1},v_{n-1}\big]$ on $\widetilde{\dconf_n^\times(a)}$ is equal to the cluster transformation defined by the mutation sequence $\rho$. Pushing the twisted cyclic rotation $\tilde{C}_a$ down via the projection $\widetilde{\dconf_n^\times(a)}\twoheadrightarrow \dconf_n^\times(a)$ and using the commutative diagram~\eqref{diagram tilde}, we see that the cyclic rotation~$R$ on~$\dconf_n^\times(a)$ is also equal to the cluster transformation defined by the mutation sequence $\rho$ with
\[
X'_g=R^* (X_g ).
\]
In particular, since the cyclic rotation $R$ is a cluster transformation, it preserves the Poisson structure on $\dconf_n^\times(a)$.

\subsection{Potential function} \label{3.3}

\begin{prop} In terms of the Poisson cluster $\{X_{i,j}\}$ associated to $Q_{a,n}$, the theta functions~$\vartheta_i$ in~\eqref{theta.f} are
\begin{gather}
\label{3.14..1214}
\vartheta_n= X_{0,0}, \qquad \vartheta_a=X_{a,b},
\\
\label{2018.2.11.21.36}
\vartheta_i=\sum_{j=1}^b X_{i,b}X_{i,b-1}\cdots X_{i,j}\qquad \text{for $0<i<a$},
\\
\label{2018.2.11.21.35}
\vartheta_i=\sum_{j=1}^a X_{a,n-i}X_{a-1,n-i}\cdots X_{j,n-i}\qquad \text{for $a<i<n$},
\end{gather}
\end{prop}
\begin{rmk}
For \eqref{2018.2.11.21.36}, the terms in $\vartheta_i$ are in bijection with rectangles of all possible lengths across the $i$th row of the quiver $Q_{a,n}$ that ends at the vertex~$(i,b)$.
For~\eqref{2018.2.11.21.35}, the terms in $\vartheta_{i}$ are in bijection with rectangles of all possible heights across the $(n-i)$th column of $Q_{a_n}$ that rises from the vertex $(a,n-i)$. For instance, the formulas of $\vartheta_2$ and $\vartheta_5$ with $a=3$ and $n=7$ are as follows:
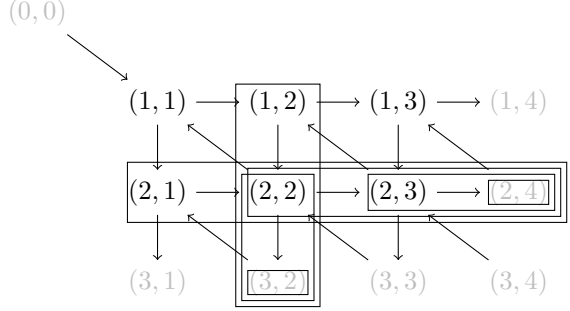
\[
\begin{tikzpicture}[scale=0.8]
\foreach \i in {1,2,3}
 {
 \foreach \j in {1,2}
 {
 \node (\j-\i) at (2*\i, 4.5-1.5*\j) [] {$(\j,\i)$};
 }
 }
\node (1-4) [lightgray] at (8,3) [] {$(1,4)$};
\node (2-4) [lightgray] at (8,1.5) [] {$(2,4)$};
\foreach \i in {1,...,4}
 {
 \node (3-\i) [lightgray] at (2*\i, 0) [] {$(3,\i)$};
 }
\node (0-0) [lightgray] at (0,4.5) [] {$(0,0)$};
\foreach \i in {1,2,3}
 {
 \pgfmathtruncatemacro{\k}{\i+1};
 \foreach \j in {1,2}
 {
 \pgfmathtruncatemacro{\l}{\j+1};
 \draw [->] (\j-\i) -- (\j-\k);
 \draw [->] (\j-\i) -- (\l-\i);
 \draw [->] (\l-\k) -- (\j-\i);
 }
 }
\draw [->] (0-0) -- (1-1);
\draw (7.5,1.3) rectangle (8.5,1.7);
\draw (5.5,1.2) rectangle (8.6,1.8);
\draw (3.5,1.1) rectangle (8.7,1.9);
\draw (1.5,1) rectangle (8.8,2);
\draw (3.5,-0.2) rectangle (4.5,0.2);
\draw (3.4,-0.3) rectangle (4.6,1.8);
\draw (3.3,-0.4) rectangle (4.7,3.3);
\end{tikzpicture}
\]
\begin{gather*}
\vartheta_2=  X_{2,4}+X_{2,3}X_{2,4}+X_{2,2}X_{2,3}X_{2,4}+X_{2,1}X_{2,2}X_{2,3}X_{2,4},\\
\vartheta_5= \vartheta_{7-2}=X_{3,2}+X_{2,2}X_{3,2}+X_{1,2}X_{2,2}X_{3,2}.
\end{gather*}
Therefore the \emph{potential function} $\mathcal{W}$ in \eqref{potential.W} can be expressed as
\begin{equation}\label{W}
\mathcal{W}=X_{0,0}+X_{a,b}+ \sum_{i=1}^{a-1}\sum_{j=1}^b X_{i,b}X_{i,b-1}\cdots X_{i,j} +\sum_{j=1}^{b-1} \sum_{i=1}^a X_{a,j} X_{a-1,j}\cdots X_{i,j}.
\end{equation}
\end{rmk}

\begin{proof} We again lift the computation up to $\widetilde{\dconf_n^\times(a)}$. By the definition of $\vartheta_n$ we have $\phi(v_{b+1})-p^*(\vartheta_n) v_{b+1}\in {\rm Span}(v_{b+2},\dots,v_n)$. Hence
\begin{align*}0= \Delta_{(\phi(v_{b+1})-p^*(\vartheta_n) v_{b+1})\wedge v_{b+2}\wedge \dots\wedge v_n}=\lambda_{b}\Delta_{\{b,b+2, \dots, n\}}-p^*(\vartheta_n) \Delta_{\{b+1,\dots, n\}}.
\end{align*}
Therefore
\begin{equation}\label{helop}
p^*(\vartheta_n) = \lambda_{b}\frac{\Delta_{\{b,b+2,\dots, n\}}}{\Delta_{\{b+1,\dots, n\}}} = p^*(X_{0, 0}),
\end{equation}
where the second equality follows from \eqref{frozen X}.
The above equality then implies $\vartheta_n=X_{0,0}$ due to the surjectivity of $p$. Similarly we prove that $\vartheta_a = X_{a,b}$.

Now let us prove \eqref{2018.2.11.21.35}. The proof of \eqref{2018.2.11.21.36} goes along the same line.

For $a<i<n$, let $k=n-i$. Then $0<k<b=n-a$. For $1\leq j\leq a$, define the $a$-element set
\[
J(j, k):= \{b-k,  \underbrace{b-k+2, \dots,  b-k+j}_\text{consecutive $j-1$ indices}, \underbrace{b+j+1, \dots,  n}_\text{consecutive $a-j$ indices}\}.
\]
Recall the definition of $a$-element sets $I(j,k)$ from \eqref{I}.
One has the Pl\"ucker relation \eqref{Plucker relation}:
\begin{equation*}
\Delta_{I(j, k-1)}\Delta_{I(j+1, k+1)}+ \Delta_{I(j+1, k)}\Delta_{J(j,k)} =\Delta_{I(j,k)}\Delta_{J(j+1,k)}.
\end{equation*}
Dividing by $\Delta_{I(j, k-1)}\Delta_{I(j+1, k+1)}$ on both sides, we get
\begin{equation*}
1+\frac{\Delta_{I(j+1, k)}\Delta_{J(j,k)}}{\Delta_{I(j, k-1)}\Delta_{I(j+1, k+1)}}=\frac{\Delta_{I(j,k)}\Delta_{J(j+1,k)}}{\Delta_{I(j, k-1)}\Delta_{I(j+1, k+1)}}.
\end{equation*}
Let us set
\[
Y_{j,k}:=1+X_{j,k} (1+X_{j-1,k} (\dots  (1+X_{1,k} )\dots ) ).
\]
Let us fix $k$. We will prove by induction on $j$ that
\begin{equation}\label{induc.for}
p^*(Y_{j,k})=\frac{\Delta_{I(j,k)}\Delta_{J(j+1,k)}}{\Delta_{I(j,k-1)}\Delta_{I(j+1,k+1)}}.
\end{equation}
Note that $I(1,k+1)=J(1,k)$. Therefore for $j=1$, by using \eqref{def poisson cluster} for $p^*$ and identifying $A_f=\Delta_{I(f)}$ we have
\begin{align*}
 p^* (Y_{1,k} )= 1+p^* (X_{1,k} )= 1+\frac{\Delta_{I(1,k+1)}\Delta_{I(2,k)}}{\Delta_{I(1,k-1)}\Delta_{I(2,k+1)}}
 = \frac{\Delta_{I(1,k)}\Delta_{J(2,k)}}{\Delta_{I(1,k-1)}\Delta_{I(2,k+1)}}.
\end{align*}
If $1<j<a$ and $\eqref{induc.for}$ is true for $j-1$, then it is true for $j$ because
\[
 p^*(Y_{j,k})=1+ p^* (X_{j,k}Y_{j-1,k} )=
 1+\frac{\Delta_{I(j+1,k)}\Delta_{J(j,k)}}{\Delta_{I(j,k-1)}\Delta_{I(j+1,k+1)}} =\frac{\Delta_{I(j,k)}\Delta_{J(j+1,k)}}{\Delta_{I(j,k-1)}\Delta_{I(j+1,k+1)}},
\]
and the induction is finished.

Recall from \eqref{frozen X} that
\[
p^* (X_{a,k} ) = p^* (X_{n-k} ) = \lambda_{b-k}\frac{\Delta_{\{b-k,\dots, n-k-1\}}\Delta_{\{b-k+2,\dots, n-k,n\}}}{\Delta_{\{b-k+1,\dots, n-k\}}\Delta_{\{b-k+1,\dots, n-k-1,n\}}}.
\]
Now by setting $j=a-1$ in \eqref{induc.for} and recalling \eqref{I} for the definition of $I(j,k)$, we get
\begin{gather*}
  p^* (X_{a,k}Y_{a-1, k} )
 =  p^* (X_{a,k} )p^* (Y_{a-1,k} )\\
\qquad{} =  \lambda_{b-k}\frac{\Delta_{\{b-k,\dots, n-k-1\}}\Delta_{\{b-k+2,\dots, n-k,n\}}}{\Delta_{\{b-k+1,\dots, n-k\}}\Delta_{\{b-k+1,\dots, n-k-1,n\}}}\frac{\Delta_{\{b-k+1,\dots,n-k-1,n\}}\Delta_{\{b-k,b-k+2,\dots,n-k\}}}{\Delta_{\{b-k+2,\dots, n-k,n\}}\Delta_{\{b-k,\dots, n-k-1\}}} \\
\qquad{} = \lambda_{b-k}\frac{\Delta_{\{b-k,b-k+2,\dots,n-k\}}}{\Delta_{\{b-k+1,b-k+2,\dots, n-k\}}}
 =  p^* (\vartheta_{n-k} ),
\end{gather*}
where the last equality is obtained analogously to \eqref{helop}. This shows that
\[
\vartheta_{n-k}=X_{a,k}Y_{a-1, k}= \sum_{j=1}^a X_{a,k}X_{a-1,k}\cdots X_{j,k}. \tag*{\qed}
\]
\renewcommand{\qed}{}
\end{proof}

\subsection{Tropicalization}

Cluster varieties are examples of \emph{positive spaces}, which admit canonical tropicalizations. The sets of integral tropical points of cluster varieties play important roles in the Fock--Goncharov cluster duality. We briefly review the related constructions in this subsection; please see~\cite{FGensemble} for more details.

Let $\mathscr{X}$ be a cluster variety (either $K_2$ type or Poisson type). A {\it positive rational function} on~$\mathscr{X}$ is a nonzero function that can be expressed as a ratio of two polynomials with non-negative integer coefficients in one (and hence all) cluster(s) of~$\mathscr{X}$. The set ${P}(\mathscr{X})$ of positive rational functions is a {\it semifield}, i.e., a set closed under addition, multiplication, and division. The pair $(\mathscr{X}, P(\mathscr{X}))$ is called a positive space.

Let $\mathbb{Z}^t:=\left(\mathbb{Z},\min, +\right)$ be the semifield of tropical integers.\footnote{There is another closely-related semifield $\mathbb{Z}^T:= (\mathbb{Z},\max ,+ )$. See \cite[p.~35]{GHKK} for its connection to $\mathbb{Z}^t$.}
We define the tropicalization of $(\mathscr{X}, P(\mathscr{X}))$ to be the set
\[
\mathscr{X}\big(\mathbb{Z}^t\big):=\Hom_\text{semifield} \big(P(\mathscr{X}), \mathbb{Z}^t \big).
\]

\looseness=-1 Let $ (\mathscr{X}, P(\mathscr{X}) )$ and $ (\mathscr{Y}, P(\mathscr{Y}) )$ be a pair of cluster varieties. A rational map $f\colon \mathscr{X}\dashrightarrow \mathscr{Y}$ is \emph{positive} if the pull back $f^*P(\mathscr{Y})$ is contained in $P(\mathscr{X})$. Every positive rational map $f$ admits a~tro\-picalization $f^t\colon \mathscr{X} (\mathbb{Z}^t )\rightarrow \mathscr{Y} \big(\mathbb{Z}^t \big)$ defined by precomposing with the pull-back map~$f^*$. In particular cluster mutations are subtraction-free. Therefore every cluster automorphism is positive and can be tropicalized. It induces a natural action of the cluster modular group~$\mathcal{G}$ on~$\mathscr{X}\big(\mathbb{Z}^t\big)$.

Let $\chi:=\left\{X_i\,|\, 1\leq i \leq m\right\}$ be a cluster coordinate chart of $\mathscr{X}$. Its tropicalization is a bijection
\[
\chi^t\colon \ \mathscr{X}\big(\mathbb{Z}^t\big)\stackrel{\sim}{\longrightarrow} \mathbb{Z}^m, \qquad l \longmapsto \left(l(X_1),l(X_2),\dots, l(X_m)\right):=(x_1, \dots, x_m).
\]
The $m$-tuple $(x_1, \dots, x_m)$ is called the \emph{tropical coordinates} on $\mathscr{X}\big(\mathbb{Z}^t\big)$ in terms of the cluster $\chi$.

Tautologically, every positive function $f\in P(\mathscr{X})$ gives rise to a $\mathbb{Z}$-valued function
\[
f^t\colon \ \mathscr{X}\big(\mathbb{Z}^t\big) \longrightarrow \mathbb{Z},\qquad l \longmapsto l(f),
\]
which can be expressed as a piecewise linear function in terms of the tropical coordinates $\left(x_1, \dots, x_m\right)$ by the following procedure:
\begin{enumerate}\setlength\itemsep{0pt}
 \item Change every addition to taking minimum.
 \item Change every multiplication into addition and every division to subtraction.
 \item Drop the coefficient of each monomial in the function.
 \item Change the variables $X_i$ to $x_i$.
\end{enumerate}
For example, the tropicalization of the positive rational function $f=\frac{2X_1X_2+X_2^2+1}{(X_1+X_2)^3}$ is
\[
f^t=\min\{x_1+x_2,2x_2,0\}-3\min\{x_1,x_2\}.
\]

Let us tropicalize the cluster variety $\dconf^\times_n(a)$. Recall that the potential $\mathcal{W}$ and the twisted monodromy $P$ are positive functions on $\dconf^\times_n(a)$.
We can hence define subsets

\begin{equation}
\label{trop.conf}
Q(a,b,c):=\left\{ q \in \dconf^\times_n(a)\left(\mathbb{Z}^t\right) \, \Big|\, \begin{matrix} \mathcal{W}^t(q)\geq 0,\\
P^t(q)=c
\end{matrix}
\right\}.
\end{equation}

\begin{rmk} The part of $\dconf_n^\times(a)\big(\mathbb{Z}^t\big)$ that is cut out by the inequality $\mathcal{W}^t\geq 0$ is also known as the \emph{superpotential cone}, and as we will see soon, it coincides with the \emph{Gelfand--Zetlin cone} as well.
\end{rmk}

The rotation $R\colon \dconf^\times_n(a)\rightarrow \dconf^\times_n(a)$ is a cluster transformation and therefore can be tropicalized into $R^t\colon \dconf^\times_n(a)\big(\mathbb{Z}^t\big)\rightarrow \dconf^\times_n(a)\big(\mathbb{Z}^t\big)$. By definition~$\mathcal{W}$ and $P$ are invariant under~$R$, and hence $R^t$ preserves the sets $Q(a,b,c)$. We will introduce another coordinate system on $\dconf^\times_n(a)\big(\mathbb{Z}^t\big)$ in the next section and construct a natural bijection between $Q(a,b,c)$ and the set of plane partitions $P(a,b,c)$ defined in Definition~\ref{P(a,b,c)}.

\subsection{Gelfand--Zetlin coordinates}\label{GZCOO}
Although we can express the potential function $\mathcal{W}$ in terms of the cluster Poisson coordinates~$X_f$ in a subtraction-free manner, the resulting tropicalization does not convey its nice combinatorial connection to plane partitions very well. In this section, we will do a monomial transformation on the cluster Poisson coordinates to define a new set of coordinates, which we call ``Gelfand--Zetlin coordinates'' due to its combinatorial connection to Gelfand--Zetlin patterns, a well-studied object in combinatorial representation theory.

\begin{defn} We define the \emph{Gelfand--Zetlin coordinate} $L_{i,j}$ associated to the vertex $(i,j)$ of the quiver $Q_{a,n}$ to be
\[
L_{i,j}:=\prod_{\substack{k\geq i\\ l\geq j}} X_{k,l}.
\]
\end{defn}

By definition these new functions $ \{L_{i,j} \}$ relate to the cluster Poisson coordinates $ \{X_{i,j}\}$ by a~monomial transformation. With the convention that $L_{i,j}=1$ for $i>a$ or $j>b$, we can express the inverse of this transformation as
\[
X_{0,0}=\frac{L_{0,0}}{L_{1,1}} \qquad \text{and} \qquad X_{i,j}=\frac{L_{i,j}L_{i+1,j+1}}{L_{i+1,j}L_{i,j+1}} \qquad \text{for $(i,j)\neq(0,0)$}.
\]
It follows that these new functions $L_{i,j}$ form a coordinate system on $\dconf_n^\times(a)$, which justifies the term ``coordinate'' in the definition.

The following Lemma is a direct consequence of Proposition~\ref{product} and~\eqref{W}.
\begin{lem}
\label{2018.2.20.3.20}
In terms of the Gelfand--Zetlin coordinates, the twisted monodromy
$P=L_{0,0}$, and the potential
\begin{equation}
\label{2018.2.11.16.01}
\mathcal{W}=\frac{L_{0,0}}{L_{1,1}} + L_{a,b}+ \sum_{i=1}^{a-1}\sum_{j=1}^b\frac{L_{i,j}}{L_{i+1, j}}+ \sum_{j=1}^{b-1}\sum_{i=1}^a \frac{L_{i,j}}{L_{i,j+1}}.
\end{equation}
\end{lem}

\begin{rmk}
By setting $L_{0,0}={\rm e}^t$, \eqref{2018.2.11.16.01} coincides with the Lax operator on Grassmannian in {\rm \cite[equation~(B.25)]{EHX}}.
\end{rmk}

The Gelfand--Zetlin coordinates $L_{i,j}$ are clearly positive functions on $\dconf_n^{\times}(a)$. Let us follow our convention and use lower case letter $l_{i,j}$ to denote the tropicalization of $L_{i,j}$. The next proposition shows that plane partitions can be obtained naturally via tropicalization of the Gelfand--Zetlin coordinates.

\begin{prop}\label{PQ bijection} The tropicalized change of coordinates $(x_{i,j})\mapsto (l_{i,j})$ gives a natural bijection between the subset $Q(a,b,c)\subset \dconf^\times_a(n)\big(\mathbb{Z}^t\big)$ and the set of plane partitions $P(a,b,c)$ defined in Definition~{\rm \ref{P(a,b,c)}}.
\end{prop}
\begin{proof}By Lemma \ref{2018.2.20.3.20}, we get
\begin{gather*}
P^t = l_{0,0},\\
\mathcal{W}^t =\min\left( \{l_{0,0}-l_{1,1},  l_{a,b} \}\cup\left(\bigcup_{\substack{1\leq i\leq a-1 \\ 1\leq j\leq b}} \{l_{i,j}-l_{i+1,j} \}\right)\cup\left( \bigcup_{\substack{1\leq i\leq a\\ 1\leq j\leq b-1}} \{l_{i,j}-l_{i,j+1} \}\right)\right).
\end{gather*}
The condition that $\mathcal{W}^t\geq 0$ is now equivalent to the conditions
\begin{equation}\label{reason for gz coord}
l_{0,0}\geq l_{1,1}, \qquad l_{a, b}\geq 0, \qquad l_{i,j-1}\geq l_{i,j} \geq l_{i+1,j}, \qquad \forall\, (i,j).
\end{equation}
After imposing the additional condition that $P^t=l_{0,0}=c$, the coordinates $(l_{i,j})_{\substack{1\leq i\leq a\\ 1\leq j\leq b}}$ of the elements of $Q(a,b,c)$~\eqref{trop.conf} naturally become plane partitions in~$P(a,b,c)$.
\end{proof}

\begin{rmk} Note the similarity between the inequalities~\eqref{reason for gz coord} and the inequalities in Gelfand--Zetlin patterns. For this reason, we name the coordinates $L_{ij}$ as ``Gelfand--Zetlin'' coordinates.
\end{rmk}

The next lemma is useful for future computation with Gelfand--Zetlin coordinates.

\begin{lem}\label{L}Recall $\lambda_k$ in~\eqref{rescalefa} and $A_{i,j}$ in~\eqref{aij}. Set $A_{i,0}=A_{0,j}=A_{0,0}$.
Then
\begin{equation}
\label{Lexp}
p^*\left(L_{i,j}\right)=\frac{A_{i-1,j-1}}{A_{i,j}}\prod_{k=i-a}^{b-j}\lambda_k, \qquad \forall\, 1\leq i \leq a,\qquad \forall\, 1\leq j \leq b,
\end{equation}
where the indices of $\lambda_k$ are defined modulo $n$ and run as $i-a, i-a+1,\dots, b-j$.
\end{lem}
\begin{proof}We prove \eqref{Lexp} by a double induction on the indices $(i,j)$ in the descending order.
It is certainly true for $(a,b)$ since
\[
p^* (L_{a,b} )=p^* (X_{a,b} )=\lambda_n\frac{\Delta_{\{2, \dots, a,n\}}}{\Delta_{\{1,2,\dots,a\}}}=\lambda_n\frac{A_{a-1,b-1}}{A_{a,b}}.
\]
Take $(a,j)$ with $1\leq j<b$. The cluster coordinate $X_{a,j}$ is associated to a boundary face. By~\eqref{frozen X} and~\eqref{I},
\begin{equation*}
p^* (X_{a,j} )=p^* (X_{n-j} )=\lambda_{b-j}\frac{\Delta_{\{b-j, b-j+1,\dots,n-j-1\}}\Delta_{\{b-j+2,\dots, n-j,n\}}}{\Delta_{\{b-j+1,b-j+2, \dots,n-j\}}\Delta_{\{b-j+1,\dots,n-j-1,n\}}}= \lambda_{b-j} \frac{A_{a,j+1}}{A_{a,j}}\frac{A_{a-1,j-1}}{A_{a-1,j}}.
\end{equation*}
If \eqref{Lexp} is true for $(a, j+1)$, then
\begin{align*}
p^* (L_{a,j} )=p^* (X_{a,j}L_{a,j+1} )& =\left(\lambda_{b-j} \frac{A_{a,j+1}A_{a-1,j-1}}{A_{a,j}A_{a-1,j}}\right)\left(\frac{A_{a-1,j}}{A_{a,j+1}}\prod_{k=i-a}^{b-j-1}\lambda_k\right)\\
& =\frac{A_{a-1,j-1}}{A_{a,j}} \prod_{k=i-a}^{b-j}\lambda_k.
\end{align*}
Similarly, for $(i,b)$ with $1\leq i<a$, by induction we get
\begin{align*}
p^* (L_{i,b} )= p^* (X_{i,b}L_{i+1,b} )
& =\left(\lambda_{i-a}\frac{A_{i+1,b}A_{i-1,b-1}}{A_{i,b}A_{i,b-1}}\right) \left(\frac{A_{i,b-1}}{A_{i+1,b}}\prod_{\substack{k=\\i-a+1}}^{n}\lambda_k\right)\\
& =\frac{A_{i-1,b-1}}{A_{i,b}}\prod_{k=i-a}^{n}\lambda_k.
\end{align*}
Take $(i,j)$ with $i<a$ and $j<b$. Recall from~\eqref{def poisson cluster} that
\[
p^* (X_{i,j} ) = \frac{A_{i-1,j-1}A_{i,j+1}A_{i+1,j}}{A_{i-1,j} A_{i,j-1}A_{i+1,j+1}}.
\]
If \eqref{Lexp} is true for $(i+1,j)$, $(i,j+1)$ and $(i+1,j+1)$, then
\begin{gather*}
 p^* (L_{i,j} )  =p^*\left(\frac{X_{i,j}L_{i+1,j}L_{i,j+1}}{L_{i+1,j+1}}\right)
  = \left(\frac{A_{i-1,j-1}A_{i,j+1}A_{i+1,j}}{A_{i-1,j} A_{i,j-1}A_{i+1,j+1}}\right)\left(\frac{A_{i,j-1}}{A_{i+1,j}}\prod_{k=i-a+1}^{b-j}\lambda_k\right) \\
 \hphantom{p^* (L_{i,j} )  =}{}\times
  \left(\frac{A_{i-1,j}}{A_{i,j+1}}\prod_{k=i-a}^{b-j-1}\lambda_k\right)
 \left(\frac{A_{i,j}}{A_{i+1,j+1}}\prod_{k=i-a+1}^{b-j-1}\lambda_k\right)^{-1}
  =\frac{A_{i-1,j-1}}{A_{i,j}}\prod_{k=i-a}^{b-j}\lambda_k. \tag*{\qed}
\end{gather*}
\renewcommand{\qed}{}
\end{proof}

\section{Cluster duality}
\subsection{Duality conjecture and canonical basis}
\label{donj.con.10.32}
Let $\left(\mathscr{A}_{|Q|}, \mathscr{X}_{|Q|}\right)$ be the cluster ensemble associated to a quiver $Q$.
Recall the algebras ${\bf up}\left(\mathscr{A}_{|Q|}\right)$ and ${\bf up}\left(\mathscr{X}_{|Q|}\right)$ and the cluster modular group $\mathcal{G}_{|Q|}$. The Fock--Goncharov cluster duality conjecture asserts that

\begin{conj}[{\cite[Conjecture 4.1]{FGensemble}}]
\label{dual.conj} The algebra ${\bf up}(\mathscr{A}_{|Q|})$ admits a canonical basis $\mathcal{G}_{|Q|}$-equivariantly parametrized by $\mathscr{X}_{|Q|}\left(\mathbb{Z}^t\right)$. The algebra $\up\left(\mathscr{X}_{|Q|}\right)$ admits a canonical basis $\mathcal{G}_{|Q|}$-equivariantly parametrized by $\mathscr{A}_{|Q|}\left(\mathbb{Z}^t\right)$.
\end{conj}

Conjecture \ref{dual.conj} has been proved in \cite{GHKK} under two combinatorial assumptions, which can be summarized as follows.

\begin{thm}[{\cite[Proposition 0.14]{GHKK}}]
\label{okarghkk1124}
Conjecture~{\rm \ref{dual.conj}} holds if the following conditions are satisfied.
\begin{itemize}\itemsep=0pt
 \item The non-frozen part of the quiver $Q$ possesses a maximal green sequence or a reddening sequence.\footnote{Maximal green sequences and reddening sequences are special sequences of quiver mutations introduced by Keller~\cite{KelDT} in order to study Donaldson--Thomas transformations. In the original statement of Gross--Hacking--Keel--Kontsevich, the theorem was stated with maximal green sequence; but they only need the existence of a~maximal green sequence to ensure that the cluster complex does not lie within a~half space, which is also a~consequence of the existence of a~reddening sequence.}
 \item The exchange matrix $\varepsilon=(\varepsilon_{ij})$ of $Q$, with $i$ running through the non-frozen vertices and $j$ running through all the vertices, is of full rank.
\end{itemize}
\end{thm}

\begin{rmk} For any quiver $Q$ that satisfies the above combinatorial conditions, cluster duality gives rise to canonical bases of ${\bf up}\big(\mathscr{A}_{|Q|}\big)$ and ${\bf up}\big(\mathscr{X}_{|Q|}\big)$, which we denote as follows:
\begin{gather*}
\Theta_{\mathscr{A}}:=\big\{\theta_q \, \big| \, q\in \mathscr{X}_{|Q|}\big(\mathbb{Z}^t\big)\big\}\subset \mathbf{up} \big(\mathscr{A}_{|Q|} \big),
\\
\Theta_{\mathscr{X}}:=\big\{\vartheta_p \, \big| \, p\in \mathscr{A}_{|Q|}\big(\mathbb{Z}^t\big)\big\}\subset \mathbf{up} \big(\mathscr{X}_{|Q|} \big).
\end{gather*}
The basis elements $\theta_q$ and $\vartheta_p$ satisfy many remarkable properties (see~\cite{GHKK} for the comprehensive list).
One of them is that every $\theta_q \in \mathbf{up} \big(\mathscr{A}_{|Q|} \big)$ in terms of the $K_2$ cluster $\{A_j\}$ associated to $Q$ is expressed as
\begin{equation}\label{theta.express}
\theta_q=\prod_jA_j^{x_j}F\left(\left(\prod_jA_j^{\varepsilon_{kj}}\right)_{k\in I_{\uf}}\right),
\end{equation}
where $(x_j)$ is the tropical coordinates of $q\in \mathscr{X}\big(\mathbb{Z}^t\big)$ in terms of the Poisson cluster associated to the same quiver $Q$, and $F$ is a polynomial with constant term 1 and variables of the form $\prod_jA_j^{\varepsilon_{kj}}$ as $k$ ranges through all unfrozen vertices. The elements $\vartheta_{p}$ admit similar formulas.

One may notice that, on the one hand, we define $p^* (X_i )=\prod_jA_j^{\varepsilon_{ji}}$ in the definition of the~$p$ map, and on the other hand, the above polynomial $F$ depends on $\prod_jA_j^{\varepsilon_{ij}}=p^*\big(X_i^{-1}\big)$. The reason this happens is that the cluster Poisson variables used in this paper are inverses of those used by Gross, Hacking, Keel, and Kontsevich in~\cite{GHKK}. Such a switch frees us from considering tropicalization in which the tropical addition is maximum rather than minimum.
\end{rmk}

\begin{proof}[Proof of Theorem \ref{main.3.1.152}]
For the quiver $Q_{a,n}$, the existence of a maximal green sequence was proved by Marsh and Scott in \cite{MStw} and the existence of a cluster Donaldson--Thomas transformation (which is equivalent to a reddening sequence) was proved by Weng in~\cite{Weng}. By Lemma~\ref{full.rank.lem}, the second combinatorial condition holds. Hence Conjecture \ref{dual.conj} holds for the cluster ensemble $(\mathscr{A}_{a,n},\mathscr{X}_{a,n})$. By Theorems~\ref{Scott} and~\ref{X isomorphism} we know that $\up (\mathscr{A}_{a,n})\cong \mathcal{O}\big(\dGr_a^\times(n)\big)$ and $\up (\mathscr{X}_{a,n} )\cong \mathcal{O}\big(\dconf_n^\times(a)\big)$. This concludes the proof of Theorem~\ref{main.3.1.152}.
\end{proof}

\begin{rmk} This $\Theta$-basis was constructed by Gross, Hacking, Keel, and Kontsevich in~\cite{GHKK} as formal power series by counting broken lines in scattering diagrams. To find an explicit procedure to determine all basis elements as regular functions and to compare them with other notable bases of representations (e.g., Lusztig's canonical basis, Mirkovic--Vilonen basis, etc.) are interesting directions for future research.
\end{rmk}

\subsection{Partial compactification, optimized quiver, and potential function}
Let $Q=\big(I^{\uf} \subset I, \varepsilon\big)$ be a quiver satifying the combinatorial conditions in Theorem~\ref{okarghkk1124}. Let $I^0:= I- I^{\uf}$ be the set of frozen vertices.
For $i\in I^0$, let $D_i$ denote the (irreducible) boundary divisor of $\mathscr{A}_{|Q|}$ defined by setting $A_i=0$.
Let us glue $\mathscr{A}_{|Q|}$ with these boundary divisors, obtaining the partially compactified space
\[\overline{\mathscr{A}}_{|Q|} := \mathscr{A}_{|Q|} \cup \left(\bigcup_{i \in I^0} D_i\right).
\]
This section is devoted to studying the ring $\mathcal{O}\big(\overline{\mathscr{A}}_{|Q|}\big)$ of regular functions on $\overline{\mathscr{A}}_{|Q|}$.

Let $f\in \mathbf{up} \big(\mathscr{A}_{|Q|} \big)$. Denote by ${\rm ord}_{D_i}(f)$ the order of $f$ along the boundary divisor $D_i$. Note that $f$ can be extended to a regular function on $D_i$ if and only if ${\rm ord}_{D_i}(f) \geq 0$. Therefore
\[
\mathcal{O}\big(\overline{\mathscr{A}}_{|Q|}\big)= \big\{ f \in \mathbf{up} (\mathscr{A}_{|Q|} ) \,\big|\, {\rm ord}_{D_i} (f)\geq 0, \, \forall\, i \in I^0 \big\}.
\]
Recall the canonical basis $\Theta_{\mathscr{A}}$ of $\mathcal{O}(\mathscr{A}_{|Q|})$. Consider the intersection
\begin{equation*}
 \Theta_{\overline{\mathscr{A}}}:=\Theta_{\mathscr{A}}\cap \mathcal{O}\left(\overline{\mathscr{A}}_{|Q|}\right) =\big\{ \theta_q \in \Theta_{\mathscr{A}} \,\big|\, {\rm ord}_{D_i} (\theta_q)\geq 0, \, \forall i \in I^0 \big\}.
\end{equation*}
Conjecture~9.8 of~\cite{GHKK} implies that the intersection $ \Theta_{\overline{\mathscr{A}}}$ descends to a linear basis of $\mathcal{O}\big(\overline{\mathscr{A}}_{|Q|}\big)$.
The paper {\it loc.\ cit.} provides a sufficient condition under which the aforementioned conjecture holds.

\begin{defn}Let $i\in I^0$. If $\varepsilon_{ki}\geq 0$ for all unfrozen vertices $k$, then we say the quiver $Q$ is \emph{optimized} for $i$. If there exists a mutation sequence $\tau$ such that the mutated quiver $\tau Q$ is optimized for $i$, then we say that $i$ admits an optimized quiver in the equivalence class~$|Q|$.
\end{defn}

\begin{rmk} Because of the different conventions used in this paper vs.\ those in \cite{GHKK}, we define~$Q$ to be optimized for~$i$ if all arrows between $i$ and unfrozen vertices point towards the unfrozen ones, as opposed to the other direction stated in \cite[Definition~9.1]{GHKK}.
\end{rmk}

\begin{prop}\label{general.dual.part1}If every frozen vertex $i$ of $Q$ admits an optimized quiver in~$|Q|$, then the set~$\Theta_{\overline{\mathscr{A}}}$ forms a linear basis of $\mathcal{O}\big(\overline{\mathscr{A}}_{|Q|}\big)$.
\end{prop}

\begin{proof} The linear independence of $\Theta_{\overline{\mathscr{A}}}$ is clear.
Suppose that \[f= \sum_{q\in \mathscr{X}(\mathbb{Z}^t)} \alpha_q \theta_q \in \mathcal{O}\big(\overline{\mathscr{A}}_{|Q|}\big).\]
By definition, ${\rm ord}_{D_i}(f)\geq 0$ for every frozen $i$. By \cite[Proposition~9.7]{GHKK}, if $i$ admits an optimized quiver in $|Q|$, then ${\rm ord}_{D_i}(\theta_q)\geq 0$ for all $q$ with $\alpha_q\neq 0$. Therefore whenever the coefficient $\alpha_q$ is nonzero, the function $\theta_q \in \Theta_{\overline{\mathscr{A}}}$. In other words, the set $\Theta_{\overline{\mathscr{A}}}$ spans $\mathcal{O}\big(\overline{\mathscr{A}}_{|Q|}\big)$.
\end{proof}

Now it is natural to address the following question: for which $q\in \mathscr{X}_{|Q|}\big(\mathbb{Z}^t\big)$ do we have $\theta_q\in \Theta_{\overline{\mathscr{A}}}$? A criterion for recognizing such $q$'s was suggested in \cite[Definition~12.7]{GS1}, and was proved in~\cite{GHKK}.

Let $i\in I^0$. Let $p_i\in \mathscr{A}_{|Q|}\big(\mathbb{Z}^t\big)$ be the tropical point such that its tropical coordinates \mbox{$A_j^t(p_i)=\delta_{ij}$}, where $\delta_{ij}$ is the Kronecker delta symbol. The existence and uniqueness of $p_i$ is a direct consequence of the definition of cluster $K_2$ mutations. Let $\vartheta_{p_i}\in \Theta_{\mathscr{X}}$ be the theta function parametrized by~$p_i$. Define the \emph{GHKK potential}
\begin{equation}\label{GHKK.potential}
W := \sum_{i\in I^0} \vartheta_{p_i}.
\end{equation}

The following Proposition is a paraphrase of \cite[Lemma~9.3]{GHKK}; we include it below for the sake of completeness.

\begin{prop}\label{general.dual.part2}Assume that every frozen vertex of $Q$ admits an optimized quiver in~$|Q|$. Let $q\in \mathscr{X}_{|Q|}\big(\mathbb{Z}^t\big)$. Then $\theta_q \in \Theta_{\overline{\mathscr{A}}}$ if and only if $W^{t}(q)\geq 0$.
\end{prop}
\begin{proof} By the definition of tropicalization, we have
\[
W^t(q)= \min_{i\in I^0}\big\{\vartheta_{p_i}^t (q)\big\}.
\]
Therefore $W^t(q)\geq 0$ if and only if every $\vartheta_{p_i}^t (q)\geq 0$. It suffices to show that if $i\in I^0$ admits an optimized quiver, then
\begin{equation}\label{528}
{\rm ord}_{D_i}(\theta_q)=\vartheta_{p_i}^t (q).
\end{equation}
Without loss of generality, let us assume that $Q$ is optimized for $i$, i.e., $\varepsilon_{ki}\geq 0$ for all unfrozen vertices $k$.
Let $\{A_j\}$ be the cluster of $\mathscr{A}_{|Q|}$ associated to~$Q$. Let~$X_i$ be the cluster variable of~$\mathscr{X}_{|Q|}$ associated to the vertex $i$ in $Q$ and let $x_i:=X_i^t$ be its tropicalization. Since~$\theta_q$ is of the form~\eqref{theta.express} and $\epsilon_{ki}\geq 0$ for all unfrozen vertices $k$, it follows that
\[
{\rm ord}_{D_i}(\theta_q) = x_i(q).
\]
On the other hand, by \cite[Lemma~9.3]{GHKK}, if $Q$ is optimized for $i$, then $\vartheta_{p_i}=X_i$. Therefore $x_i(q)=X_i^t(q)=\vartheta_{p_i}^t(q)$, which concludes the proof of~\eqref{528}.
\end{proof}

Let us apply the above results to the cases of Grassmannians. It boils down to finding optimized quivers for frozen vertices in the quiver $Q_{a,n}$. As observed by L. Williams and stated in \cite[Proposition~9.4]{GHKK}, the quiver $Q_{a,n}$ is optimized for the vertices $(0,0)$ and $(a,b)$; since the mutation sequence $\rho$ in~\eqref{rho} rotates the frozen vertices of~$Q_{a,n}$ clockwise to their neighbors, by applying $\rho$ repeatedly, each frozen vertex $i$ has a chance to be at the position of $(0,0)$ and therefore admits an optimized quiver. Indeed, the quiver $\rho^{n-i}Q_{a,n}$ is optimized for the frozen vertex~$i$.

The next Proposition shows that in the Grassmannian case, the potential function $\mathcal{W}=\sum_{i} \vartheta_i$ in~\eqref{potential.W} coincides with the function~$W$ in~\eqref{GHKK.potential}.

\begin{prop}\label{vartheta_i} Under the algebra isomorphism $\up (\mathscr{X}_{a,n} )\cong \mathcal{O}\big(\dconf_n^\times(a)\big)$, the theta function~$\vartheta_{p_i}$ is identified with the function $\vartheta_i$ defined in~\eqref{theta.f}.
\end{prop}
\begin{proof}Note that $Q_{a,n}$ is optimized for $(0,0)$. By \eqref{3.14..1214} and \cite[Lemma~9.3]{GHKK}, we have
\[
\vartheta_n= X_{0,0}=\vartheta_{p_n}.\]
For the other frozen vertices, let us apply the rotation mutation sequence $\rho$. Then
\begin{equation}
\vartheta_{p_i}=X_n^{\rho^{n-i}}=\big(R^{n-i}\big)^*X_n=\big(R^{n-i}\big)^*\vartheta_n=\vartheta_i. \tag*{\qed}
\end{equation}
\renewcommand{\qed}{}
\end{proof}

By Theorem \ref{Scott}, or more precisely by the original version \cite[Theorem 3]{Sco}, the coordinate ring $\mathcal{O}(\dGr_a(n))$ is isomorphic to $\mathcal{O}\big(\overline{\mathscr{A}}_{a,n}\big)$. Combining Propositions \ref{general.dual.part1}, \ref{general.dual.part2}, and \ref{vartheta_i}, we get the following theorem.

\begin{thm} Under the isomorphisms $\mathcal{O}(\dGr_a(n))\!\stackrel{\sim}{=}\!\mathcal{O}\big(\overline{\mathscr{A}}_{a,n}\big)$ and $ \mathcal{O}\big(\dconf_n^\times(a)\big)\!\cong\! \up (\mathscr{X}_{a,n} )$,
the coordinate ring $\mathcal{O} (\dGr_a(n) )$ admits a natural basis
\[
\big\{\theta_q \, \big| \, q\in \dconf_n^\times(a)\big(\mathbb{Z}^t\big), \, \mathcal{W}^t(q)\geq 0\big\}.
\]
\end{thm}

\subsection[$\mathbb{G}_m$-action]{$\boldsymbol{\mathbb{G}_m}$-action}\label{gmaction.17.06}

Recall the $\mathbb{G}_m$-action on $\dGr_a(n)$ in \eqref{torus action1044}. Let us restrict the $\mathbb{G}_m$-action to the open subset~$\dGr_a^\times(n)$. It induces a $\mathbb{G}_m$-action on $\mathcal{O}\big(\dGr_a^\times(n)\big)$ extending the one on $\mathcal{O} (\dGr_a(n))$.

Recall the twisted monodromy function $P$ on $\dconf_n^\times(a)$.

\begin{prop} Let $q\in \dconf_n^\times(a)\big(\mathbb{Z}^t\big)$. Its corresponding theta function $\theta_q \in \mathcal{O}\big(\dGr_a^\times(n)\big)$ is an eigenvector of the $\mathbb{G}_m$-action with weight $P^t(q)$:
\[
t.\theta_q = t^{P^t(q)}\theta_q
\]
\end{prop}
\begin{proof}Let $\{A_{i,j}\}$ be the $K_2$ cluster of $\dGr_a^\times(n)$ associated to $Q_{a,n}$. By~\eqref{theta.express}, the function $\theta_q$ can be expressed as a Laurent polynomial
\begin{equation}\label{thetaex1228}
\theta_q=\prod_{(i,j)}A_{i,j}^{x_{i,j}}F\left(\left(\prod_gA_g^{\varepsilon_{fg}}\right)_{f\in I_{\uf}}\right),
\end{equation}
where $ (x_{i,j})$ is the tropical coordinates of $q$ with respect to the quiver $Q_{a,n}$, and $F$ is a polynomial with constant term 1 and variables of the form $\prod_gA_g^{\varepsilon_{fg}}$ for $f\in I_{\uf}$.

By definition, every $A_f$ is a Pl\"{u}cker coordinate and therefore is of weight 1 with respect to the $\mathbb{G}_m$-action. Since $\sum_f\varepsilon_{fg}=0$ for all $g\in I_{\uf}$ by construction, the whole factor $F$ is invariant under the $\mathbb{G}_m$-action. It implies that $\theta_q$ is an eigenvector of weight $\sum_{(i,j)}x_{i,j}$. By tropicalizing $P=\prod_{f\in I}X_f$ (Proposition \ref{product}), we have
\[
P^t(q) = \sum_{(i,j)}x_{i,j},
\]
which concludes the proof.
\end{proof}

As a direct consequence, we get the following corollary.

\begin{cor}\label{basis} The representation $\mathcal{O} (\dGr_a(n) )_c=V_{c\omega_a}$ has a canonical basis
\begin{equation}
\label{theta.basis.repv1120}
\Theta(a,b,c):= \{\theta_q \,  | \, q\in Q(a,b,c) \},
\end{equation}
where $Q(a,b,c)$ is defined in~\eqref{trop.conf}. By combining this result with Proposition~{\rm \ref{PQ bijection}} we deduce that there is a natural bijection between $\Theta(a,b,c)$ and the set of plane partitions~$P(a,b,c)$.
\end{cor}

\begin{rmk} Rietsch and Williams proved a similar result on parametrization of a basis of~$V_{c\omega_a}$ by plane partitions \cite[Lemma~16.16]{RW}, which also relies on cluster duality. However, they approached the problem from the other direction by realizing the open positroid variety~$\Gr_a^\times(n)$ as a cluster Poisson variety and choosing basis elements from the canonical basis of $\up (\mathscr{X} )$, which is parametrized by $\mathscr{A}\big(\mathbb{Z}^t\big)$. In order to do this, they break the cyclic symmetry of Grassmannian by fixing a particular Pl\"{u}cker coordinate, denoted by $P_{\max}$ in \cite[Section~1.6]{RW}. As a~consequence, the basis in \cite[Lemma~16.16]{RW} does not respect the twisted cyclic rotation on Grassmannian. In contrast, the basis $\Theta(a,b,c)$ in Corollary~\ref{basis} is invariant under twisted cyclic rotation, which is crucial to our proof of a cyclic sieving phenomenon of plane partitions (Theorem~\ref{main1}).

In this paper, the basis elements in~\eqref{theta.basis.repv1120} are realized as functions on the decorated Grassmannian~$\dGr_a(n)$, as opposed to the ordinary Grassmannian in {\it loc.\ cit}. And the space $\dGr_a(n)$ is realized as a cluster $K_2$ variety, as opposed to the space $\Gr_a^\times(n)$ as a cluster Poisson variety in {\it loc.\ cit}. Our basis~$\Theta(a,b,c)$ has an invariance property with respect to the twisted cyclic rotation $C_a$, which plays a crucial role in our proof of cyclic sieving phenomenon.
Moreover, we will prove in the next section that our basis also fits into the weight space decomposition of~$V_{c\omega_a}$, which is a stronger result.
\end{rmk}

\subsection{Torus action and weight space decomposition}

By \eqref{gr isomorphic to conf of column vectors}, $\dGr_a^\times(n)\cong \SL_a \backslash \mat_{a,n}^\times$.
The group ${\rm GL}_n$ acts on the right of $\mat_{a,n}^\times$ by matrix multiplication. The maximal torus
$T= (\mathbb{G}_m )^n\subset {\rm GL}_n$ of diagonal matrices acts by rescaling the column vectors $v_i$ of the matrices in $\mat_{a,n}^\times$
\[
 (v_1,\dots, v_n ).  (t_1,\dots, t_n ) := (t_1v_1,t_2v_2,\dots, t_nv_n ).
\]
Its induced left $\left(\mathbb{G}_m\right)^n$-action on $\mathcal{O}\big(\dGr_a^\times(n)\big)$ gives rise to a weight space decomposition
\[
\mathcal{O}\big(\dGr_a^\times(n)\big)= \bigoplus_\mu\mathcal{O}(\mu),
\]
where $\mu= (\mu_1,\dots, \mu_n )\in \mathbb{Z}^n$ and $\mathcal{O}(\mu)$ consists of the functions $F$ such that
\[
 (t_1,\dots, t_n ).F=t_1^{\mu_1}\cdots t_n^{\mu_n}F.
\]
In particular, if we restrict to the representation $V_{c\omega_a}=\mathcal{O} (\dGr_a(n) )_c$, then we get the weight space decomposition
\[V_{c\omega_a}=\bigoplus_{\mu}V_{c\omega_a}(\mu), \qquad \mbox{where} \quad
V_{c\omega_a}(\mu):=V_{c\omega_a} \cap \mathcal{O}(\mu).
\]
In this section, we show that the theta basis $\Theta(a,b,c)$ is compatible with the weight space decomposition of the representation $V_{c\omega_a}$.

Recall the following dual torus projection defined in \eqref{weightmap1142}
\[
M=(M_1, \dots, M_n)\colon \  {\rm Conf}_n^\times(a)\longrightarrow T^\vee.
\]
Let us tropicalize the map $M$, obtaining
\[
M^t=\big(M_1^t, \dots, M_n^t\big)\colon \ \dconf_n^\times(a)\big(\mathbb{Z}^t\big)\longrightarrow T^\vee\big(\mathbb{Z}^t\big)\stackrel{\sim}{=}\mathbb{Z}^n.
\]

\begin{prop}\label{weightprop1220}Let $q\in \dconf_n^\times(a)$. The theta function $\theta_q$ is an eigenvector of the $T$-action on $\mathcal{O}\big(\dGr_a^\times(n)\big)$ with weight $M^t(q)$, i.e.,
\[
(t_1, \dots, t_n). \theta_q = t_1^{M_1^t(q)}\cdots t_n^{M_n^t(q)} \theta_q.
\]
\end{prop}
The proof of Proposition \ref{weightprop1220} will require a little preparation.
Recall the $a$-element set $I(i,j)$ assigned to each vertex $(i,j)$ of the quiver $Q_{a,n}$ \eqref{I}:
\[
I(i,j)=\{\underbrace{b-j+1,\dots, b-j+i}_{\text{$i$ indices}},\underbrace{b+i+1,\dots, n}_{\text{$a-i$ indices}}\}.
\]
For $k\in \{1, \dots, n\}$, let $F_k$ denote the collection of vertices $(i,j)$ in $Q_{a,n}$ such that $k\in I(i,j)$.
By the definition of $I(i,j)$, it is easy to check that the sets $F_k$ are of two patterns.
When $1\leq k \leq b$, the vertices in $F_k$ are enclosed in a stair-shape diagram, such that the difference between the consecutive steps is 1, until its height becomes $0$ or until it touches the rightmost column. When $b<k\leq n$, we get a smaller stair with another part that consists of the $(0,0)$ vertex and possibly a rectangle:
\[
\begin{tikzpicture}[scale=0.8]
\draw[fill=gray!20] (4,1) -- (4,0) -- (5,0) -- (5, -2) -- (0,-2) -- (0,4) -- (1,4) -- (1,3) -- (2,3) -- (2,2) -- (3,2) -- (3,1);
\node at (3.5,0.75) [] {$\ddots$};
\node at (2, 1) [] {$F_k$};
\node [scale=0.5] at (0.5,3.5) [] {$(1,b-k+1)$};
\node [scale=0.5] at (1.5,2.5) [] {$(2,b-k+2)$};
\node at (2,-3) [] {$1\leq k\leq b$};
\end{tikzpicture} \quad \quad \quad \quad
\begin{tikzpicture}[scale=0.8]
\draw[fill=gray!20] (-2,2) -- (-2,3) -- (-3, 3) -- (-3,0) -- (0,0) -- (0,1) -- (-1,1) -- (-1,2);
\draw[fill=gray!20] (-3,4) -- (4,4) -- (4,6) -- (-3,6);
\node at (-1.5,2) [] {$\ddots$};
\draw (-3,3) -- (-3,6);
\draw[fill=gray!20] (-4,6) rectangle (-3,7);
\node at (1,5) [] {$F_k$};
\node at (-2,1) [] {$F_k$};
\node [scale=0.5] at (-3,3.5) [right] {$(k-b,1)$};
\node [scale=0.5] at (-3.75,6.5) [right] {$(0,0)$};
\node [scale=0.5] at (3.25,5.5) [right] {$(1,b)$};
\node at (2,-1) [] {$b\leq k\leq n$};
\end{tikzpicture}
\]

\begin{exmp}
When $(a,n)=(2,5)$, the sets $F_k$ are depicted as follows:
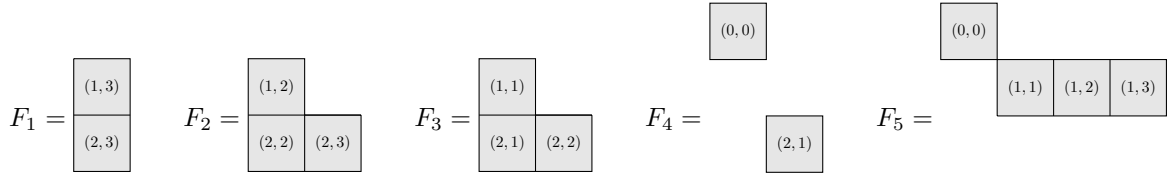
\[
F_1=\begin{tikzpicture}[baseline=4ex,scale=0.75]
\draw[fill=gray!20] (0,2) -- (0,0) -- (1,0) -- (1,2) -- (0,2);
\draw (0,1) -- (1,1);
\node [scale=0.6] at (0.5,0.5) [] {$(2,3)$};
\node [scale=0.6] at (0.5,1.5) [] {$(1,3)$};
\end{tikzpicture}
\quad
F_2=\begin{tikzpicture}[baseline=4ex,scale=0.75]
\draw[fill=gray!20] (2,1) -- (1,1) -- (1,2) -- (0,2) -- (0,0) -- (2,0) -- (2,1);
\draw (1,0) -- (1,1);
\draw (0,1) -- (2,1);
\node [scale=0.6] at (0.5,0.5) [] {$(2,2)$};
\node [scale=0.6] at (1.5,0.5) [] {$(2,3)$};
\node [scale=0.6] at (0.5,1.5) [] {$(1,2)$};
\end{tikzpicture}
\quad
F_3=\begin{tikzpicture}[baseline=4ex,scale=0.75]
\draw[fill=gray!20] (2,1) -- (1,1) -- (1,2) -- (0,2) -- (0,0) -- (2,0) -- (2,1);
\draw (1,0) -- (1,1);
\draw (0,1) -- (2,1);
\node [scale=0.6] at (0.5,0.5) [] {$(2,1)$};
\node [scale=0.6] at (1.5,0.5) [] {$(2,2)$};
\node [scale=0.6] at (0.5,1.5) [] {$(1,1)$};
\end{tikzpicture}\quad
F_4=\begin{tikzpicture}[baseline=14ex,scale=0.75]
\draw[fill=gray!20] (-1,4) rectangle (0,5);
\draw[fill=gray!20] (0,2) rectangle (1,3);
\node [scale=0.6] at (-0.5,4.5) [] {$(0,0)$};
\node [scale=0.6] at (0.5,2.5) [] {$(2,1)$};
\end{tikzpicture}
\quad
F_5=\begin{tikzpicture}[baseline=14ex,scale=0.75]
\draw[fill=gray!20] (0,3) -- (3,3) -- (3,4) -- (0,4) -- (0,3);
\draw (1,3) -- (1,4);
\draw (2,3) -- (2,4);
\node [scale=0.6] at (2.5,3.5) [] {$(1,3)$};
\node [scale=0.6] at (1.5,3.5) [] {$(1,2)$};
\node [scale=0.6] at (0.5,3.5) [] {$(1,1)$};
\draw[fill=gray!20] (-1,4) rectangle (0,5);
\node [scale=0.6] at (-0.5,4.5) [] {$(0,0)$};
\end{tikzpicture}
\]
\end{exmp}

\begin{lem} \label{1241}
Recall the clusters $\{X_{i,j}\}$ of $\dconf_n^\times(a)$ associated to $Q_{a,n}$ in \eqref{def poisson cluster}. The function
\begin{equation*}
M_k=\prod_{(i,j)\in F_k}X_{i,j}.
\end{equation*}
\end{lem}
\begin{proof} We will prove the lemma under the assumption that $a\leq b$ (the cases with $a\geq b$ are completely analogous). When $1\leq k\leq a$, we have
\[
\prod_{(i,j)\in F_k}X_{i,j} = \frac{L_{1, b-k+1}}{L_{1, b-k+2}}\frac{L_{2, b-k+2}}{L_{2, b-k+3}}\cdots \frac{L_{k-1, b-1}}{L_{k-1, b}} L_{k, b}=\frac{\prod\limits_{i=1}^k L_{i,b-k+i}}{\prod\limits_{i=1}^{k-1}L_{i,b-k+i+1}}
\]
Then by Lemma \ref{L},
\[
p^*\left(\frac{\prod\limits_{i=1}^k L_{i,b-k+i}}{\prod\limits_{i=1}^{k-1}L_{i,b-k+i+1}}\right)= \frac{\Delta_{\{k-a,\dots,k-1\}}}{\Delta_{\{k-a+1,\dots, k\}}}\prod_{i=k-a}^{k-1} \lambda_i=p^* (M_k ),
\]
where the last equality follows from the definition of $M_k$ in \eqref{M} (please keep in mind that all indices of $\lambda$ and $\Delta$ are taken modulo $n$). When $a< k\leq b$, we get
\[
\prod_{(i,j)\in F_k}X_{i,j}=\frac{L_{1,b-k+1}}{L_{1,b-k+2}}\frac{L_{2,b-k+2}}{L_{2,b-k+3}}\cdots \frac{L_{a,n-k}}{L_{a,n-k+1}}=\frac{\prod\limits_{i=1}^aL_{i,b-k+i}}{\prod\limits_{i=1}^aL_{i,b-k+i+1}}.
\]
Again by Lemma \ref{L},
\[
p^*\left(\frac{\prod\limits_{i=1}^aL_{i,b-k+i}}{\prod\limits_{i=1}^{a}L_{i,b-k+i+1}}\right)=\frac{\Delta_{\{k-a,\dots, k-1\}}}{\Delta_{\{k-a+1,\dots, k\}}}\prod_{i=k-a}^{k-1}\lambda_i
=p^* (M_k ).
\]
Lastly, when $b<k\leq n$, we have
\[
p^*\left(\prod_{(i,j)\in F_k}X_{i,j}\right)=p^*\left(L_{0,0}\frac{\prod\limits_{i=1}^{n-k}L_{k-b+i,i}}{\prod\limits_{i=0}^{n-k}L_{k-b+i,i+1}}\right) =\frac{\Delta_{\{k-a,\dots, k-1\}}}{\Delta_{\{k-a+1,\dots, k\}}}\prod_{i=k-a}^{k-1}\lambda_i=p^*\left(M_k\right).\tag*{\qed}
\]
\renewcommand{\qed}{}
\end{proof}

\begin{proof}[Proof of Proposition \ref{weightprop1220}]
The proof makes use of the expression \eqref{thetaex1228} of $\theta_q$ again. For a~nonfrozen vertex $f$, the product $\prod_gA_g^{\varepsilon_{fg}}=p^*\big(X_f^{-1}\big)$ is independent of the rescaling of the column vectors $v_i$ due to the well-defined-ness of the unfrozen variable $X_f$. Therefore the polynomial $F$ in \eqref{thetaex1228} is invariant under the rescaling $T$-action. For the Pl\"ucker coordinates~$A_{i,j}$, note that it is affected by the $t_k$ component of $T$ if and only if $(i,j)\in F_k$. Therefore
\[
 (t_1,\dots, t_n ).\theta_q=t_1^{\mu_1}t_2^{\mu_2}\cdots t_n^{\mu_n}\theta_q, \qquad \mbox{where} \quad \mu_k = \sum_{(i,j)\in F_k}x_{i,j}.
\]
By Lemma \ref{1241}, we get $\sum\limits_{(i,j)\in F_k}x_{i,j}=M_k^t(q)$.
\end{proof}

Combining Corollary \ref{basis} with Proposition \ref{weightprop1220}, we get the following result.
\begin{cor} \label{weight decomposition}
The weight space $V_{c\omega_a}(\mu)$ has a canonical basis
\begin{equation*}
\Theta(a,b,c)(\mu):= \big\{\theta_q\in \Theta(a,b,c)\, \big| \, M^t(q)=\mu\big\}.
\end{equation*}
\end{cor}

Recall that $\Theta(a,b,c)$ is parametrized by the set $P(a,b,c)$ of plane partitions.
In the rest of this section, we present a concrete decomposition of $P(a,b,c)$ compatible with the above decomposition of $\Theta(a,b,c)$.

Recall that the Gelfand--Zetlin patterns \cite{GZ} for ${\rm GL}_n$ are triangular arrays of integers with non-increasing rows and columns as follows
\[
\Lambda=\left(\begin{matrix} \lambda_{1,1} & \lambda_{1,2} & \lambda_{1,3} & \cdots & \lambda_{1,n} \\
& \lambda_{2,2} & \lambda_{2,3} & \cdots & \lambda_{2,n} \\
& & \lambda_{3,3} & \cdots & \lambda_{3,n}\\
& & & \ddots & \vdots \\
& & & & \lambda_{n,n}
\end{matrix}\right).
\]
Let $\delta_i:=\sum\limits_{k=1}^i \lambda_{k,n-i+k}$ be the sums of entries along diagonals.
Define
\begin{equation}\label{gz formula}
\wt(\Lambda):=(\delta_1, \delta_2-\delta_1, \dots, \delta_{n}-\delta_{n-1}).
\end{equation}

Now let $\pi=(\pi_{ij})\in P(a,b,c)$. Note that $0\leq \pi_{i,j}\leq c$. It uniquely determines a Gelfand--Zetlin pattern as follows
\[
\Lambda_\pi:=\left( \begin{matrix}
c & c & \cdots & c & \pi_{1,1} & \pi_{1,2} & \cdots & \pi_{1,b} \\
& c & \cdots & c & \pi_{2,1} & \pi_{2,2} & \cdots & \pi_{2,b}\\
& & \ddots & \vdots & \vdots & \vdots & \ddots & \vdots \\
& & & c & \pi_{a,1} & \pi_{a,2} & \cdots & \pi_{a,b} \\
& & & & 0 & 0 & \cdots & 0 \\
& & & & & 0 & \cdots & 0\\
& & & & & & \ddots & \vdots\\
& & & & & & & 0
\end{matrix}\right).
\]
Consider the decomposition
\[
P(a,b,c)=\bigsqcup_{\mu} P(a,b,c)(\mu), \qquad \mbox{where}\quad
P(a,b,c)(\mu) := \{\pi \in P(a,b,c) \,|\, \wt(\Lambda_\pi)=\mu \}.
\]
\begin{prop} \label{prof1500}
The basis $\Theta(a,b,c)(\mu)$ is in natural bijection with $P(a,b,c)(\mu)$.
\end{prop}
\begin{proof}
Recall the tropical Gelfand--Zetlin coordinates $\{l_{i,j}\}$ of $\dconf_{n}^\times(a)\big(\mathbb{Z}^t\big)$. Let us arrange them in the following triangular pattern
\[ \begin{matrix}
l_{0,0} & l_{0,0} & \cdots & l_{0,0} & l_{1,1} & l_{1,2} & \cdots & l_{1,b} \\
& l_{0,0} & \cdots & l_{0,0} & l_{2,1} & l_{2,2} & \cdots & l_{2,b}\\
& & \ddots & \vdots & \vdots & \vdots & \ddots & \vdots \\
& & & l_{0,0} & l_{a,1} & l_{a,2} & \cdots & l_{a,b} \\
& & & & 0 & 0 & \cdots & 0 \\
& & & & & 0 & \cdots & 0\\
& & & & & & \ddots & \vdots\\
& & & & & & & 0
\end{matrix}
\]
Consider the sums of $l_{i,j}$ along each diagonal:
\[
\delta_1= l_{1,b}, \quad \delta_2= l_{1, b-1}+l_{2,b}, \quad \dots, \quad \delta_n= al_{0,0}.
\]
 Following the same argument as in the proof of Lemma \ref{1241}, we get
\begin{equation}\label{Hello1416}
M^t=\big(M_1^t, M_2^t, \dots, M_n^t\big)=  (\delta_1, \delta_2-\delta_1, \dots, \delta_n-\delta_{n-1} ).
\end{equation}
The proposition then follows from a comparison between \eqref{Hello1416} and \eqref{gz formula}.
\end{proof}

\section{Cyclic sieving phenomenon of plane partitions}

As an application of cluster duality for Grassmannians, we use the basis for $V_{c\omega_a}$ obtained in Corollary~\ref{basis} to prove a cyclic sieving phenomenon of plane partitions under the sequence of toggles $\eta$ defined in Section~\ref{intro csp}.

First we recall that the parametrization of basis we obtained from cluster duality is equivariant with respect to the cluster modular group action (Theorem~\ref{main.3.1.152}). Since both the twisted cyclic rotation $C_a$ on $\dGr_a^\times(n)$ and the rotation~$R$ on $\dconf_n^\times(a)$ come from the same cluster modular group element (Section~\ref{cyc.rot}), it follows that
\[
C_a^*\theta_q=\theta_{R^t(q)}.
\]

On the other hand, we discovered a natural bijection between the parametrization set $Q(a,b,c)\!$ of the basis $\{\theta_q\, | \, q\in Q(a,b,c)\}$ of $V_{c\omega_a}$ and the set of plane partitions $P(a,b,c)$ using tropical Gelfand--Zetlin coordinates (Proposition \ref{PQ bijection}). We claimed in Theorem \ref{main9} that the action of~$C_a$ on the basis element~$\theta_q$ is equivalent to the sequence of toggles $\eta$ on plane partitions, so let us first use Gelfand--Zetlin coordinates once more to prove Theorem~\ref{main9}.

We begin by setting $L_{i,0}=L_{0,j}=P$ and $L_{a+1,j}=L_{i,n-a+1}=1$.
Let
\[
L_{i,j}':= R^* L_{i,j},
\]
where $R^*$ denotes the pull-back via the biregular morphism $R$. Note that $R^*P=P$. Therefore $L_{i,0}'=L_{0,j}'=P$ and $L_{a+1,j}'=L_{i,n-a+1}'=1$.

\begin{lem}\label{mutation L} We have
\[L'_{i,j}=\frac{\big(L'_{i,j-1}+L_{i-1,j}\big)L'_{i+1,j}L_{i,j+1}}{L_{i,j}\big(L'_{i+1,j}+L_{i,j+1}\big)}, \qquad \forall\, 1\leq i \leq a,\quad \forall\, 1\leq j \leq b. \]
\end{lem}
\begin{proof} Throughout the proof we will adopt the convention of using a prime superscript to denote the coordinates after the cyclic rotation action. Let
$
A'_{i,j}:=C_a^* A_{i,j}=\Delta_{I'(i,j)},$ with~$I'(i,j)$ defined as \eqref{I'}:
\[ I'(i,j) =\{\underbrace{b-j,\dots, b-j+i-1}_\text{consecutive $i$ indices}, \underbrace{b+i, \dots, n-1}_\text{consecutive $a-i$ indices}\}
\]
By Lemma \ref{L} we get
\[
p^*\left(L'_{i,j}\right)=\frac{A'_{i-1,j-1}}{A'_{i,j}}\prod_{k=i-a}^{b-j}\lambda'_k=\frac{A'_{i-1,j-1}}{A'_{i,j}}\prod_{k=i-a-1}^{b-j-1}\lambda_k.
\]
The last equality is obtained from the fact that $\lambda'_k=\lambda_{k-1}$ for all $k$, which is due to the cyclic rotation. Using $I(i,j)$ \eqref{I} and $I'(i,j)$ \eqref{I'} one can deduce the following identity from the Pl\"{u}cker relations \eqref{Plucker relation}:
\begin{gather*}
A'_{i-1,j-2}A_{i-1,j}+A_{i-2,j-1}A'_{i,j-1}
= A'_{i-1,j-1}A_{i-1,j-1},\\
A'_{i,j-1}A_{i,j+1}+A_{i-1,j}A'_{i+1,j}
= A'_{i,j}A_{i,j}.
\end{gather*}
Let \looseness=-1 us take the ratio of them and multiply by $\prod\limits_{k=i-a}^{b-j}\lambda_k\lambda_{k-1}$ on both sides. The right hand side is
\[
{\rm r.h.s.} = \left(\frac{A'_{i-1,j-1}}{A'_{i,j}}\prod_{k=i-a-1}^{b-j-1}\lambda_k\right)\left( \frac{A_{i-1,j-1}}{A_{i,j}}\prod_{k=i-a}^{b-j}\lambda_k
\right)= p^*\big(L_{i,j}'L_{i,j}\big).\]
The left hand side becomes
\begin{align*}
{\rm l.h.s.}
&=\frac{\big(A'_{i-1,j-2}A_{i-1,j}+A_{i-2,j-1}A'_{i,j-1}\big)/ \big(A_{i-1,j}A'_{i,j-1}\big)}{\big(A'_{i,j-1}A_{i,j+1}+A_{i-1,j}A'_{i+1,j}\big)/ \big(A_{i-1,j}A'_{i,j-1}\big)}\cdot \prod_{k=i-a}^{b-j}\lambda_k\lambda_{k-1}\\
&= \left(\frac{A'_{i-1,j-2}}{A'_{i,j-1}}+\frac{A_{i-2,j-1}}{A_{i-1,j}}\right) \cdot \left(\frac{A_{i,j+1}}{A_{i-1,j}}+\frac{A'_{i+1,j}}{A'_{i,j-1}}\right)^{-1}\cdot \prod_{k=i-a}^{b-j}\lambda_k\lambda_{k-1}\\
& = \left(\frac{A'_{i-1,j-2}}{A'_{i,j-1}}\prod_{k=i-a-1}^{b-j}\lambda_k+\frac{A_{i-2,j-1}}{A_{i-1,j}} \prod_{k=i-a-1}^{b-j}\lambda_k\right)   \\
& \hphantom{=}{} \times \left(\frac{A_{i,j+1}}{A_{i-1,j}\prod\limits_{k={i-a}}^{b-j-1}\lambda_k}+\frac{A'_{i+1,j}}{A'_{i,j-1} \prod\limits_{k={i-a}}^{b-j-1}\lambda_k}\right)^{-1}\\
&=p^*\left( \big(L'_{i,j-1}+L_{i-1,j}\big)\cdot \big(L_{i,j+1}^{-1}+L_{i+1,j}'^{-1}\big)^{-1}\right)
=p^*\left(\frac{\big(L'_{i,j-1}+L_{i-1,j}\big)L'_{i+1,j}L_{i,j+1}}{\big(L'_{i+1,j}+L_{i,j+1}\big)}\right),
\end{align*}
which is precisely the right hand side.
\end{proof}

Lemma \ref{mutation L} allows us to compute $L'_{i,j}$ recursively. Let $\Pi= (\Pi_{ij} )$ be a matrix such that its rows are numbered $0,\dots, a+1$ and its columns are numbered $0,\dots, b+1$. For $1\leq i\leq a$ and $1\leq j\leq b$, we define a \emph{birational toggling} action~$\tau_{i,j}$ sending~$\Pi$ to the matrix $\tau_{i,j}\Pi$ such that
\[
 (\tau_{i,j}\Pi )_{k,l}:= \begin{cases} \Pi_{k,l} & \text{if $(k,l)\neq (i,j)$}, \\
\displaystyle\frac{ (\Pi_{i,j-1}+\Pi_{i-1,j} )\Pi_{i+1,j}\Pi_{i,j+1}}{\Pi_{i,j} (\Pi_{i+1,j}+\Pi_{i,j+1} )} & \text{if $(k,l)=(i,j)$}.\end{cases}
\]
Recall the toggling sequence $\eta$ in \eqref{toggle}.

\begin{lem}\label{lemma.tog.cyc}
Let us apply the toggling sequence $\eta$ to the initial $(a+2)\times (b+2)$ matrix
\[
\Pi=\begin{pmatrix} P & P & P & \cdots & P & 1 \\
P & L_{1,1} & L_{1,2} & \cdots & L_{1,b} & 1 \\
\vdots & \vdots &\vdots & \ddots & \vdots & \vdots\\
P & L_{a,1} & L_{a,2}& \cdots & L_{a,b} & 1 \\
1 & 1 & 1& \cdots & 1 & 1\end{pmatrix}.
\]
Then the matrix resulting from the application of $\eta$ is
\[
\eta\Pi=\begin{pmatrix}P & P & P & \cdots & P & 1 \\
P & L'_{1,1} & L'_{1,2} & \cdots & L'_{1,n-a} & 1 \\
\vdots & \vdots &\vdots & \ddots & \vdots & \vdots\\
P & L'_{a,1} & L'_{a,2} & \cdots & L'_{a,n-a} & 1 \\
1 & 1 &1 & \cdots & 1 & 1
\end{pmatrix}.
\]
In other words, the pull-back of the Gelfand--Zetlin coordinates via the rotation~$R$ is given by~$\eta$.
\end{lem}
\begin{proof} Note that the sequence $\eta$ toggles at each (internal) entry exactly once. It suffices to prove that the step $\tau_{i,j}$ within the sequence $\eta$ changes $L_{i,j}$ to $L'_{i,j}$. This follows from the fact that the toggling sequence $\eta$ goes from bottom to top within each column and from left to right through all columns. So when we toggle at the entry $(i,j)$, the matrix entry to the left and the matrix entry below have already been changed to $L'_{i,j-1}$ and $L'_{i+1,j}$ respectively. The rest of the proof is just a straightforward comparison between the toggling formula and Lemma~\ref{mutation L}.
\end{proof}

\begin{proof}[Proof of Theorem~\ref{main9}] Set $l'_{i,j}:=(L'_{i,j})^t=\big(R^t\big)^* (l_{i,j} )$. By Lemma~\ref{lemma.tog.cyc}, with the convention that $l_{0,j}=l_{i,0}=l_{0,0}=c$ and $l_{a+1,j}=l_{i,b+1}=0$, we see that $l'_{i,j}$ can be computed recursively as
$
l'_{0,0}=l_{0,0}=c
$
and for $(i,j)\neq (0,0)$,
\begin{align}
l'_{i,j}&=\min \{l'_{i,j-1},l_{i-1,j} \}+l'_{i+1}+l_{i,j+1}-\min  \{l'_{i+1,j},l_{i,j+1} \}-l_{i,j}\nonumber\\
& =\min \{l'_{i,j-1},l_{i-1,j} \}+\max  \{l'_{i+1,j},l_{i,j+1} \}-l_{i,j}.\label{lusztig coordinate transform}
\end{align}
In the process of computing $l'_{i,j}$, the coordinates below and to the left of $l_{i,j}$ has been toggled in the way the mutation sequence is constructed. Therefore the formula~\eqref{lusztig coordinate transform} recovers the toggling formula~\eqref{toggle}. It concludes the proof of Theorem~\ref{main9}.
\end{proof}

Now let us deduce the cyclic sieving phenomenon of plane partitions claimed in Theorem~\ref{main1} from Theorem~\ref{main9}.

\begin{proof}[Proof of Theorem \ref{main1}]Finite dimensional irreducible representations of $\GL_n$ are classified by their highest weights, which are $n$-tuples of integers $\lambda=\left(\lambda_1,\dots, \lambda_n\right)$ with $\lambda_1\geq \lambda_2\geq \dots \geq \lambda_n$. Denote by $V_\lambda$ the finite dimension irreducible representation with highest weight $\lambda$ (see, e.g., \cite[Lecture 15]{FH}). Let $D(p,q)$ be the diagonal matrix $\diag\big(pq^{n-1},\dots, pq,p\big)$. We compute the trace of the action of $D(p,q)$ on $V_\lambda$ as
\[
\Tr_{V_\lambda}D(p,q)=p^{\inprod{\omega_n}{\lambda}}\sum_\mu\dim V_\lambda(\mu)q^{\inprod{\rho}{\mu}},
\]
where $V_\lambda(\mu)$ is the $\mu$-weight subspace of $V_\lambda$, $\rho$ is the dominant weight $(n-1,\dots,2, 1,0)$, and $\omega_n=(1,\dots, 1)$.

Let $\zeta={\rm e}^{2\pi \sqrt{-1}/n}$. The characteristic polynomial of the matrix $C_a$ in \eqref{cyclic.twist.ration} is
\[
\det\left(t\,\mathrm{Id}_n-C_a\right)=t^n-(-1)^{a-1}.
\]
It has $n$-distinct roots $\zeta^{-\frac{a-1}{2}},\zeta^{-\frac{a-1}{2}}\zeta, \dots, \zeta^{-\frac{a-1}{2}}\zeta^{n-1}$ over $\mathbb{C}$. Therefore $C_a$ is conjugate to $D\big(\zeta^{-\frac{a-1}{2}},\zeta\big)$ and $C_a^k$ is conjugate to $D\big(\zeta^{-\frac{(a-1)k}{2}},\zeta^k\big)$. By Theorem \ref{main9}, the number of plane partitions fixed by $\eta^k$ is equal to the number of basis vectors in $ \{\theta_{\pi} \}_{\pi\in P(a,b,c)}$ fixed by~$C_a^k$. Therefore
\begin{align*}
 \#\left\{\pi\, \Big|\, \begin{matrix}\pi \in P(a,b,c)\\
 \eta^k(\pi)=\pi\end{matrix}\right\}
& =\Tr_{V_{c\omega_a}}C_a^k  =\Tr_{V_{c\omega_a}}D\big(\zeta^{-\frac{(a-1)k}{2}},\zeta^k\big)\\
& =\big(\zeta^k\big)^{-\frac{(a-1)ac}{2}}\sum_\mu \dim V_{c\omega_a}(\mu)\big(\zeta^k\big)^{\inprod{\rho}{\mu}}.
\end{align*}
By Proposition \ref{prof1500}, we have
\[
{\dim}V_{c\omega_a}(\mu)=\# P(a,b,c)(\mu).
\]
Let $\pi\in P(a,b,c)$. By using \eqref{gz formula} and setting $\delta_0=0$ we get
\[
\inprod{\rho}{\wt(\Lambda_\pi)}=\sum_{k=1}^n(n-k) (\delta_k-\delta_{k-1} )=\sum_{k=1}^{n-1} \delta_k=\frac{(a-1)ac}{2}+\sum_{i,j}\pi_{i,j}=\frac{(a-1)ac}{2}+|\pi|.
\]
Therefore
\begin{align*}
\big(\zeta^k\big)^{-\frac{(a-1)ac}{2}}\sum_\mu\dim V_{c\omega_a}(\mu)\big(\zeta^k\big)^{\inprod{\rho}{\mu}}& =
\sum_\mu\sum_{\pi\in P(a,b,c)(\mu)}\big(\zeta^k\big)^{|\pi|}\\
&=\sum_{\pi\in P(a,b,c)}\big(\zeta^k\big)^{|\pi|} =M_{a,b,c}\big(\zeta^k\big). \tag*{\qed}
\end{align*}
\renewcommand{\qed}{}
\end{proof}

\appendix

\section{Generalities on cluster ensembles}\label{appendixA}

We briefly recall the definition of cluster ensembles following \cite{FGensemble}.

{\bf Quiver mutations.}
Let $Q=\big(I^{\uf}\subset I, \varepsilon \big)$ be a quiver without loops or 2-cycles: $I$ is the set of vertices, $I^{\uf}$ is the set of unfrozen vertices, and $\varepsilon$ is an $I\times I$ skew-symmetric matrix called the \emph{exchange matrix} encoding the data of number of arrows between vertices
\[
\varepsilon_{ij}= \# \{j\rightarrow i\} - \# \{i\rightarrow j\}.
\]
For the rest of this appendix we let $m=\#I$ and $l=\#I^{\uf}$.

Given a quiver $Q$, the {\it quiver mutation} $\mu_k$ at a non-frozen vertex $k\in I^{\uf}$ creates a~new quiver~$\mu_k(Q)$ by the following procedure
\begin{enumerate}\setlength\itemsep{0pt}
\item For each pair of arrows $i \longrightarrow k \longrightarrow j$, create a ``composite'' arrow $i \longrightarrow j$.
\item Reverse all arrows incident to $k$.
\item Remove any maximal disjoint collection of oriented 2-cycles.
\end{enumerate}
This quiver mutation is involutive: $\mu_k^2(Q)=Q$. Repeating the process at every non-frozen vertex for each new quiver obtained via quiver mutations, we get an infinite $l$-valent tree
$\mathbb{T}_l$ such that every vertex~$t$ of $\mathbb{T}_l$ is assigned a quiver $Q_t=\big(I^{\uf}\subset I, \varepsilon_t\big)$.

{\bf Cluster ensembles.}
Now assign to each vertex $t$ two coordinate charts: the $K_2$ cluster $\alpha_t=\{A_{i,t} \,|\, i \in I\}$ and the Poisson cluster $\chi_t=\{X_{i,t} \,|\, i \in I\}$. Geometrically, they correspond to a pair of algebraic tori:
\begin{equation}
\label{cluster.tori.4.45}
\mathcal{T}_{t, \alpha} = {\rm Spec}\big(\mathds{k}[A_{1,t}^\pm,\dots, A_{m,t}^\pm] \big), \qquad \mathcal{T}_{t, \chi} = {\rm Spec}\big(\mathds{k}[X_{1,t}^\pm,\dots, X_{m,t}^\pm] \big).
\end{equation}
There is a homomorphism $p$ relating them:
\begin{equation}\label{nat.bb}
p^* X_{i,t} = \prod_{j\in I} A_{j,t} ^{\varepsilon_{ji,t}}.
\end{equation}
The transition maps between the pairs of tori assigned to $Q_t$ and $Q_{t'}=\mu_k(Q_t)$ are as follows:
\begin{gather}\label{mutation.a}
\mu_k^* A_{i,t'} =  \begin{cases}
\displaystyle A_{k,t}^{-1} \left(\prod_{j \in I} A_{j,t} ^{[\varepsilon_{jk,t}]_+}+\prod_{j\in I } A_{j,t} ^{[-\varepsilon_{jk,t}]_+}\right) & \mbox{if $i=k$}, \\
 A_{i,t} & \mbox{if $i\neq k$}, \\
 \end{cases} \\
 \label{mutation.x}
\mu_k^* X_{i,t'} =  \begin{cases}
 X_{k,t}^{-1}   & \mbox{if $i=k$}, \\
 X_{i,t}\big(1+X_{k,t}^{{\rm sgn}(\varepsilon_{ik,t})}\big)^{\varepsilon_{ik,t}} & \mbox{if $i\neq k$}. \\
 \end{cases}
\end{gather}
Here $[\varepsilon]_+=\max\{\varepsilon,0\}$.

Let us glue all the algebraic tori via the transitions \eqref{mutation.a}, \eqref{mutation.x}, obtaining a pair of varieties called a cluster ensemble
\begin{equation}
\label{cluster.ensemble.def.7.19}
\mathscr{A}_{|Q|}=\bigcup_{t} \mathcal{T}_{t, \alpha}, \qquad \mathscr{X}_{|Q|}=\bigcup_{t} \mathcal{T}_{t, \chi}.
\end{equation}
where $t$ runs through all the vertices of the tree $\mathbb{T}_l$. The map $p$ in \eqref{nat.bb} is compatible with the transition maps \eqref{mutation.a}\eqref{mutation.x}. Therefore we get a natural map
\begin{equation}\label{pmap.cluster.ensemble.2.21}
p\colon \ \mathscr{A}_{|Q|} \longrightarrow \mathscr{X}_{|Q|}.
\end{equation}
In general the map $p$ is neither injective nor surjective.

The coordinate rings of these varieties are the algebras of universal Laurent polynomials
\begin{gather*}
\mathcal{O}(\mathscr{A}_{|Q|})= {\bf up}(\mathscr{A}_{|Q|}):=\bigcap_t \mathds{k}\big[A_{1,t}^\pm,\dots, A_{m,t}^\pm\big],
\\
\mathcal{O}(\mathscr{X}_{|Q|})= {\bf up}(\mathscr{X}_{|Q|}):=\bigcap_t \mathds{k}\big[X_{1,t}^\pm,\dots, X_{m,t}^\pm\big].
\end{gather*}

{\bf Cluster modular group.} To each torus $\mathcal{T}_{t, \alpha}$ in~\eqref{cluster.tori.4.45} is associated a differential form
\[
\Omega_{t}:=\sum_{i,j} \varepsilon_{ij,t}\frac{{\rm d} A_{i,t}}{A_{i,t}}\wedge \frac{{\rm d} A_{j,t}}{A_{j,t}}.
\]
This $\Omega_t$ is compatible with the transition \eqref{mutation.a} and therefore can be lifted to a global differential form $\Omega$ on $\mathscr{A}_{|Q|}$. A {\it cluster automorphism} $\tau$ of $\mathscr{A}_{|Q|}$ is a biregular isomorphism of $ \mathscr{A}_{|Q|}$ such that
\begin{itemize}\setlength\itemsep{0pt}
\item It preserves the differential form: $\tau^* \Omega = \Omega$.
\item For every $K_2$-cluster $\alpha_t=\{A_{i,t}\}$, the pullback $\tau^*(\alpha_t):= \{\tau^*A_{i,t}\}$ remains a $K_2$-cluster, with its indices $i$ possibly permuted.
\end{itemize}
Locally, $\tau$ can be realized by a sequence of mutations that sends a quiver $Q$ to itself up to permutations of vertices.
The {\it cluster modular group} $\mathcal{G}_{\mathscr{A}, |Q|}$ consists of cluster automorphisms of~$\mathscr{A}_{|Q|}$.

To each torus $\mathcal{T}_{t, \chi}$ is associated a bi-vector
\[
B_t: =\sum_{i,j} \varepsilon_{ij,t} X_{i,t} \frac{\partial }{\partial X_{i,t}}\wedge X_{j,t}\frac{\partial }{\partial X_{j,t}},
\]
which can be lifted to a global bi-vector $B$ on $\mathscr{X}_{|{Q}|}$. A cluster automorphism of $\mathscr{X}_{|Q|}$ is a~biregular isomorphism of $\mathscr{X}_{|Q|}$ that preserves the bi-vector $B$ and permutes its Poisson clusters. Denote by $\mathcal{G}_{\mathscr{X}, |Q|}$ the group of cluster automorphisms of $\mathscr{X}_{|Q|}$.

From the tropical cluster duality proved by Nakanishi and Zelevinsky \cite{NZ} one can deduce that the cluster modular group $\mathcal{G}_{\mathscr{X}, |Q|}= \mathcal{G}_{\mathscr{A}, |Q|}$. Hence, we will drop the subscripts $\mathscr{A}$ and $\mathscr{X}$ in the notation and denote it by $\mathcal{G}_{|Q|}$.

{\bf Quiver extensions.}
Let $Q=\big(I^{\uf} \subset I, \varepsilon\big)$ be a quiver. Let $\widetilde{Q}=\big({I}^{\uf}\subset \tilde{I},\tilde{\varepsilon}\big)$ be a quiver obtained from $Q$ by adding frozen vertices labelled by $I'=\{1',\dots, f'\}$ and arrows such that $\widetilde{Q}$ contains $Q$ as a full subquiver. In other words, $\tilde{I}= I \sqcup I'$ and $\tilde{\varepsilon}$ contains $\varepsilon$ as a submatrix. Let $\big(\mathscr{A}_{|\widetilde{Q}|}, \mathscr{X}_{|\widetilde{Q}|}\big)$ be the cluster ensemble associated to $\widetilde{Q}$.

Following \cite[equation~(3.5)]{Sh}, we define the following map
\[
k\colon \  \mathscr{A}_{|\widetilde{Q}|}\stackrel{\tilde{p}}{\longrightarrow}\mathscr{X}_{|\widetilde{Q}|}\stackrel{j}{\longrightarrow}\mathscr{X}_{|Q|}
\]
The map $\tilde{p}$ is a natural map as \eqref{pmap.cluster.ensemble.2.21}.
The map $j$ is a surjective map such that $j^*X_{i,t}=X_{i,t}$ for all $i\in I$.
The map $k$ is the composition of $\tilde{p}$ and $j$. It is surjective if and only if the sub\-mat\-rix~$ \tilde{\varepsilon} |_{\tilde{I}\times I}$ of the exchange matrix $\tilde{\varepsilon}$ is of full rank. In this case we get a natural injection
\[
k^*\colon \  {\bf up} \big(\mathscr{X}_{|Q|} \big) \longrightarrow {\bf up}\big( \mathscr{A}_{|\widetilde{Q}|}\big).
\]

The following easy Lemma generalizes one key part of the proof of Theorem~\ref{2018.3.2.15.06}.
\begin{lem}\label{tech.lem.545}
Assume $k$ is surjective. Let $F\in \mathds{k}\big(\mathscr{X}_{|Q|}\big)$ be a rational function on $\mathscr{X}_{|Q|}$. Then $F\in {\bf up}\big(\mathscr{X}_{|Q|}\big)$ if and only if $k^*(F)\in {\bf up}\big( \mathscr{A}_{|\widetilde{Q}|}\big)$.
 \end{lem}

 \begin{proof} Let $t$ be a vertex of $\mathbb{T}_l$. Let $\chi_t= \{X_{i,t} \,|\, i\in I \}$ and $\widetilde{\alpha}_t:= \big\{A_{j,t} \,|\, j\in \tilde{I}  \big\}$ be its corresponding clusters. Since $k^*$ is injective and it maps Laurent monomials to Laurent monomials, we have
 \[
 F\in \mathds{k}\big[X_{1,t}^\pm, \dots, X_{m,t}^\pm\big]\quad \Longleftrightarrow \quad k^*(F) \in \mathds{k}\big[A_{1,t}^\pm, \dots, A_{m,t}^\pm, A_{1',t}^\pm, \dots, A_{f',t}^\pm\big].
 \]
 By definition ${\bf up} (\mathscr{X}_{|Q|} )$ and ${\bf up}\big( \mathscr{A}_{|\widetilde{Q}|}\big)$ are the intersections of Laurent polynomial rings. The lemma follows directly.
 \end{proof}

\section{Generalities on reduced plabic graphs}\label{appendixB}

Reduced plabic graphs were first introduced by Postnikov in~\cite{Pos}. We briefly recall the definition and basic constructions of reduced plabic graphs here, and we mainly follow the convention used in~\cite{Weng}.

Let $D_n$ be a disk with $n$ marked points on its boundary labeled $1,\dots, n$ in the clockwise direction. Let $\Gamma$ be a bipartite graph embedded into $D_n$ with a single edge connected to every boundary marked point. We draw \emph{zig-zag} strands on $\Gamma$ by drawing the following pattern around each vertex according to the color of the vertex:
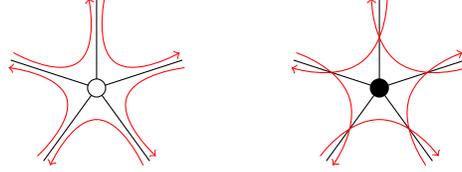
\[
\begin{tikzpicture}[scale=0.6]
\foreach \i in {0,...,4}
 {
 \draw (0,0) -- +(90-72*\i:2);
 \coordinate (i\i) at +(85-72*\i:2);
 \coordinate (m\i) at +(54-72*\i:0.7);
 \coordinate (o\i) at +(23-72*\i:2);
 \draw [red,->] plot [smooth, tension=1] coordinates {(i\i)(m\i)(o\i)};
 }
\draw [fill=white] (0,0) circle [radius=0.2];
\end{tikzpicture}
\quad \quad \quad \quad
\begin{tikzpicture}[scale=0.6]
\foreach \i in {0,...,4}
 {
 \draw (0,0) -- +(90-72*\i:2);
 \coordinate (i\i) at +(85-72*\i:2);
 \coordinate (m\i) at +(126-72*\i:0.7);
 \coordinate (o\i) at +(167-72*\i:2);
 \draw [red,->] plot [smooth, tension=1] coordinates {(i\i)(m\i)(o\i)};
 }
\draw [fill] (0,0) circle [radius=0.2];
\end{tikzpicture}
\]

\begin{defn} A rank $a$ \emph{reduced plabic graph} on $D_n$ is a bipartite graph with trivalent black vertices whose zig-zag strands do not self-intersect or form parallel bigons and go from $i$ to $i+a$ for every boundary marked point~$i$.\footnote{This definition is equivalent to Postnikov's original definition, but the bipartite and trivalent conditions help reduce the types of equivalence transformations to two.}
\end{defn}

\begin{defn}\label{dominating set} Let $\Gamma$ be a rank $a$ reduced plabic graph on $D_n$.
Connected components of the complement of $\Gamma$ are called {\it faces}. A \emph{boundary face} is a face that contains part of the boundary of $D_n$.
Let $\zeta_i$ be the zig-zag strand going from the boundary marked point $i$ to $i+a$. A face $f$ is said to be \emph{dominated} by $\zeta_i$ if it lies to the left of $\zeta_i$ with respect to the orientation of $\zeta_i$.
 The \emph{dominating set} $I(f)$ of $f$ is a collection of the indices of zig-zag strands that dominate $f$.
 \end{defn}

It is know that every dominating set is of size $a$. The set $\{\Delta_{I(f)}\}$ of Pl\"{u}cker coordinates with~$f$ runs through faces of $\Gamma$ forms a $K_2$ cluster chart on $\dGr_a^\times(n)$ associated to $Q_{\Gamma}$, where $Q_\Gamma$ is the quiver determined by $\Gamma$ as in Section~\ref{clustersec1712}. For example, the cluster $\{\Delta_{I(i,j)}\}$ associated to~$Q_{a,n}$ in~\eqref{aij} is defined in this way.

There are two types of transformations on reduced plabic graphs $\Gamma$ called \emph{$2$-by-$2$} moves.
\begin{itemize}
 \item {\bf Type I}
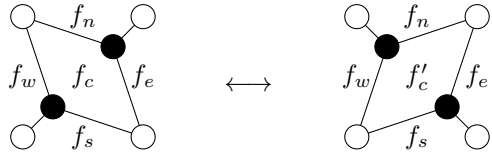
\[
 \quad \quad
\begin{tikzpicture}[scale=0.8,baseline=4ex]
\node at (1,1) [] {$f_c$};
\node at (1,0) [] {$f_s$};
\node at (0,1) [] {$f_w$};
\node at (2,1) [] {$f_e$};
\node at (1,2) [] {$f_n$};
\draw (0,0) -- (0.5,0.5) -- (0,2);
\draw (2,0) -- (1.5,1.5) -- (2,2);
\draw (1.5,1.5) -- (0,2);
\draw (0.5,0.5) --(2,0);
\draw [fill=white] (0,0) circle [radius=0.2];
\draw [fill=white] (2,0) circle [radius=0.2];
\draw [fill=white] (2,2) circle [radius=0.2];
\draw [fill=white] (0,2) circle [radius=0.2];
\draw [fill] (0.5,0.5) circle [radius=0.2];
\draw [fill] (1.5,1.5) circle [radius=0.2];
\end{tikzpicture}
\quad \quad \longleftrightarrow \quad \quad
\begin{tikzpicture}[scale=0.8,baseline=4ex]
\node at (1,1) [] {$f_c'$};
\node at (1,0) [] {$f_s$};
\node at (0,1) [] {$f_w$};
\node at (2,1) [] {$f_e$};
\node at (1,2) [] {$f_n$};
\draw (0,0) -- (0.5,1.5) -- (0,2);
\draw (2,0) -- (1.5,0.5) -- (2,2);
\draw (0.5,1.5) -- (2,2);
\draw (1.5,0.5) --(0,0);
\draw [fill=white] (0,0) circle [radius=0.2];
\draw [fill=white] (2,0) circle [radius=0.2];
\draw [fill=white] (2,2) circle [radius=0.2];
\draw [fill=white] (0,2) circle [radius=0.2];
\draw [fill] (1.5,0.5) circle [radius=0.2];
\draw [fill] (0.5,1.5) circle [radius=0.2];
\end{tikzpicture}
\]
 The dominating sets associated to the 5 related faces on picture are of the forms
\begin{gather*}
I(f_c)= J\cup \{i,k\}, \qquad I(f_s) = J \cup \{k,l\}, \qquad I(f_w)= J\cup \{i,l\},\\
 I(f_e)= J\cup \{i,j\}, \qquad I(f_n)= J\cup \{j,k\},
\end{gather*}
where $1\leq i<j<k<l\leq n$ and $J$ is an $(a-2)$-element subset of $\{1,\dots, n\}\setminus\{i,j,k,l\}$.
The type I 2-by-2 move changes the dominating set associated to the central face to
\[
I(f_c')= J\cup\{j,l\},
\]
and keeps the rest intact. Note that one has the Pl\"ucker relation
\begin{equation}\label{plucker609}
\Delta_{I(f_c)}\Delta_{I(f_c')}=\Delta_{I(f_s)}{\Delta_{I(f_n)}+\Delta_{I(f_e)}\Delta_{I(f_w)}}.
\end{equation}
On the quiver level, it corresponds to the quiver mutation on $Q_\Gamma$ at the vertex assigned to the central face, which locally is as follows
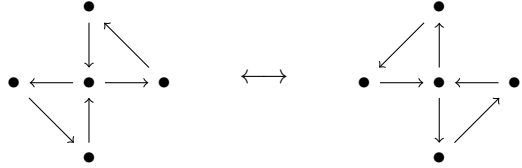
\[
\begin{tikzpicture}[baseline=0ex]
\node (0-0) at (0,0) [] {$\bullet$};
\node (0-1) at (0,1) [] {$\bullet$};
\node (1-0) at (1,0) [] {$\bullet$};
\node (0-2) at (0,-1) [] {$\bullet$};
\node (2-0) at (-1,0) [] {$\bullet$};
\draw [<-] (0-0) -- (0-1);
\draw [<-] (2-0) -- (0-0);
\draw [<-] (1-0) -- (0-0);
\draw [<-] (0-0) -- (0-2);
\draw [->] (1-0) -- (0-1);
\draw [->] (2-0) -- (0-2);
\end{tikzpicture}
\quad \quad \longleftrightarrow \quad \quad
\begin{tikzpicture}[baseline=0ex]
\node (0-0) at (0,0) [] {$\bullet$};
\node (0-1) at (0,1) [] {$\bullet$};
\node (1-0) at (1,0) [] {$\bullet$};
\node (0-2) at (0,-1) [] {$\bullet$};
\node (2-0) at (-1,0) [] {$\bullet$};
\draw [->] (0-0) -- (0-1);
\draw [->] (2-0) -- (0-0);
\draw [->] (1-0) -- (0-0);
\draw [->] (0-0) -- (0-2);
\draw [<-] (1-0) -- (0-2);
\draw [<-] (2-0) -- (0-1);
\end{tikzpicture}
\]
Therefore the Pl\"{u}cker relation \eqref{plucker609} is compatible with the cluster mutation \eqref{cluster.tori.4.45}.
\item {\bf Type II}.
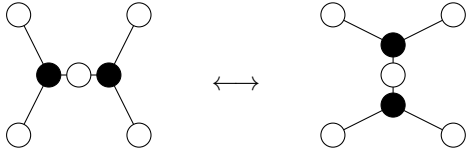
\[
\begin{tikzpicture}[scale=0.8,baseline=4ex]
\draw (0,0) -- (0.5,1) -- (1.5,1) -- (2,0);
\draw (0,2) -- (0.5,1);
\draw (1.5,1) -- (2,2);
\draw [fill=white] (0,0) circle [radius=0.2];
\draw [fill=white] (2,0) circle [radius=0.2];
\draw [fill=white] (0,2) circle [radius=0.2];
\draw [fill=white] (2,2) circle [radius=0.2];
\draw [fill=white] (1,1) circle [radius=0.2];
\draw [fill] (0.5,1) circle [radius=0.2];
\draw [fill] (1.5,1) circle [radius=0.2];
\end{tikzpicture} \quad \quad \longleftrightarrow \quad \quad \begin{tikzpicture}[scale=0.8,baseline=4ex]
\draw (0,0) -- (1,0.5) -- (1,1.5) -- (0,2);
\draw (2,0) -- (1,0.5);
\draw (1,1.5) -- (2,2);
\draw [fill=white] (0,0) circle [radius=0.2];
\draw [fill=white] (2,0) circle [radius=0.2];
\draw [fill=white] (0,2) circle [radius=0.2];
\draw [fill=white] (2,2) circle [radius=0.2];
\draw [fill=white] (1,1) circle [radius=0.2];
\draw [fill] (1,0.5) circle [radius=0.2];
\draw [fill] (1,1.5) circle [radius=0.2];
\end{tikzpicture}
\]
A type II 2-by-2 move changes neither the quiver nor the dominating sets of faces.
\end{itemize}

A result of Thurston \cite[Theorem~6]{Thu} can be restated in the reduced plabic graph language as saying that any two rank $a$ reduced plabic graphs on $D_n$ can be transformed into one another via a sequence of 2-by-2 moves of the above two types.
In conclusion, all the $K_2$ cluster structures on $\dGr_a^\times(n)$ defined by rank $a$ reduced plabic graphs on $D_n$ are equivalent.

\section[Example of a reduced plabic graph transforming under $\rho$]{Example of a reduced plabic graph transforming under $\boldsymbol{\rho}$}

Below is an example showing how a reduced plabic graph transforms under the rotation mutation sequence $\rho$. This example can be easily generalized to other standard reduced plabic graphs $\Gamma_{a,n}$ with arbitrary parameters $(a,n)$. The example we choose to do is with parameters $(a=3$, $n=7)$.
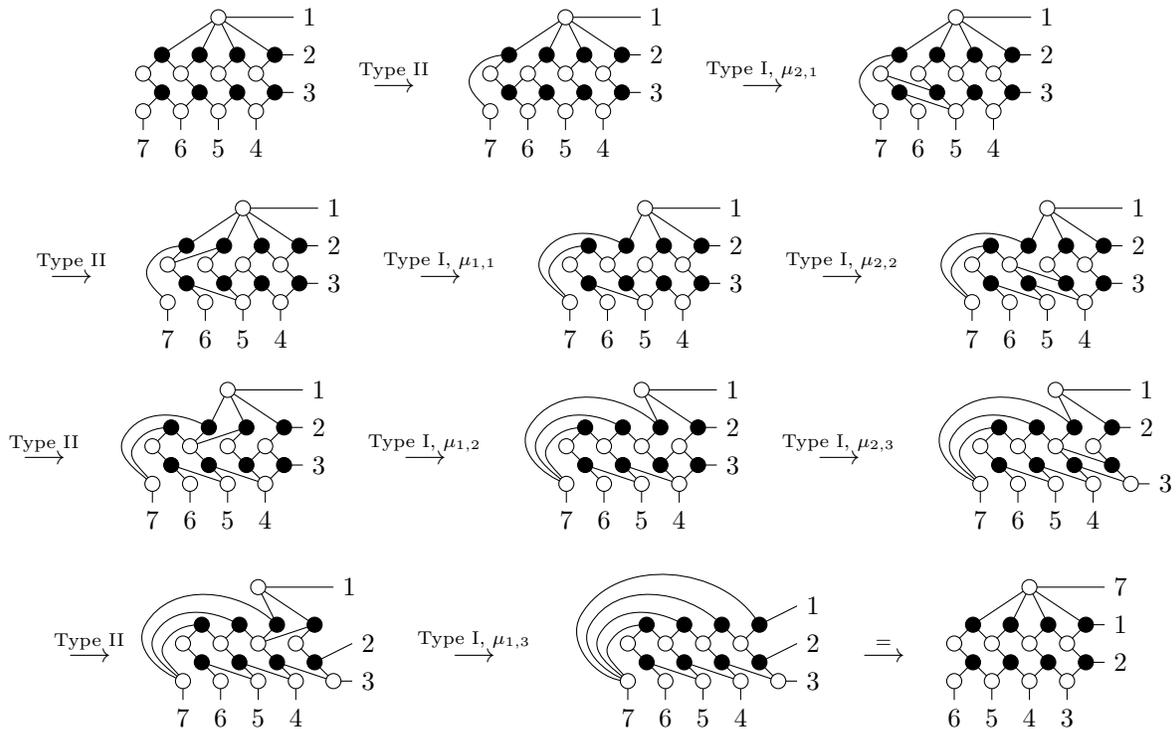
\[
\begin{tikzpicture}[scale=0.5, baseline=5ex]
\draw (3,3.5) -- (1.5,2.5);
\draw (3,3.5) -- (2.5,2.5);
\draw (3,3.5) -- (3.5,2.5);
\draw (3,3.5) -- (4.5,2.5);
\draw (1.5,2.5) -- (1,2);
\draw (2.5,2.5) -- (2,2);
\draw (3.5,2.5) -- (3,2);
\draw (4.5,2.5) -- (4,2);
\draw (1.5,2.5) -- (2,2);
\draw (2.5,2.5) -- (3,2);
\draw (3.5,2.5) -- (4,2);
\draw (1,2) -- (1.5,1.5);
\draw (2,2) -- (2.5,1.5);
\draw (3,2) -- (3.5,1.5);
\draw (4,2) -- (4.5,1.5);
\draw (1.5,1.5) -- (1,1);
\draw (2.5,1.5) -- (2,1);
\draw (3.5,1.5) -- (3,1);
\draw (4.5,1.5) -- (4,1);
\draw (1.5,1.5) -- (2,1);
\draw (2.5,1.5) -- (3,1);
\draw (3.5,1.5) -- (4,1);
\draw (3,3.5) -- (5,3.5) node [right] {$1$};
\draw (4.5,2.5) -- (5,2.5) node [right] {$2$};
\draw (4.5,1.5) -- (5,1.5) node [right] {$3$};
\draw (4,1) -- (4,0.5) node [below] {$4$};
\draw (3,1) -- (3,0.5) node [below] {$5$};
\draw (2,1) -- (2,0.5) node [below] {$6$};
\draw (1,1) -- (1,0.5) node [below] {$7$};
\draw [fill=white] (1,1) circle [radius=0.2];
\draw [fill=white] (2,1) circle [radius=0.2];
\draw [fill=white] (3,1) circle [radius=0.2];
\draw [fill=white] (4,1) circle [radius=0.2];
\draw [fill] (1.5,1.5) circle [radius=0.2];
\draw [fill] (2.5,1.5) circle [radius=0.2];
\draw [fill] (3.5,1.5) circle [radius=0.2];
\draw [fill] (4.5,1.5) circle [radius=0.2];
\draw [fill=white] (1,2) circle [radius=0.2];
\draw [fill=white] (2,2) circle [radius=0.2];
\draw [fill=white] (3,2) circle [radius=0.2];
\draw [fill=white] (4,2) circle [radius=0.2];
\draw [fill] (1.5,2.5) circle [radius=0.2];
\draw [fill] (2.5,2.5) circle [radius=0.2];
\draw [fill] (3.5,2.5) circle [radius=0.2];
\draw [fill] (4.5,2.5) circle [radius=0.2];
\draw [fill=white] (3,3.5) circle [radius=0.2];
\end{tikzpicture}
\quad \overset{\text{Type II}}{\longrightarrow} \quad
\begin{tikzpicture}[scale=0.5, baseline=5ex]
\draw (3,3.5) -- (1.5,2.5);
\draw (3,3.5) -- (2.5,2.5);
\draw (3,3.5) -- (3.5,2.5);
\draw (3,3.5) -- (4.5,2.5);
\draw (1.5,2.5) -- (1,2);
\draw (2.5,2.5) -- (2,2);
\draw (3.5,2.5) -- (3,2);
\draw (4.5,2.5) -- (4,2);
\draw (1.5,2.5) to [out=180, in=60] (0.5,2) to [out=-120, in=135] (1,1);
\draw (2.5,2.5) -- (3,2);
\draw (3.5,2.5) -- (4,2);
\draw (1,2) -- (1.5,1.5);
\draw (2,2) -- (2.5,1.5);
\draw (3,2) -- (3.5,1.5);
\draw (4,2) -- (4.5,1.5);
\draw (1.5,1.5) -- (2,2);
\draw (2.5,1.5) -- (2,1);
\draw (3.5,1.5) -- (3,1);
\draw (4.5,1.5) -- (4,1);
\draw (1.5,1.5) -- (2,1);
\draw (2.5,1.5) -- (3,1);
\draw (3.5,1.5) -- (4,1);
\draw (3,3.5) -- (5,3.5) node [right] {$1$};
\draw (4.5,2.5) -- (5,2.5) node [right] {$2$};
\draw (4.5,1.5) -- (5,1.5) node [right] {$3$};
\draw (4,1) -- (4,0.5) node [below] {$4$};
\draw (3,1) -- (3,0.5) node [below] {$5$};
\draw (2,1) -- (2,0.5) node [below] {$6$};
\draw (1,1) -- (1,0.5) node [below] {$7$};
\draw [fill=white] (1,1) circle [radius=0.2];
\draw [fill=white] (2,1) circle [radius=0.2];
\draw [fill=white] (3,1) circle [radius=0.2];
\draw [fill=white] (4,1) circle [radius=0.2];
\draw [fill] (1.5,1.5) circle [radius=0.2];
\draw [fill] (2.5,1.5) circle [radius=0.2];
\draw [fill] (3.5,1.5) circle [radius=0.2];
\draw [fill] (4.5,1.5) circle [radius=0.2];
\draw [fill=white] (1,2) circle [radius=0.2];
\draw [fill=white] (2,2) circle [radius=0.2];
\draw [fill=white] (3,2) circle [radius=0.2];
\draw [fill=white] (4,2) circle [radius=0.2];
\draw [fill] (1.5,2.5) circle [radius=0.2];
\draw [fill] (2.5,2.5) circle [radius=0.2];
\draw [fill] (3.5,2.5) circle [radius=0.2];
\draw [fill] (4.5,2.5) circle [radius=0.2];
\draw [fill=white] (3,3.5) circle [radius=0.2];
\end{tikzpicture}
\quad\!
\overset{\text{Type I, $\mu_{2,1}$}}{\longrightarrow} \!\quad
\begin{tikzpicture}[scale=0.5, baseline=5ex]
\draw (3,3.5) -- (1.5,2.5);
\draw (3,3.5) -- (2.5,2.5);
\draw (3,3.5) -- (3.5,2.5);
\draw (3,3.5) -- (4.5,2.5);
\draw (1.5,2.5) -- (1,2);
\draw (2.5,2.5) -- (2,2);
\draw (3.5,2.5) -- (3,2);
\draw (4.5,2.5) -- (4,2);
\draw (1.5,2.5) to [out=180, in=60] (0.5,2) to [out=-120, in=135] (1,1);
\draw (2.5,2.5) -- (3,2);
\draw (3.5,2.5) -- (4,2);
\draw (1,2) -- (1.5,1.5);
\draw (2,2) -- (2.5,1.5);
\draw (3,2) -- (3.5,1.5);
\draw (4,2) -- (4.5,1.5);
\draw (1.5,1.5) -- (3,1);
\draw (2.5,1.5) -- (1,2);
\draw (3.5,1.5) -- (3,1);
\draw (4.5,1.5) -- (4,1);
\draw (1.5,1.5) -- (2,1);
\draw (2.5,1.5) -- (3,1);
\draw (3.5,1.5) -- (4,1);
\draw (3,3.5) -- (5,3.5) node [right] {$1$};
\draw (4.5,2.5) -- (5,2.5) node [right] {$2$};
\draw (4.5,1.5) -- (5,1.5) node [right] {$3$};
\draw (4,1) -- (4,0.5) node [below] {$4$};
\draw (3,1) -- (3,0.5) node [below] {$5$};
\draw (2,1) -- (2,0.5) node [below] {$6$};
\draw (1,1) -- (1,0.5) node [below] {$7$};
\draw [fill=white] (1,1) circle [radius=0.2];
\draw [fill=white] (2,1) circle [radius=0.2];
\draw [fill=white] (3,1) circle [radius=0.2];
\draw [fill=white] (4,1) circle [radius=0.2];
\draw [fill] (1.5,1.5) circle [radius=0.2];
\draw [fill] (2.5,1.5) circle [radius=0.2];
\draw [fill] (3.5,1.5) circle [radius=0.2];
\draw [fill] (4.5,1.5) circle [radius=0.2];
\draw [fill=white] (1,2) circle [radius=0.2];
\draw [fill=white] (2,2) circle [radius=0.2];
\draw [fill=white] (3,2) circle [radius=0.2];
\draw [fill=white] (4,2) circle [radius=0.2];
\draw [fill] (1.5,2.5) circle [radius=0.2];
\draw [fill] (2.5,2.5) circle [radius=0.2];
\draw [fill] (3.5,2.5) circle [radius=0.2];
\draw [fill] (4.5,2.5) circle [radius=0.2];
\draw [fill=white] (3,3.5) circle [radius=0.2];
\end{tikzpicture}
\]
\[
\overset{\text{Type II}}{\longrightarrow} \quad\!
\begin{tikzpicture}[scale=0.5, baseline=5ex]
\draw (3,3.5) -- (1.5,2.5);
\draw (3,3.5) -- (2.5,2.5);
\draw (3,3.5) -- (3.5,2.5);
\draw (3,3.5) -- (4.5,2.5);
\draw (1.5,2.5) -- (1,2);
\draw (2.5,2.5) -- (2,2);
\draw (3.5,2.5) -- (3,2);
\draw (4.5,2.5) -- (4,2);
\draw (1.5,2.5) to [out=180, in=60] (0.5,2) to [out=-120, in=135] (1,1);
\draw (2.5,2.5) -- (1,2);
\draw (3.5,2.5) -- (4,2);
\draw (1,2) -- (1.5,1.5);
\draw (2,2) -- (2.5,1.5);
\draw (3,2) -- (3.5,1.5);
\draw (4,2) -- (4.5,1.5);
\draw (1.5,1.5) -- (3,1);
\draw (3,2) -- (2.5,1.5);
\draw (3.5,1.5) -- (3,1);
\draw (4.5,1.5) -- (4,1);
\draw (1.5,1.5) -- (2,1);
\draw (2.5,1.5) -- (3,1);
\draw (3.5,1.5) -- (4,1);
\draw (3,3.5) -- (5,3.5) node [right] {$1$};
\draw (4.5,2.5) -- (5,2.5) node [right] {$2$};
\draw (4.5,1.5) -- (5,1.5) node [right] {$3$};
\draw (4,1) -- (4,0.5) node [below] {$4$};
\draw (3,1) -- (3,0.5) node [below] {$5$};
\draw (2,1) -- (2,0.5) node [below] {$6$};
\draw (1,1) -- (1,0.5) node [below] {$7$};
\draw [fill=white] (1,1) circle [radius=0.2];
\draw [fill=white] (2,1) circle [radius=0.2];
\draw [fill=white] (3,1) circle [radius=0.2];
\draw [fill=white] (4,1) circle [radius=0.2];
\draw [fill] (1.5,1.5) circle [radius=0.2];
\draw [fill] (2.5,1.5) circle [radius=0.2];
\draw [fill] (3.5,1.5) circle [radius=0.2];
\draw [fill] (4.5,1.5) circle [radius=0.2];
\draw [fill=white] (1,2) circle [radius=0.2];
\draw [fill=white] (2,2) circle [radius=0.2];
\draw [fill=white] (3,2) circle [radius=0.2];
\draw [fill=white] (4,2) circle [radius=0.2];
\draw [fill] (1.5,2.5) circle [radius=0.2];
\draw [fill] (2.5,2.5) circle [radius=0.2];
\draw [fill] (3.5,2.5) circle [radius=0.2];
\draw [fill] (4.5,2.5) circle [radius=0.2];
\draw [fill=white] (3,3.5) circle [radius=0.2];
\end{tikzpicture}
\quad\!
\overset{\text{Type I, $\mu_{1,1}$}}{\longrightarrow} \quad\!
\begin{tikzpicture}[scale=0.5, baseline=5ex]
\draw (3,3.5) -- (2.5,2.5);
\draw (3,3.5) -- (3.5,2.5);
\draw (3,3.5) -- (4.5,2.5);
\draw (1.5,2.5) -- (1,2);
\draw (2.5,2.5) -- (2,2);
\draw (3.5,2.5) -- (3,2);
\draw (4.5,2.5) -- (4,2);
\draw (1.5,2.5) -- (2,2);
\draw (1.5,2.5) to [out=180, in=60] (0.5,2) to [out=-120, in=120] (1,1);
\draw (2.5,2.5) to [out=150, in=60] (0.25,2.25) to [out=-120, in=140] (1,1);
\draw (3.5,2.5) -- (4,2);
\draw (1,2) -- (1.5,1.5);
\draw (2,2) -- (2.5,1.5);
\draw (3,2) -- (3.5,1.5);
\draw (4,2) -- (4.5,1.5);
\draw (1.5,1.5) -- (3,1);
\draw (3,2) -- (2.5,1.5);
\draw (3.5,1.5) -- (3,1);
\draw (4.5,1.5) -- (4,1);
\draw (1.5,1.5) -- (2,1);
\draw (2.5,1.5) -- (3,1);
\draw (3.5,1.5) -- (4,1);
\draw (3,3.5) -- (5,3.5) node [right] {$1$};
\draw (4.5,2.5) -- (5,2.5) node [right] {$2$};
\draw (4.5,1.5) -- (5,1.5) node [right] {$3$};
\draw (4,1) -- (4,0.5) node [below] {$4$};
\draw (3,1) -- (3,0.5) node [below] {$5$};
\draw (2,1) -- (2,0.5) node [below] {$6$};
\draw (1,1) -- (1,0.5) node [below] {$7$};
\draw [fill=white] (1,1) circle [radius=0.2];
\draw [fill=white] (2,1) circle [radius=0.2];
\draw [fill=white] (3,1) circle [radius=0.2];
\draw [fill=white] (4,1) circle [radius=0.2];
\draw [fill] (1.5,1.5) circle [radius=0.2];
\draw [fill] (2.5,1.5) circle [radius=0.2];
\draw [fill] (3.5,1.5) circle [radius=0.2];
\draw [fill] (4.5,1.5) circle [radius=0.2];
\draw [fill=white] (1,2) circle [radius=0.2];
\draw [fill=white] (2,2) circle [radius=0.2];
\draw [fill=white] (3,2) circle [radius=0.2];
\draw [fill=white] (4,2) circle [radius=0.2];
\draw [fill] (1.5,2.5) circle [radius=0.2];
\draw [fill] (2.5,2.5) circle [radius=0.2];
\draw [fill] (3.5,2.5) circle [radius=0.2];
\draw [fill] (4.5,2.5) circle [radius=0.2];
\draw [fill=white] (3,3.5) circle [radius=0.2];
\end{tikzpicture}
\quad\! \overset{\text{Type I, $\mu_{2,2}$}}{\longrightarrow} \quad\!
\begin{tikzpicture}[scale=0.5, baseline=5ex]
\draw (3,3.5) -- (2.5,2.5);
\draw (3,3.5) -- (3.5,2.5);
\draw (3,3.5) -- (4.5,2.5);
\draw (1.5,2.5) -- (1,2);
\draw (2.5,2.5) -- (2,2);
\draw (3.5,2.5) -- (3,2);
\draw (4.5,2.5) -- (4,2);
\draw (1.5,2.5) -- (2,2);
\draw (1.5,2.5) to [out=180, in=60] (0.5,2) to [out=-120, in=120] (1,1);
\draw (2.5,2.5) to [out=150, in=60] (0.25,2.25) to [out=-120, in=140] (1,1);
\draw (3.5,2.5) -- (4,2);
\draw (1,2) -- (1.5,1.5);
\draw (2,2) -- (2.5,1.5);
\draw (3,2) -- (3.5,1.5);
\draw (4,2) -- (4.5,1.5);
\draw (1.5,1.5) -- (3,1);
\draw (2,2) -- (3.5,1.5);
\draw (2.5,1.5) -- (4,1);
\draw (4.5,1.5) -- (4,1);
\draw (1.5,1.5) -- (2,1);
\draw (2.5,1.5) -- (3,1);
\draw (3.5,1.5) -- (4,1);
\draw (3,3.5) -- (5,3.5) node [right] {$1$};
\draw (4.5,2.5) -- (5,2.5) node [right] {$2$};
\draw (4.5,1.5) -- (5,1.5) node [right] {$3$};
\draw (4,1) -- (4,0.5) node [below] {$4$};
\draw (3,1) -- (3,0.5) node [below] {$5$};
\draw (2,1) -- (2,0.5) node [below] {$6$};
\draw (1,1) -- (1,0.5) node [below] {$7$};
\draw [fill=white] (1,1) circle [radius=0.2];
\draw [fill=white] (2,1) circle [radius=0.2];
\draw [fill=white] (3,1) circle [radius=0.2];
\draw [fill=white] (4,1) circle [radius=0.2];
\draw [fill] (1.5,1.5) circle [radius=0.2];
\draw [fill] (2.5,1.5) circle [radius=0.2];
\draw [fill] (3.5,1.5) circle [radius=0.2];
\draw [fill] (4.5,1.5) circle [radius=0.2];
\draw [fill=white] (1,2) circle [radius=0.2];
\draw [fill=white] (2,2) circle [radius=0.2];
\draw [fill=white] (3,2) circle [radius=0.2];
\draw [fill=white] (4,2) circle [radius=0.2];
\draw [fill] (1.5,2.5) circle [radius=0.2];
\draw [fill] (2.5,2.5) circle [radius=0.2];
\draw [fill] (3.5,2.5) circle [radius=0.2];
\draw [fill] (4.5,2.5) circle [radius=0.2];
\draw [fill=white] (3,3.5) circle [radius=0.2];
\end{tikzpicture}
\]
\[
\overset{\text{Type II}}{\longrightarrow} \quad\!
\begin{tikzpicture}[scale=0.5, baseline=5ex]
\draw (3,3.5) -- (2.5,2.5);
\draw (3,3.5) -- (3.5,2.5);
\draw (3,3.5) -- (4.5,2.5);
\draw (1.5,2.5) -- (1,2);
\draw (2.5,2.5) -- (2,2);
\draw (3.5,2.5) -- (3,2);
\draw (4.5,2.5) -- (4,2);
\draw (1.5,2.5) -- (2,2);
\draw (1.5,2.5) to [out=180, in=60] (0.5,2) to [out=-120, in=120] (1,1);
\draw (2.5,2.5) to [out=150, in=60] (0.25,2.25) to [out=-120, in=140] (1,1);
\draw (3.5,2.5) -- (2,2);
\draw (1,2) -- (1.5,1.5);
\draw (2,2) -- (2.5,1.5);
\draw (3,2) -- (3.5,1.5);
\draw (4,2) -- (4.5,1.5);
\draw (1.5,1.5) -- (3,1);
\draw (4,2) -- (3.5,1.5);
\draw (2.5,1.5) -- (4,1);
\draw (4.5,1.5) -- (4,1);
\draw (1.5,1.5) -- (2,1);
\draw (2.5,1.5) -- (3,1);
\draw (3.5,1.5) -- (4,1);
\draw (3,3.5) -- (5,3.5) node [right] {$1$};
\draw (4.5,2.5) -- (5,2.5) node [right] {$2$};
\draw (4.5,1.5) -- (5,1.5) node [right] {$3$};
\draw (4,1) -- (4,0.5) node [below] {$4$};
\draw (3,1) -- (3,0.5) node [below] {$5$};
\draw (2,1) -- (2,0.5) node [below] {$6$};
\draw (1,1) -- (1,0.5) node [below] {$7$};
\draw [fill=white] (1,1) circle [radius=0.2];
\draw [fill=white] (2,1) circle [radius=0.2];
\draw [fill=white] (3,1) circle [radius=0.2];
\draw [fill=white] (4,1) circle [radius=0.2];
\draw [fill] (1.5,1.5) circle [radius=0.2];
\draw [fill] (2.5,1.5) circle [radius=0.2];
\draw [fill] (3.5,1.5) circle [radius=0.2];
\draw [fill] (4.5,1.5) circle [radius=0.2];
\draw [fill=white] (1,2) circle [radius=0.2];
\draw [fill=white] (2,2) circle [radius=0.2];
\draw [fill=white] (3,2) circle [radius=0.2];
\draw [fill=white] (4,2) circle [radius=0.2];
\draw [fill] (1.5,2.5) circle [radius=0.2];
\draw [fill] (2.5,2.5) circle [radius=0.2];
\draw [fill] (3.5,2.5) circle [radius=0.2];
\draw [fill] (4.5,2.5) circle [radius=0.2];
\draw [fill=white] (3,3.5) circle [radius=0.2];
\end{tikzpicture}
\quad\! \overset{\text{Type I, $\mu_{1,2}$}}{\longrightarrow} \quad\!
\begin{tikzpicture}[scale=0.5, baseline=5ex]
\draw (2.5,2.5) -- (3,2);
\draw (3,3.5) -- (3.5,2.5);
\draw (3,3.5) -- (4.5,2.5);
\draw (1.5,2.5) -- (1,2);
\draw (2.5,2.5) -- (2,2);
\draw (3.5,2.5) -- (3,2);
\draw (4.5,2.5) -- (4,2);
\draw (1.5,2.5) -- (2,2);
\draw (1.5,2.5) to [out=180, in=60] (0.5,2) to [out=-120, in=120] (1,1);
\draw (2.5,2.5) to [out=150, in=60] (0.25,2.25) to [out=-120, in=140] (1,1);
\draw (3.5,2.5) to [out=135, in=60] (0,2.5) to [out=-120, in=160] (1,1);
\draw (1,2) -- (1.5,1.5);
\draw (2,2) -- (2.5,1.5);
\draw (3,2) -- (3.5,1.5);
\draw (4,2) -- (4.5,1.5);
\draw (1.5,1.5) -- (3,1);
\draw (4,2) -- (3.5,1.5);
\draw (2.5,1.5) -- (4,1);
\draw (4.5,1.5) -- (4,1);
\draw (1.5,1.5) -- (2,1);
\draw (2.5,1.5) -- (3,1);
\draw (3.5,1.5) -- (4,1);
\draw (3,3.5) -- (5,3.5) node [right] {$1$};
\draw (4.5,2.5) -- (5,2.5) node [right] {$2$};
\draw (4.5,1.5) -- (5,1.5) node [right] {$3$};
\draw (4,1) -- (4,0.5) node [below] {$4$};
\draw (3,1) -- (3,0.5) node [below] {$5$};
\draw (2,1) -- (2,0.5) node [below] {$6$};
\draw (1,1) -- (1,0.5) node [below] {$7$};
\draw [fill=white] (1,1) circle [radius=0.2];
\draw [fill=white] (2,1) circle [radius=0.2];
\draw [fill=white] (3,1) circle [radius=0.2];
\draw [fill=white] (4,1) circle [radius=0.2];
\draw [fill] (1.5,1.5) circle [radius=0.2];
\draw [fill] (2.5,1.5) circle [radius=0.2];
\draw [fill] (3.5,1.5) circle [radius=0.2];
\draw [fill] (4.5,1.5) circle [radius=0.2];
\draw [fill=white] (1,2) circle [radius=0.2];
\draw [fill=white] (2,2) circle [radius=0.2];
\draw [fill=white] (3,2) circle [radius=0.2];
\draw [fill=white] (4,2) circle [radius=0.2];
\draw [fill] (1.5,2.5) circle [radius=0.2];
\draw [fill] (2.5,2.5) circle [radius=0.2];
\draw [fill] (3.5,2.5) circle [radius=0.2];
\draw [fill] (4.5,2.5) circle [radius=0.2];
\draw [fill=white] (3,3.5) circle [radius=0.2];
\end{tikzpicture}
\quad\!
\overset{\text{Type I, $\mu_{2,3}$}}{\longrightarrow} \quad\!
\begin{tikzpicture}[scale=0.5, baseline=5ex]
\draw (2.5,2.5) -- (3,2);
\draw (3,3.5) -- (3.5,2.5);
\draw (3,3.5) -- (4.5,2.5);
\draw (1.5,2.5) -- (1,2);
\draw (2.5,2.5) -- (2,2);
\draw (3.5,2.5) -- (3,2);
\draw (4.5,2.5) -- (4,2);
\draw (1.5,2.5) -- (2,2);
\draw (1.5,2.5) to [out=180, in=60] (0.5,2) to [out=-120, in=120] (1,1);
\draw (2.5,2.5) to [out=150, in=60] (0.25,2.25) to [out=-120, in=140] (1,1);
\draw (3.5,2.5) to [out=135, in=60] (0,2.5) to [out=-120, in=160] (1,1);
\draw (1,2) -- (1.5,1.5);
\draw (2,2) -- (2.5,1.5);
\draw (3,2) -- (3.5,1.5);
\draw (4,2) -- (4.5,1.5);
\draw (1.5,1.5) -- (3,1);
\draw (3,2) -- (4.5,1.5);
\draw (2.5,1.5) -- (4,1);
\draw (3.5,1.5) -- (5,1);
\draw (1.5,1.5) -- (2,1);
\draw (2.5,1.5) -- (3,1);
\draw (3.5,1.5) -- (4,1);
\draw (4.5,1.5) -- (5,1);
\draw (3,3.5) -- (5,3.5) node [right] {$1$};
\draw (4.5,2.5) -- (5,2.5) node [right] {$2$};
\draw (5,1) -- (5.5,1) node [right] {$3$};
\draw (4,1) -- (4,0.5) node [below] {$4$};
\draw (3,1) -- (3,0.5) node [below] {$5$};
\draw (2,1) -- (2,0.5) node [below] {$6$};
\draw (1,1) -- (1,0.5) node [below] {$7$};
\draw [fill=white] (1,1) circle [radius=0.2];
\draw [fill=white] (2,1) circle [radius=0.2];
\draw [fill=white] (3,1) circle [radius=0.2];
\draw [fill=white] (4,1) circle [radius=0.2];
\draw [fill] (1.5,1.5) circle [radius=0.2];
\draw [fill] (2.5,1.5) circle [radius=0.2];
\draw [fill] (3.5,1.5) circle [radius=0.2];
\draw [fill] (4.5,1.5) circle [radius=0.2];
\draw [fill=white] (1,2) circle [radius=0.2];
\draw [fill=white] (2,2) circle [radius=0.2];
\draw [fill=white] (3,2) circle [radius=0.2];
\draw [fill=white] (4,2) circle [radius=0.2];
\draw [fill] (1.5,2.5) circle [radius=0.2];
\draw [fill] (2.5,2.5) circle [radius=0.2];
\draw [fill] (3.5,2.5) circle [radius=0.2];
\draw [fill] (4.5,2.5) circle [radius=0.2];
\draw [fill=white] (3,3.5) circle [radius=0.2];
\draw [fill=white] (5,1) circle [radius=0.2];
\end{tikzpicture}
\]
\[\overset{\text{Type II}}{\longrightarrow}
\begin{tikzpicture}[scale=0.5, baseline=5ex]
\draw (2.5,2.5) -- (3,2);
\draw (3,3.5) -- (3.5,2.5);
\draw (3,3.5) -- (4.5,2.5);
\draw (1.5,2.5) -- (1,2);
\draw (2.5,2.5) -- (2,2);
\draw (3.5,2.5) -- (3,2);
\draw (4.5,2.5) -- (4,2);
\draw (1.5,2.5) -- (2,2);
\draw (1.5,2.5) to [out=180, in=60] (0.5,2) to [out=-120, in=120] (1,1);
\draw (2.5,2.5) to [out=150, in=60] (0.25,2.25) to [out=-120, in=140] (1,1);
\draw (3.5,2.5) to [out=135, in=60] (0,2.5) to [out=-120, in=160] (1,1);
\draw (1,2) -- (1.5,1.5);
\draw (2,2) -- (2.5,1.5);
\draw (3,2) -- (3.5,1.5);
\draw (4,2) -- (4.5,1.5);
\draw (1.5,1.5) -- (3,1);
\draw (4.5,2.5) -- (3,2);
\draw (2.5,1.5) -- (4,1);
\draw (3.5,1.5) -- (5,1);
\draw (1.5,1.5) -- (2,1);
\draw (2.5,1.5) -- (3,1);
\draw (3.5,1.5) -- (4,1);
\draw (4.5,1.5) -- (5,1);
\draw (3,3.5) -- (5,3.5) node [right] {$1$};
\draw (4.5,1.5) -- (5.5,2) node [right] {$2$};
\draw (5,1) -- (5.5,1) node [right] {$3$};
\draw (4,1) -- (4,0.5) node [below] {$4$};
\draw (3,1) -- (3,0.5) node [below] {$5$};
\draw (2,1) -- (2,0.5) node [below] {$6$};
\draw (1,1) -- (1,0.5) node [below] {$7$};
\draw [fill=white] (1,1) circle [radius=0.2];
\draw [fill=white] (2,1) circle [radius=0.2];
\draw [fill=white] (3,1) circle [radius=0.2];
\draw [fill=white] (4,1) circle [radius=0.2];
\draw [fill] (1.5,1.5) circle [radius=0.2];
\draw [fill] (2.5,1.5) circle [radius=0.2];
\draw [fill] (3.5,1.5) circle [radius=0.2];
\draw [fill] (4.5,1.5) circle [radius=0.2];
\draw [fill=white] (1,2) circle [radius=0.2];
\draw [fill=white] (2,2) circle [radius=0.2];
\draw [fill=white] (3,2) circle [radius=0.2];
\draw [fill=white] (4,2) circle [radius=0.2];
\draw [fill] (1.5,2.5) circle [radius=0.2];
\draw [fill] (2.5,2.5) circle [radius=0.2];
\draw [fill] (3.5,2.5) circle [radius=0.2];
\draw [fill] (4.5,2.5) circle [radius=0.2];
\draw [fill=white] (3,3.5) circle [radius=0.2];
\draw [fill=white] (5,1) circle [radius=0.2];
\end{tikzpicture}
\quad\!
\overset{\text{Type I, $\mu_{1,3}$}}{\longrightarrow} \quad\!
\begin{tikzpicture}[scale=0.5, baseline=5ex]
\draw (2.5,2.5) -- (3,2);
\draw (3.5,2.5) -- (4,2);
\draw (1.5,2.5) -- (1,2);
\draw (2.5,2.5) -- (2,2);
\draw (3.5,2.5) -- (3,2);
\draw (4.5,2.5) -- (4,2);
\draw (1.5,2.5) -- (2,2);
\draw (1.5,2.5) to [out=180, in=60] (0.5,2) to [out=-120, in=120] (1,1);
\draw (2.5,2.5) to [out=150, in=60] (0.25,2.25) to [out=-120, in=140] (1,1);
\draw (3.5,2.5) to [out=135, in=60] (0,2.5) to [out=-120, in=160] (1,1);
\draw (1,2) -- (1.5,1.5);
\draw (2,2) -- (2.5,1.5);
\draw (3,2) -- (3.5,1.5);
\draw (4,2) -- (4.5,1.5);
\draw (1.5,1.5) -- (3,1);
\draw (4.5,2.5) to [out=120, in=60] (-0.25,2.75) to [out=-120, in=180] (1,1);
\draw (2.5,1.5) -- (4,1);
\draw (3.5,1.5) -- (5,1);
\draw (1.5,1.5) -- (2,1);
\draw (2.5,1.5) -- (3,1);
\draw (3.5,1.5) -- (4,1);
\draw (4.5,1.5) -- (5,1);
\draw (4.5,2.5) -- (5.5,3) node [right] {$1$};
\draw (4.5,1.5) -- (5.5,2) node [right] {$2$};
\draw (5,1) -- (5.5,1) node [right] {$3$};
\draw (4,1) -- (4,0.5) node [below] {$4$};
\draw (3,1) -- (3,0.5) node [below] {$5$};
\draw (2,1) -- (2,0.5) node [below] {$6$};
\draw (1,1) -- (1,0.5) node [below] {$7$};
\draw [fill=white] (1,1) circle [radius=0.2];
\draw [fill=white] (2,1) circle [radius=0.2];
\draw [fill=white] (3,1) circle [radius=0.2];
\draw [fill=white] (4,1) circle [radius=0.2];
\draw [fill] (1.5,1.5) circle [radius=0.2];
\draw [fill] (2.5,1.5) circle [radius=0.2];
\draw [fill] (3.5,1.5) circle [radius=0.2];
\draw [fill] (4.5,1.5) circle [radius=0.2];
\draw [fill=white] (1,2) circle [radius=0.2];
\draw [fill=white] (2,2) circle [radius=0.2];
\draw [fill=white] (3,2) circle [radius=0.2];
\draw [fill=white] (4,2) circle [radius=0.2];
\draw [fill] (1.5,2.5) circle [radius=0.2];
\draw [fill] (2.5,2.5) circle [radius=0.2];
\draw [fill] (3.5,2.5) circle [radius=0.2];
\draw [fill] (4.5,2.5) circle [radius=0.2];
\draw [fill=white] (5,1) circle [radius=0.2];
\end{tikzpicture}
\quad\! \overset{=}{\longrightarrow}\quad\!
\begin{tikzpicture}[scale=0.5, baseline=5ex]
\draw (3,3.5) -- (1.5,2.5);
\draw (3,3.5) -- (2.5,2.5);
\draw (3,3.5) -- (3.5,2.5);
\draw (3,3.5) -- (4.5,2.5);
\draw (1.5,2.5) -- (1,2);
\draw (2.5,2.5) -- (2,2);
\draw (3.5,2.5) -- (3,2);
\draw (4.5,2.5) -- (4,2);
\draw (1.5,2.5) -- (2,2);
\draw (2.5,2.5) -- (3,2);
\draw (3.5,2.5) -- (4,2);
\draw (1,2) -- (1.5,1.5);
\draw (2,2) -- (2.5,1.5);
\draw (3,2) -- (3.5,1.5);
\draw (4,2) -- (4.5,1.5);
\draw (1.5,1.5) -- (1,1);
\draw (2.5,1.5) -- (2,1);
\draw (3.5,1.5) -- (3,1);
\draw (4.5,1.5) -- (4,1);
\draw (1.5,1.5) -- (2,1);
\draw (2.5,1.5) -- (3,1);
\draw (3.5,1.5) -- (4,1);
\draw (3,3.5) -- (5,3.5) node [right] {$7$};
\draw (4.5,2.5) -- (5,2.5) node [right] {$1$};
\draw (4.5,1.5) -- (5,1.5) node [right] {$2$};
\draw (4,1) -- (4,0.5) node [below] {$3$};
\draw (3,1) -- (3,0.5) node [below] {$4$};
\draw (2,1) -- (2,0.5) node [below] {$5$};
\draw (1,1) -- (1,0.5) node [below] {$6$};
\draw [fill=white] (1,1) circle [radius=0.2];
\draw [fill=white] (2,1) circle [radius=0.2];
\draw [fill=white] (3,1) circle [radius=0.2];
\draw [fill=white] (4,1) circle [radius=0.2];
\draw [fill] (1.5,1.5) circle [radius=0.2];
\draw [fill] (2.5,1.5) circle [radius=0.2];
\draw [fill] (3.5,1.5) circle [radius=0.2];
\draw [fill] (4.5,1.5) circle [radius=0.2];
\draw [fill=white] (1,2) circle [radius=0.2];
\draw [fill=white] (2,2) circle [radius=0.2];
\draw [fill=white] (3,2) circle [radius=0.2];
\draw [fill=white] (4,2) circle [radius=0.2];
\draw [fill] (1.5,2.5) circle [radius=0.2];
\draw [fill] (2.5,2.5) circle [radius=0.2];
\draw [fill] (3.5,2.5) circle [radius=0.2];
\draw [fill] (4.5,2.5) circle [radius=0.2];
\draw [fill=white] (3,3.5) circle [radius=0.2];
\end{tikzpicture}
\]

\section{Connection to Rietsch--Williams's cluster duality}

In \cite[Theorem~1.1]{RW}, Rietsch and Williams identified the cluster dual of the pair $\big(\dGr_a^\times(n), D\big)$ as $\big(\Gr_a^\times(n)\times \mathbb{G}_m,\mathcal{W}_q\big)$, where $D=\bigcup_iD_i$ is the same boundary divisor as the one considered in this paper. On the other side, besides ratios of Pl\"{u}cker coordinates, their potential function $\mathcal{W}_q$ also carries an auxiliary variable $q$ that is the coordinate of $\mathbb{G}_m$. Below we construct a map to identify the Rietsch--Williams cluster dual with the one considered in this paper, which explains the origin of the auxiliary variable~$q$ from our perspective.

Let $ [\phi_1,l_1,\dots, \phi_n,l_n ]\in \dconf_n^\times(a)$. Pick a non-zero vector $v$ in $l_n$. Define
\begin{align*}
 \Phi\colon \ \dconf_n^\times(a)&\longrightarrow \Gr_a^\times(n)\times \mathbb{G}_m,\\
 [\phi_1,l_1,\dots, \phi_n,l_n ]& \longmapsto \big( \big[v, \phi(v),\dots, \phi^{n-1}(v) \big],P [\phi_1,l_1,\dots, \phi_n,l_n ]\big).
\end{align*}
Here $P$ is the twisted monodromy, and $\big[ \phi^{n-1}(v), \dots, \phi(v),v\big]$ denotes the matrix representative of the Grassmannian point (we drop the subscripts of $\phi$ to simplify the notation).

This map $\Phi$ is a biregular isomorphism. To define the inverse map, first pick an $a\times n$ matrix whose row span is an element in $\Gr_a^\times(n)$; the spans of its column vectors $v_i$ define the lines $l_{2-i}$ (indices taken modulo~$n$), and $v_i\mapsto v_{i+1}$ define isomrophisms $\phi_i\colon l_{2-i}\rightarrow l_{1-i}$ for $1< i\leq n$; the remaining isomorphism~$\phi_1$ is then uniquely defined so that the twisted monodromy $P$ would take the value that is equal to the auxiliary variable $q\in \mathbb{G}_m$. It is not hard to see that the resulting image of the inverse map does not depend on the choice of the matrix representative.

Recall that the potential function constructed by Rietsch and Williams takes the form
\[
\mathcal{W}_q=q\frac{\Delta_{\{b+1,\dots, n-1,1\}}}{\Delta_{\{b+1,\dots, n\}}}+\sum_{i=1}^{n-1} \frac{\Delta_{\{i-a+1,\dots, i-1,i+1\}}}{\Delta_{\{i-a+1,\dots, i\}}}.
\]
Fix a volume form $\omega$ on $\mathds{k}^a$. For $1\leq i\leq n-1$, we see that
\begin{gather*}
 \omega\left(\left(\phi^{i}(v)-\Phi^*\left(\frac{\Delta_{\{i-a+1,\dots, i-1,i+1\}}}{\Delta_{\{i-a+1,\dots, i\}}}\right)\phi^{i-1}(v)\right)\wedge \phi^{i-2}(v)\wedge \dots \wedge \phi^{i-a}(v)\right)\\
 = \omega\left(\left(\phi^{i}(v)-\frac{\omega\left(\phi^{i-a}(v)\wedge \dots \wedge \phi^{i-2}(v)\wedge \phi^{i}(v)\right)}{\omega\left(\phi^{i-a}(v)\wedge \dots \wedge \phi^{i-1}(v)\right)}\phi^{i-1}(v)\right)\wedge \phi^{i-2}(v)\wedge \dots \wedge \phi^{i-a}(v)\right)\\
 = 0;
\end{gather*}
Comparing it with \eqref{theta.f} we conclude that
\[
\Phi^*\left(\frac{\Delta_{\{i-a+1,\dots, i-1,i+1\}}}{\Delta_{\{i-a+1,\dots, i\}}}\right)=\vartheta_{i-a}.
\]
For the remaining term we also see that
\begin{gather*}
 \omega\left(\left(\phi^n(v)-\Phi^*\left(q\frac{\Delta_{\{b+1,\dots, n-1,1\}}}{\Delta_{\{b+1,\dots, n\}}}\right)\phi^{n-1}(v)\right)\wedge \phi^{n-2}(v)\wedge \dots \wedge \phi^{b}(v)\right)\\
 = \omega \left(\left(\phi^n(v)-P\frac{\omega\left(v\wedge \phi^{b}(v)\wedge \dots \wedge \phi^{n-2}(v)\right)}{\omega\left(\phi^{b}(v)\wedge \dots \wedge \phi^{n-1}(v)\right)}\phi^{n-1}(v)\right)\wedge \phi^{n-2}(v)\wedge \dots \wedge \phi^b(v)\right)\\
 = 0,
\end{gather*}
which implies that
\[
\Phi^*\left(q\frac{\Delta_{\{b+1,\dots, n-1,1\}}}{\Delta_{\{b+1,\dots, n\}}}\right)=\vartheta_b.
\]
Therefore we can conclude that our potential function $\mathcal{W}$ is precisely
\[
\mathcal{W}=\Phi^*\left(\mathcal{W}_q\right).
\]
Hence, the map $\Phi\colon \left(\dconf_n^\times(a),\mathcal{W}\right)\rightarrow \left(\Gr_a^\times(n)\times \mathbb{G}_m, \mathcal{W}_q\right)$ is an isomorphism between our version of cluster dual space and Rietsch--Williams' cluster dual space. In particular, the auxiliary parameter $q$ in Rietsch--Williams potential function $\mathcal{W}_q$ now has a geometric interpretation as the twisted monodromy $P$ in our construction of the cluster dual.

\subsection*{Acknowledgement}

We are grateful to Alexander Goncharov for the inspiration on the construction of the cluster dual space, and to Jiuzu Hong for many helpful discussions on the representation theoretical aspects of the cyclic sieving problem. We would also like to thank Michael Gekhtman, Li Li, Tim Magee, Gregg Musiker, Brendon Rhoades, Bruce Sagan, Lauren Williams, Eric Zaslow and Peng Zhou for useful conversations in the process of drafting this paper. Finally, we thank the referees for very careful reading of this paper and for many useful suggestions.

\pdfbookmark[1]{References}{ref}
\LastPageEnding

\end{document}